\documentclass{article}

\usepackage[T1]{fontenc}
\usepackage[utf8]{inputenc}
\usepackage{microtype}
\usepackage{libertine}
\usepackage{amsmath}
\usepackage{amssymb}
\usepackage{amsthm}
\usepackage{mathtools}
\usepackage{dsfont}
\usepackage{bm}
\usepackage{physics}
\usepackage{caption}
\usepackage{subcaption}
\usepackage[dvipsnames]{xcolor}
\usepackage[textsize=footnotesize,disable]{todonotes}
\presetkeys%
		{todonotes}%
		{color=Dandelion}{}%
\usepackage{hyperref}
\hypersetup{
    colorlinks,
    linkcolor={red!50!black},
    citecolor={blue!50!black},
    urlcolor={blue!80!black}
}

\theoremstyle{plain}
\newtheorem{definition}{Definition}[section]
\newtheorem{lemma}[definition]{Lemma}
\newtheorem{prop}[definition]{Proposition}
\newtheorem{theorem}[definition]{Theorem}
\newtheorem{corol}[definition]{Corollary}

\newtheorem{constr}[definition]{Construction}
\newtheorem{conj}[definition]{Conjecture}

\theoremstyle{remark}

\newtheorem{ex}[definition]{Example}
\newtheorem{remark}[definition]{Remark}

\usepackage{tikz}
\usetikzlibrary{calc}
\usetikzlibrary{ decorations.markings}

\usepackage{geometry}
\usepackage{authblk}

\usepackage[doi=false,isbn=false,url=false,eprint=true,date=year,backend=biber, citestyle=alphabetic, style=alphabetic]{biblatex}
\newcommand{\ind}{\mathds{1}}

\newcommand{\N}{\mathbf{N}}
\newcommand{\Z}{\mathbf{Z}}
\newcommand{\E}{\mathds{E}}

\newcommand{\Real}{\mathbb{R}}

\newcommand{\ee}{\mathrm{e}}

\newcommand{\Sym}{\mathfrak{S}}

\newcommand{\Part}{\mathcal{P}}
\newcommand{\Inv}{\mathcal{I}}

\newcommand{\motz}{\mathrm{Motz}}

\newcommand{\map}{\mathfrak{m}}
\newcommand{\Maps}{\mathfrak{M}}

\newcommand{\hypermap}{\mathfrak{h}}

\newcommand{\Comb}{\mathfrak{C}}
\newcommand{\Cat}{{\rm Cat}}
\newcommand{\eq}{{\rm eq}}

\newcommand{\Suitably}{\mathcal{S}}
\newcommand{\welllab}{\mathcal{H}}

\newcommand{\RP}{\Real{\rm P}}

\newcommand{\vertex}{{\rm vert}}

\newcommand{\inv}{{\rm inv}}

\newcommand{\Flag}{{\rm Fl}}
\newcommand{\flag}{\mathfrak{f}}
\newcommand{\class}{\mathcal{C}}
\newcommand{\Orbit}{\mathcal{O}}
\newcommand{\He}{{\rm He}}

\usepackage{listofitems}
\newcommand\cycle[2][\,]{%
  \readlist\thecycle{#2}%
  (\foreachitem\i\in\thecycle{\ifnum\icnt=1\else#1\fi\i})%
}

\DeclareMathOperator{\Cyc}{Cycles}

\DeclareMathOperator{\Id}{Id}

\DeclareMathOperator{\Supp}{Supp}

\author[1]{Thomas Buc--d'Alch\'e \thanks{bucdalche(at)unistra.fr}}
\affil[1]{IRMA UMR 7501, Université de Strasbourg, France}
\title{Enumeration of maps with the Dumitriu-Edelman model}
\date{}

\begin{document}

\maketitle

\begin{abstract}
  We give an expansion in $1/N$ and $\beta$ of the cumulants of power sums
  of the particles of the $\beta$-ensemble. This new expansion is obtained
  using the tridiagonal model of Dumitriu and Edelman. The
  coefficients of the expansion are expressed in terms of suitably
  labelled maps introduced by Bouttier, Fusy, and Guitter. Our
  expansion is of a different nature than the one obtained by LaCroix
  in is study of the $b$-conjecture of Goulden and Jackson, and
  involves only orientable maps. We are able to relate bijectively the
  first two orders of our expansion to the one of LaCroix using a
  novel many-to-one mapping that relates suitably labelled maps with
  two minima and maps on the projective plane $\RP^{2}$.
\end{abstract}

\section{Introduction}
\label{sec:introduction}

Since the seminal work of Brézin, Itzykson, Parisi and Zuber \cite{brezin_planar_1978},
random matrix techniques have been a powerful tool for enumerating
maps. Informally, a map is a graph embedded in a compact surface. An
important problem in combinatorics is to count the number of maps with
some constraints: on the number of edges, the genus of the surface the
graph is embedded in, the number and degrees of the faces or vertices.
In the sequel the data of these degrees will be called the vertex or
face \emph{profile} respectively. It is a partition of twice the number of
edges. This problem was studied first by Tutte, who gave several
important results concerning the number of planar maps \cite{tutte_census_1963, tutte_enumeration_1968}.
Random matrix theory allows to tackle this problem using tools from
analysis and probability. This technique allowed to address many
different questions, for instance concerning the moduli space of
curves \cite{harer_euler_1986} or 2d quantum gravity \cite{difrancesco_2d_1995}.

Arguably the most famous random matrix models are the Gaussian
orthogonal, unitary, and symplectic ensembles (GOE, GUE, and GSE
respectively). These are respectively real symmetric, complex
Hermitian, and quaternionic self-adjoint matrices whose coefficients
are (up to symmetry) independent real, complex, or quaternionic
Gaussian variables. Moments of the GUE are related to the enumeration
of orientable maps, while moments of the GOE or the GSE are related to
possibly non-orientable maps, see \cite{cicuta_topological_1982} or
\cite{mulase_duality_2003}.

Let $X^{N}$ be a $N \times N$ matrix sampled according to one of these
three matrix ensembles. Denote by $\bm{\lambda} = (\lambda_{1}, \ldots, \lambda_{N})$ its
eigenvalues. The eigenvalues $\bm{\lambda}$ are distributed according to
\begin{equation}\label{eq:beta-ensemble}
  \dd \mu^{N}_{\beta}(\bm{\lambda}) = \frac{1}{Z^{N}_{\beta}} \left| \Delta(\bm{\lambda}) \right|^{\beta}\ee^{-\frac{\beta}{4}\sum_{i=1}^{N}\lambda_{i}^{2}}\dd \bm{\lambda},
\end{equation}
where $Z^{N}_{\beta}$ is a normalization constant usually called the
partition function,
$\Delta(\bm{\lambda}) = \prod_{1 \leq i < j \leq N}(\lambda_{i} - \lambda_{j})$ is the
Vandermonde determinant, $\dd \bm{\lambda} = \dd \lambda_{1}\cdots \dd \lambda_{N}$ is
the Lebesgue measure on $\Real^{N}$, and $\beta > 0$ is one of $\beta=1$,
$\beta =2$, or $\beta= 4$ if $X^{N}$ is sampled according to the GOE, the GUE,
or the GSE respectively.

The probability measure $\mu^{N}_{\beta}$ is called the $\beta$-ensemble. It is
a well-defined probability measure for all $\beta > 0$. We may ask whether
the moments of the $\beta$-ensemble are related to the enumeration of maps
when $\beta \notin \{1, 2, 4\}$.

A relation between enumeration of maps, and the $\beta$-ensemble for
$\beta \notin \{1, 2, 4\}$ appeared first in the context of the (marginal)
$b$-conjecture, which we now describe. In a series of articles Goulden
and Jackson \cite{goulden_connection_1996} studied what they called the \emph{map series}:
\begin{equation*}
  M_{\alpha}(\bm{y}, \bm{x}, z) = 2 \alpha z \frac{\partial}{\partial z} \ln \sum_{\theta}z^{|\theta|/2} \frac{J_{\theta}(\bm{y}, \alpha)J_{\theta}(\bm{x}, \alpha)}{ \left< J_{\theta}, J_{\theta} \right>_{\alpha}} \left[ p_{2}(\bm{z})^{|\theta|/2} \right]J_{\theta}(\bm{z}, \alpha),
\end{equation*}
where the sum is on all partitions of integers $\theta$ (including the
empty partition), $|\theta|$ is the size of the partition $\theta$,
$J$ is a Jack polynomial, $\left< \cdot \right>_{\alpha}$ is an inner product
between symmetric polynomials, and
$\left[ p_{2}(\bm{z})^{|\theta|/2} \right]J_{\theta}(\bm{z}, \alpha)$ denotes the
coefficient of
$p_{2}^{|\theta|/2}(\bm{z}) = \left( \sum_{i} z_{i}^{2} \right)^{|\theta|/2}$
in the expansion of $J_{\theta}(\bm{z}, \alpha)$ in terms of power sum symmetric
polynomials. The Jack polynomials, defined in \cite{jack_class_1970}, are symmetric
polynomials which constitute a continuous deformation between the
Schur functions at $\alpha = 1$, and the zonal polynomials at $\alpha = 2$.

It has been shown by Goulden, Jackson, and Harer
\cite{goulden_geometric_2001} that the map series is related to the
measure $\mu^{N}_{\beta}$ through the formal relation
\begin{equation*}
  M_{2/\beta}(\bm{p}(\bm{y}), N, z) = \beta z \frac{\partial}{\partial z} \ln \mu_{\beta}^{N} \left[ \ee^{\frac{\beta}{2}\sum_{k \geq 1}\frac{z^{k/2}}{k}p_{k}(y)p_{k}(\lambda)} \right],
\end{equation*}
where $\bm{p}(\bm{y}) = (p_{k}(\bm{y}))_{k \geq 1}$. Note that in
general, this identity holds only as an equality between formal
series: differentiating any number of times and taking the parameters
$\bm{p}(\bm{y})$ and $z$ to zero on both sides yield the same result.

The maps series can be expressed in the basis of symmetric polynomials as
\begin{equation*}
  M_{2/\beta}(\bm{p}(\bm{x}), \bm{p}(\bm{y}), z) = \sum_{n \geq 0}\sum_{\mu, \nu}c_{\mu, \nu, n}(2/\beta-1)p_{\mu}(\bm{x})p_{\nu}(\bm{y})z^{n},
\end{equation*}
where the coefficients $c_{\mu, \nu, n}(b)$ are related to number of maps
with $n$ edges and profile of vertices and faces specified by $\mu$ and
$\nu$. The maps enumerated are orientable when $b = 0$ and possibly
non-orientable when $b = 1$. A more general series could be
considered, the hypermap series where maps are replaced by bipartite
maps and the profiles of the two types of vertices are specified
independently. Goulden and Jackson conjectured \cite{goulden_connection_1996} that the
coefficients of the hypermap series are polynomials in
$b = \frac{2}{\beta} - 1$ encoding a measure of how non-orientable a map
is. In our case, this becomes:
\begin{conj}[Marginal $b$-conjecture]\label{conj:marginal-b-conj}
  For all $n, f \geq 1$, $\nu \vdash 2n$,
  \begin{equation*}
    \sum_{\substack{\mu \vdash 2n\\l(\mu) = f}}c_{\mu, \nu, n}(b) = \sum_{\map} b^{\vartheta(\map)},
  \end{equation*}
  where the sum is on maps (possibly non-orientable) with $n$ edges and
  profile of vertices and faces prescribed by $\mu$ and $\nu$. The exponent
  $\vartheta(\map)$ is a measure of how non-orientable $\map$ is, with
  $\vartheta(\map) = 0$ if and only if $\map$ is orientable.
\end{conj}
This marginal $b$-conjecture has been solved by LaCroix \cite[Theorem
4.16]{lacroix_combinatorics_2009}. He described the exponent $\vartheta$ using an inductive
procedure, similar to Tutte's decomposition of maps. A generalization
of the conjecture of Goulden and Jackson has been studied by Chapuy
and Dołęga \cite{chapuy_nonorientable_2022}. They studied a
$b$-deformation of a tau function of the 2-Toda integrable hierarchy,
and showed that it is a generating function of generalized branched
covering of the sphere, with $b$-weight depending on a measure of
non-orientability. 

We study the cumulants of the $\beta$-ensemble, related to the marginal
$b$-conjecture, and propose a different answer than the one of
LaCroix. It is based on the tridiagonal matrix model for the
$\beta$-ensemble, introduced by Dumitriu and Edelman \cite{dumitriu_matrix_2002}. This
tridiagonal ensemble was used by Abdesselam, Anderson, and Miller \cite{abdesselam_tridiagonalized_2014}
to recover that the number of planar maps corresponds to the leading
order of the cumulants of the GUE ($\beta = 2$). A remarkable fact was
that the natural combinatorial objects they obtained were \emph{mobiles}, a
family of labelled trees shown by Bouttier, di Francesco, and Guitter
\cite{bouttier_planar_2004} to be in bijection with pointed, rooted, planar
maps. The proof of Abdesselam and al. relied on the
Brydges-Kennedy-Abdesselam-Rivasseau formula, a complicated identity
coming from cluster expansion theory. We simplify and generalize their
work. We show that the cumulants of the symmetric power sum
polynomials in the eigenvalues of the $\beta$-ensemble admit a large $N$
expansion whose coefficients are expressed using \emph{suitably labelled
  maps}, a family of maps with labelled vertices
introduced by Bouttier, Fusy and Guitter \cite{bouttier_twopoint_2014} which are in
bijection with a family of maps generalizing the mobiles. For the two
leading orders, we are able to reinterpret the coefficients as being
sums of maps on the sphere or on the projective plane $\RP^{2}$
respectively. This is done using a novel many-to-one mapping that
relate some suitably labelled maps and maps on $\RP^{2}$.

Another approach based on the Virasoro (or Dyson-Schwinger) equations
is proposed by Cassia et al. \cite{cassia_betaensembles_2024}. They express their result not in
terms of maps but in terms of \emph{generalized Catalan numbers}.

We prove the following theorem in Section \ref{sec:expr-terms-dist}. We denote by $\N$ the
set of non-negative integers $\left\{ 0, 1, 2, \ldots \right\}$ and
given a partition $\bm{n}$ of an integer $n$ by $\class_{\bm{n}}$ the
conjugacy class in the set of permutation of
$\left\{ 1, \ldots, n \right\}$ given by $\bm{n}$.
\begin{theorem}\label{thm:exp-cumulants}
  Let $\bm{n} = (n_{1}, \ldots, n_{l}) \in \N^{l}$ be a partition of $n \geq 2$
  with $l$ parts, i.e.\ $n_{1} + \cdots + n_{l} = n$, and
  $\theta \in \class_{\bm{n}}$. We have the following expansion for the
  cumulants $\kappa_{l}$ of the $\beta$-ensemble:
  \begin{equation}\label{eq:exp-main-thm-intro}
    \kappa_{l}(\bm{n}) = \sum_{p+q+r=n/2}\left( \frac{2}{\beta}\right)^{p} (-1)^{q}P_{r}(N)
    \left\langle e_{q}\right \rangle_{\theta, p}
  \end{equation}
  where $P_{u}(k) = \sum_{i = 1}^{k}i^{u}$ is the $u$-th Newton sum
  and $\left< e_{q}\right>_{\theta, p}$ denotes a sum over suitably
  labelled maps of elementary symmetric polynomials evaluated at the
  labels of the map. The suitably labelled maps have face profile
  $\theta$ and $n/2 - p$ vertices that are not local minima. Suitably
  labelled maps are defined in Definition \ref{def:suitably-labelled}
  and this term will be described in Section
  \ref{sec:expr-terms-dist}.
\end{theorem}
A remarkable fact is that the expression \eqref{eq:exp-main-thm-intro}
involves quantities related to distances in maps whose face profile is
given by $\theta$. An interesting particular case is that when
$p = l-1 = 0$, $\langle e_{q} \rangle_{\theta, l-1}$ is the following sum
over pointed planar maps:
\begin{equation*}
  \langle e_{q} \rangle_{\theta, l-1} = \sum_{\map}e_{q}(d_{v}; v \text{ vertex of }\map),
\end{equation*}
where the sum is on half-edge labelled pointed planar maps, and $d_{v}$
is the graph distance from the pointed vertex to $v$. For instance,
$\frac{\langle e_{1} \rangle_{\theta, l-1}}{\langle e_{0} \rangle_{\theta, l-1}}$ is the average sum of
distances from a pointed vertex in a random planar map with face
profile $\theta$ chosen uniformly. Other values of $q$ give us different
statistics on those distances. Hence, provided we could obtain an
expansion in $1/N$ and $\beta$ of the cumulant $\kappa_{l}(\bm{n})$
analytically, we would be able to compute statistics of the distances
of a uniformly chosen planar (or higher genus) map. A simple case
related to the asymptotics of the power sums of roots of Hermite
polynomials is discussed in Appendix \ref{sec:limit-beta-inf}. In general, this question is
left for future investigation. A possible avenue for studying such
cumulant is provided by the work of Popescu \cite{popescu_general_2009}, and Babet and Popescu
\cite{babet_tridiagonal_2025} on the leading order of such tridiagonal matrices.

There is an apparent mismatch between the expansion obtained by
LaCroix, in terms of orientable and non-orientable maps, and the
expansion of Theorem \ref{thm:exp-cumulants}, in which the main combinatorial objects are
orientable, vertex-labelled maps. That the leading order of both
expansion coincide is a direct consequence of the bijection of
Bouttier, Fusy, and Guitter discussed in Section \ref{sec:maps-mobiles}. We are able to
relate the sub-leading order of both expansions through a novel
many-to-one mapping described in Section \ref{sec:many-betw-suit}. Given a permutation $\theta$
with $c(\theta)$ cycles, the construction described in Section \ref{sec:many-betw-suit} gives a
$2^{c(\theta) - 1}$-to-$1$ mapping between the set of pointed labelled maps
on the projective plane with face determined by $\theta\bar{\theta}$, and the
set of suitably labelled maps with two local minima and face profile
$\theta$. The precise result is stated in Theorem \ref{thm:bijection-1/2}. Thanks to Theorem
\ref{thm:bijection-1/2}, the two leading orders of the expansion of Theorem \ref{thm:exp-cumulants} can be
interpreted in the following way.
\begin{corol}\label{corol:leading-orders}
  Let $\Maps_{0}(\theta)$ be the number of edge-labelled planar maps
  with face profile $\theta$, and $\Maps_{1/2}(\theta)$ be the number
  of edge-labelled maps on $\RP^{2}$ with face profile $\theta$. We
  have
  \begin{equation*}
    \kappa_{l}(\bm{n}) = N^{n/2-l+2} \left( \frac{2}{\beta} \right)^{l-1} \left( \#\Maps_{0}(\theta) + \frac{1}{2^{1 - l}N}(\frac{2}{\beta} - 1)\# \Maps_{1/2}(\theta) + \order{\frac{1}{N^{2}}}\right).
  \end{equation*}
\end{corol}

In Section \ref{sec:moments-beta-ensembl}, we give first expressions for the cumulants of power
sums of the ``eigenvalues'' of the $\beta$-ensemble. We then describe in
Section \ref{sec:maps-mobiles} the main combinatorial objects involved, labelled maps. We
recall known facts and bijections, and give a combinatorial way to
describe them in terms of Motzkin paths and permutations. We use these
objects to re-express the cumulants of the beta ensemble in terms of
sum of combinatorial objects in Section \ref{sec:combinatorial-cumulants}. Finally, in Section \ref{sec:many-betw-suit},
we propose a novel many-to-one mapping that bridges the gap between
our expansion and expansion in terms of non-orientable maps on the
projective plane. This result, Theorem \ref{thm:bijection-1/2} is one
of the main result of our article.

\paragraph{Acknowledgment}

The author is grateful to Grégory Miermont for numerous discussion
regarding the result presented here and more. The author was supported
by the ERC Project LDRAM 884584 during the main part of this project,
and by the ERC Project InSpecGMos 101096550 in the late stages of this
project.

\section{The moments of the $\beta$-ensemble}
\label{sec:moments-beta-ensembl}
We now compute a formula for the moments of the $\beta$-ensemble, which we
now define.
\begin{definition}\label{def:moment}
  Let $l \geq 1$ and $k_{1}, \ldots, k_{l} \geq 0$ be integers. The moment
  of order $\bm{k} = (k_{1}, \ldots, k_{l})$ is
  \begin{equation*}
    m_{l}(\bm{k}) \coloneq \mu_{\beta}^{N} \Bigl( p_{k_{1}}(\bm{\lambda}) \cdots p_{k_{l}}(\bm{\lambda})\Bigr),
  \end{equation*}
  where $p_{k}$ is the power sum symmetric polynomial
  \begin{equation*}
    p_{k}(\bm{\lambda}) = \sum_{i = 1}^{N}\lambda_{i}^{k}.
  \end{equation*}
\end{definition}

\subsection{The tridiagonal model}
\label{sec:tridiagonal-model}

For some time, it was an open question whether there was a matrix
model for $\mu^{N}_{\beta}$, that is, whether there existed a simple random
matrix whose eigenvalues are distributed according to $\mu^{N}_{\beta}$. The
celebrated paper of Dumitriu and Edelman \cite{dumitriu_matrix_2002} gave a positive answer to
this question by exhibiting a symmetric tridiagonal real random matrix
with independent (up to symmetry) entries. Recall that the chi
distribution with parameter $\alpha > 0$, $\chi_{\alpha}$, is the measure on
$\Real^{+}$ whose density with respect to the Lebesgue measure is
\begin{equation}\label{eq:density-chi}
  \rho_{\chi_{\alpha}}(x) = \frac{x^{\alpha-1}\ee^{-x^{2}/2}}{2^{(\alpha/2) - 1}\Gamma(\alpha/2)}.
\end{equation}
\begin{theorem}[\cite{dumitriu_matrix_2002}]
  Let $\beta > 0$ be real and let $(a_{i}, b_{j})_{1 \leq i \leq N, 1 \leq j < N}$
  be a family of real independent random variables such that for all
  $i \in [N]$, $a_{i}$ is a standard Gaussian, and for all
  $i \in [N-1]$, $\sqrt{2}b_{i}$ is distributed according to the chi
  distribution with parameter $(N-i)\beta$. The eigenvalues of the random
  tridiagonal matrix
  \begin{equation}\label{eq:tridiag-matrix}
    \begin{split}
      T^{N}_{\beta} = \sqrt{\frac{2}{\beta}}\begin{pmatrix}
        a_{1} & b_{1} & 0 & 0 & 0 & \dots \\
        b_{1} & a_{2} & b_{2} & 0 & 0 & \dots  \\
        0 & b_{2} & a_{3} & b_{3} & \ddots & \dots  \\
        \vdots & \ddots & \ddots & \ddots & \ddots & 0 \\
        0 & \cdots & 0 &  b_{N-2} & a_{N-1} & b_{N-1}\\
        0 & \cdots & 0 & 0 & b_{N-1} & a_{N}\\
      \end{pmatrix}\,,
    \end{split}
  \end{equation}
  are distributed according to \(\mu_{\beta}^{N}\), the $\beta$-ensemble
  distribution \eqref{eq:beta-ensemble}.
\end{theorem}

\begin{remark}
  Notice that the factor $\beta/2$ is not present in the result of
  Dumitriu and Edelman. Here, we use the convention of the book
  \cite{anderson_introduction_2010}. The difference is that the
  density of eigenvalues is proportional to
  $\exp(-\sum_{i}\lambda_{i}^{2}/2)$ in \cite{dumitriu_matrix_2002},
  and $\exp(-\beta\sum_{i}\lambda_{i}^{2}/4)$ in
  \cite{anderson_introduction_2010}.
\end{remark}

\subsection{Combinatorial interpretation of the $\chi$ distribution}
\label{sec:comb-interpr-chi}

The moments of the standard Gaussian distribution
$\mu^{\mathcal{N}}_{k}$ are
\begin{equation*}
  \mu^{\mathcal{N}}_{k} =
  \begin{cases}
    (k-1)!!\, &\text{ if $k$ is even,}\\
    0\, &\text{ if $k$ is odd.}
  \end{cases}
\end{equation*}
They can be directly interpreted combinatorially as follows. Denote by
$\Sym(I)$ the group of permutation of the elements of a finite subset
$I$ of $\N$, and by $\Id$ its neutral element. For convenience, we
write for $n \geq 1$, $[n] = \{1, 2, \ldots, n\}$ and
$\Sym_{n}\coloneq \Sym([n])$.
\begin{definition}\label{def:matching}
  The set of matchings of a finite set $I \subset \N$ is
  \begin{equation*}
    \Inv^{*}(I) = \{\alpha \in \Sym(I)\colon \alpha^{2} = \Id, \forall i \in I, \alpha(i) \neq i\},
  \end{equation*}
  i.e.\ the set of involution without fixed point. Let $n \geq 1$, we
  write $\Inv^{*}_{n} = \Inv^{*}([n])$.
\end{definition}
The set of matchings could also be defined as the set of partitions of
$[n]$ whose blocks are of size 2. The following lemma is well-known.
\begin{lemma}
  Let $n \geq 1$, we have
  \begin{equation*}
    \# \Inv^{*}_{2n} = (2n - 1)!! = \mu^{\mathcal{N}}_{2n}.
  \end{equation*}
\end{lemma}

In Lemma \ref{lem:comb-chi} below, we give a combinatorial interpretation of the
moments of a chi variable in terms of permutations. Before stating it,
we introduce more notation related to permutations. Fix a finite set
$I \subset \N$. Denote by $\Part(I)$ the set of partitions of $I$. The group
of permutations $\Sym(I)$ acts naturally on $I$. Given $G$ be a
subgroup of $\Sym(I)$, we denote by $\Orbit(G)$ the set of orbits of
the action of $G$ on $I$. It is a partition of $I$. We denote the
subgroup generated by permutations $\sigma_{1}, \ldots, \sigma_{d}\in\Sym(I)$ by
$\left< \sigma_{1}, \ldots, \sigma_{d} \right>$. For convenience, we shall abuse
notation and write $\Orbit(\sigma_{1}, \ldots, \sigma_{d})$ to mean
$\Orbit(\left< \sigma_{1}, \ldots, \sigma_{d} \right>)$. Let $\sigma \in \Sym(I)$. Each
block $B \in \Orbit( \sigma )$ defines a cyclic permutation
$\sigma\vert_{B} \in \Sym(B)$. We write
\begin{equation*}
  \Cyc(\sigma) = \left\{ \sigma\vert_{B} \colon B \in \Orbit(\sigma) \right\},
\end{equation*}
and set
\begin{equation*}
  \#\sigma = \# \Cyc(\sigma) = \# \Orbit(\sigma).
\end{equation*}
The support of $\sigma$ is
\begin{equation*}
  \Supp \sigma = \{i \in I \colon \sigma(i) \neq i\}.
\end{equation*}
We denote by $|\sigma|$ the length of $\sigma$, i.e.\ the minimal number $l$ such
that $\sigma$ can be written as a product of $l$ transpositions. The length
satisfies $|\sigma| = \# I - \# \sigma$.
\begin{lemma}\label{lem:comb-chi}
  Let $n \geq 1$ an integer, $\alpha > 0$, and $X$ be random variable
  distributed as $\chi_{\alpha}$. We have
  \begin{equation*}
    \E \left[ \left( \frac{X}{\sqrt{2}} \right)^{2n} \right] = \sum_{i = 0}^{n-1} \left( \frac{\alpha}{2} \right)^{n-i}\#\{\sigma \in \Sym_{n}\colon |\sigma| = i\} = \sum_{\sigma \in \Sym_{n}} \left( \frac{\alpha}{2} \right)^{\# \sigma}.
  \end{equation*}
\end{lemma}
\begin{proof}
  We have
  \begin{equation*}
    \E \left[ \left( \frac{X}{\sqrt{2}} \right)^{2n} \right] = \frac{\Gamma(n + \frac{\alpha}{2})}{\Gamma(\frac{\alpha}{2})} = \prod_{i=1}^{n} \left( \frac{\alpha}{2} + i - 1\right),
  \end{equation*}
  where we used that for $x > 0$, $\Gamma(x+1) = x\Gamma(x)$. We expand the
  product to obtain
  \begin{equation*}
    \E \left[ \left( \frac{X}{\sqrt{2}} \right)^{2n} \right] = \sum_{i=0}^{n} \left( \frac{\alpha}{2} \right)^{n-i}\sum_{\substack{J \subset [n]\\\# J = i}}\prod_{j \in J}(j - 1).
  \end{equation*}
  The number of transpositions of the form $\cycle{i, j}$ with $i < j$
  is $j - 1$. A product of strictly increasing transpositions is a product
  \begin{equation*}
     \tau_{r}\tau_{r-1}\cdots \tau_{1}
  \end{equation*}
  of transpositions $\tau_{i} = \cycle{a_{i}, b_{i}}$ such that
  $a_{i} < b_{i}$ for all $i$ and $b_{i} < b_{j}$ for all $i < j$. A
  permutation of length $i$ admits a unique decomposition as a product
  of $i$ strictly increasing transpositions. Thus,
  \begin{equation*}
    \sum_{\substack{J \subset [k]\\ \# J = i}}\prod_{j \in J}(j - 1) = \# \{\sigma \in \Sym_{n}\colon |\sigma| = i\}.
  \end{equation*}

  The second equality is a consequence of the fact that
  $|\sigma| = n - \#\sigma$ .
\end{proof}

\subsection{Moments and Motzkin paths}
\label{sec:moments-motzk-paths}
The computation of powers of a tridiagonal matrix naturally involves
the notion of Motzkin paths.

\begin{definition}
  A Motzkin bridge of size $k \geq 1$ and with profile
  $\theta \in \Sym_{k}$ is a function
  $\gamma \colon [k] \to \N$ such that for all $i \in [k]$,
  \begin{equation*}
    |\gamma(i) - \gamma(\theta(i))| \leq 1\,.
  \end{equation*}

  We define the two following sets of Motzkin bridges
  \begin{equation*}
    \begin{split}
      \motz_{k, 0}(\theta) &= \{\gamma \colon [k] \to \N\colon \min\gamma = 0, |\gamma(i) - \gamma(\theta(i))| \leq 1\}\,,\\
      \motz^{[N]}_{k}(\theta) &= \{\gamma \colon [k] \to [N]\colon |\gamma(i) - \gamma(\theta(i))| \leq 1\}\,.\\
    \end{split}
  \end{equation*}
\end{definition}

We observe that for $k \geq 1$,
\begin{equation*}
  \Tr \left( (T_{\beta}^{N})^{k} \right) = \sum_{i_{1}, \ldots, i_{k}=1}^{N}(T^{N}_{\beta})_{i_{1}i_{2}}\cdots (T^{N}_{\beta})_{i_{k-1}i_{k}}(T^{N}_{\beta})_{i_{k}i_{1}} = \sum_{\gamma\in\motz^{N}_{k}(\cycle{1, 2, \cdots, k})}\prod_{i=1}^{k}(T_{\beta}^{N})_{\gamma(i)\gamma(i+1)},
\end{equation*}
since $(T_{\beta}^{N})_{ij} = 0$ if $|i - j| > 1$. We use the convention
that $\gamma(k+1) = \gamma(1)$. We proceed similarly for a product of such
traces. For $k_{1}, \ldots, k_{l} \geq 1$ and $k = \sum_{i=1}^{l}k_{i}$, define
the permutation with $l$ cycles
\begin{equation}\label{eq:theta-n}
  \theta(\bm{k}) = \cycle{1, \cdots, k_{1}}\cdots \cycle{\sum_{i=1}^{l-1}k_{i}+1, \cdots, \sum_{i=1}^{l}k_{i}}.
\end{equation}
We then have
\begin{equation*}
  \prod_{i=1}^{l}\Tr \left( (T_{\beta}^{N})^{k_{i}} \right) = \sum_{\gamma\in\motz^{N}_{k}(\theta(\bm{k}))}\prod_{i=1}^{k}(T_{\beta}^{N})_{\gamma(i)(\gamma\theta(\bm{k}))(i))}.
\end{equation*}

We are ready to compute the moments $m_{l}(\bm{k})$ (recall Definition
\ref{def:moment}).

\begin{definition}\label{def:compatible-permutation}
  Let $\gamma$ be a Motzkin bridge of size $k$ and profile $\theta$. We
  define for $\epsilon \in \{+1, 0, -1\}$, the set
  \begin{equation*}
    \Delta\gamma_{\epsilon} = \{i \in [k]\colon \gamma(\theta(i)) - \gamma(i) = \epsilon\}\,.
  \end{equation*}
  We write $\Delta\gamma_{+1} = \Delta\gamma_{+}$ and $\Delta\gamma_{-1} = \Delta\gamma_{-}$ for convenience.

  A permutation $\sigma \in \Sym_{k}$ is said to be compatible with $\gamma$ if
  $\gamma\circ \sigma = \gamma$, and its restrictions to
  $\Delta\gamma_{+}, \Delta\gamma_{-},$ and $\Delta\gamma_{0}$ satisfy the following conditions:
  \begin{itemize}
    \item $\sigma_{-} \coloneq \sigma\vert_{\Delta\gamma_{-}}$ is a permutation of $\Delta\gamma_{-}$,
    \item $\sigma_{+} \coloneq \sigma\vert_{\Delta\gamma_{+}}$ is the identity on $\Delta\gamma_{+}$,
    \item $\sigma_{0}\coloneq \sigma\vert_{\Delta\gamma_{0}}$ is a matching (recall Definition
          \ref{def:matching}).
  \end{itemize}
  The set of permutations compatible with $\gamma$ is denoted by
  $\Sym^{\gamma}$.
\end{definition}
Note that no permutation can be compatible with $\gamma$ if $\# \Delta\gamma_{0}$ is
odd: in that case $\sigma_{0}$ cannot be a matching.

\begin{prop}\label{prop:compute-moments}
  Let $l \geq 1$, $\bm{k} \in (\N^{*})^{l}$, and
  $k = \sum_{i=1}^{l}k_{i}$. The moments can be expressed as
  \begin{equation}\label{eq:moments-combin}
    m_{l}(\bm{k}) = (\frac{2}{\beta})^{k/2}\sum_{\gamma \in \motz_{k}^{[N]}(\theta(\bm{k}))}\sum_{\sigma \in \Sym^{\gamma}}\left( \frac{\beta}{2} \right)^{\# \sigma_{-}}\prod_{\pi\in\Cyc(\sigma_{-})}\left( \gamma(\pi) - 1 \right)\,,
  \end{equation}
  where $\gamma(\pi)$ denotes the value of $\gamma$ on the support of the cycle
  $\pi$.
\end{prop}
\begin{proof}
  We introduce the local times at height $n$ and $n + 1/2$:
  \begin{equation*}
    \begin{aligned}
      t_{n} &= \# L_{n} &\text{ with } ~&L_{n} = \{i \in \Delta\gamma_{0}\colon \gamma(i) = n\}\,,\\
      t_{n+1/2} &= \# L_{n+1/2} &\text{ with } ~&L_{n+1/2} = \{i \in \Delta\gamma_{-}\colon \gamma(i) = n+1\}.
    \end{aligned}
  \end{equation*}
  They allow us to write the moments as
  \begin{equation*}
    m_{l}(\bm{k}) = \sum_{\gamma\in\motz^{[N]}_{k}(\theta)} \E \left(  \prod_{i=1}^{k}(T_{\beta}^{N})_{\gamma(i)\gamma(\theta(i))}\right) = \sum_{\gamma\in\motz^{N}_{k}(\theta)}\E \left( \prod_{n=1}^{N}a_{n}^{t_{n}}b_{n}^{2t_{n+1/2}} \right).
  \end{equation*}
  with $\theta = \theta(\bm{k})$. Notice that there is a factor $2$ in
  front of $t_{n+1/2}$ as we have to account for indices in
  $\Delta\gamma_{-}$. By independence and Lemma \ref{lem:comb-chi}, we have
  \begin{equation*}
    \begin{split}
      m_{l}(\bm{k}) &= \sum_{\gamma\in\motz^{[N]}_{k}(\theta)} \prod_{n=1}^{N}\E \left(a_{n}^{t_{n}}\right)\mu_{\beta}^{N} \left( b_{n}^{2t_{n+1/2}} \right)\\
                    &= \sum_{\gamma \in \motz^{[N]}_{n}(\theta)}\prod_{n=1}^{N} \left( \#\Inv^{*}(L_{n}) \right) \left( \sum_{\sigma \in \Sym(L_{n+1/2})}\left( \frac{\beta}{2}n \right)^{\# \sigma} \right).
    \end{split}
  \end{equation*}

  Notice that
  \begin{equation*}
    \prod_{n=1}^{N} \left( \#\Inv^{*}(L_{n}) \right) \left( \sum_{\sigma \in \Sym(L_{n+1/2})}\left( \frac{\beta}{2}n \right)^{\# \sigma} \right)
    = \sum_{\sigma \in \Sym^{\gamma}} \left( \frac{\beta}{2} \right)^{\# \sigma_{-}}\prod_{n=1}^{N}n^{\# \sigma\vert_{L_{n+1/2}}}.
  \end{equation*}
  Indeed, $\#\Inv^{*}(L_{n})$ is the number of matching on $L_{n}$
  (corresponding to the permutation \(\sigma_{0}\) in Definition
  \ref{def:compatible-permutation}), and the condition that $\gamma\sigma = \gamma$
  in Definition \ref{def:compatible-permutation} corresponds to each
  cycle of $\sigma\vert_{\Delta\gamma_{-}}$ having support in one of the $L_{n+1/2}$.

  Finally, we have for each $\sigma \in \Sym^{\gamma}$ that
  \begin{equation*}
    \prod_{n=1}^{N}n^{\# \sigma\vert_{L_{n+1/2}}} = \prod_{\pi \in \Cyc(\sigma_{-})} (\gamma(\pi) - 1).
  \end{equation*}
  We obtain
  \begin{equation*}
     m_{l}(\bm{k}) = \sum_{\gamma \in \motz^{N}_{n}(\theta)}\sum_{\sigma \in \Sym^{\gamma}}\left( \frac{\beta}{2} \right)^{\# \sigma_{-}}\prod_{\pi \in \Cyc(\sigma_{-})} \left(\gamma(\pi) - 1  \right).
  \end{equation*}
\end{proof}

It will prove more convenient to consider cumulants rather than
moments, so as to have connected rather than disconnected objects. Let
us first recall the definition of a cumulant.
\begin{definition}\label{def:cumulant}
  Let $X_{1}, \ldots, X_{n}$ be $n$ real random variables. The joint
  cumulants $(\kappa_{l})_{l \geq 1}$ of these random variables are
  $l$-multilinear symmetric maps defined inductively by
  \begin{equation*}
    \E \left[ X_{i_{1}}\cdots X_{i_{l}} \right] = \sum_{\Pi \in\Part([l])}\prod_{V \in \Pi}\kappa_{|V|}(X_{i_{k}}, k \in V)\,.
  \end{equation*}

  We denote by $\kappa_{l}(\bm{k})$ the joint cumulant of
  $p_{k_{1}}(\bm{\lambda}), p_{k_{2}}(\bm{\lambda}), \ldots, p_{k_{l}}(\bm{\lambda})$ under
  $\mu^{N}_{\beta}$.
\end{definition}
Proposition \ref{prop:compute-moments} then translates into the
following result.
\begin{corol}\label{corol:compute-cumulants}
  Let $l \geq 1$, $\bm{k} \in (\N^{*})^{l}$, and
  $k = \sum_{i=1}^{l}k_{i}$. The cumulants can be expressed as
  \begin{equation}\label{eq:cumulants-combin}
    \kappa_{l}(\bm{k}) = (\frac{2}{\beta})^{k/2}\sum_{\gamma \in \motz_{k}^{[N]}(\theta(\bm{k}))}\sum_{\substack{\sigma \in \Sym^{\gamma}\\ \Orbit(\theta(\bm{k}), \sigma ) = 1}}\prod_{\pi\in\Cyc(\sigma_{-})}\frac{\beta}{2} \left( \gamma(\pi) - 1\right).
  \end{equation}
\end{corol}
\begin{proof}
  We decompose the formula of Proposition \ref{prop:compute-moments} depending on the number of orbits of $\left< \theta(\bm{k}), \sigma \right>$ and get
  \begin{equation*}
    \begin{split}
      m_{l}(\bm{k})
      &= (\frac{2}{\beta})^{k/2}\sum_{\gamma \in \motz_{k}^{[N]}(\theta(\bm{k}))}\sum_{\sigma \in \Sym^{\gamma}}\prod_{\pi\in\Cyc(\sigma_{-})}\frac{\beta}{2} \left( \gamma(\pi) - 1 \right)\\
      &= \sum_{\Pi \in \Part([l])}(\frac{2}{\beta})^{k/2}\sum_{\gamma \in \motz_{k}^{[N]}(\theta(\bm{k}))}\sum_{\substack{\sigma \in \Sym^{\gamma}\\\Orbit( \theta(\bm{k}), \sigma ) = \Pi}}\prod_{\pi\in\Cyc(\sigma_{-})}\frac{\beta}{2} \left( \gamma(\pi) - 1 \right)\\
      &= \sum_{\Pi \in \Part([l])}\prod_{B \in \Pi} \left[(\frac{2}{\beta})^{k_{B}/2}\sum_{\gamma \in \motz_{k_{B}}^{[N]}(\theta(\bm{k}_{B}))}\sum_{\substack{\sigma \in \Sym^{\gamma}\\\Orbit(\theta(\bm{k}_{B}), \sigma) = 1}}\prod_{\pi\in\Cyc(\sigma_{-})}\frac{\beta}{2} \left( \gamma(\pi) - 1 \right)  \right],
    \end{split}
  \end{equation*}
  where we introduced the notation $\bm{k}_{B} = (k_{i})_{i \in B}$ and $k_{B} = \sum_{i \in B}k_{i}$ for
  $B \subset [l]$.

  On the other hand, the moments are related to the cumulant through
  \begin{equation*}
    m_{l}(\bm{k}) = \sum_{\Pi\in \Part([l])}\prod_{B \in \Pi}\kappa_{|B|}(k_{i}, i \in B)\,.
  \end{equation*}
  This implies that $\kappa_{l}(\bm{k})$ coincides with the cumulant of
  $(p_{k_{i}})_{1 \leq i \leq l}$ under $\mu_{\beta}^{N}$.
\end{proof}

\subsection{Large $N$ expansion}
\label{sec:large-n-expansion}

We now consider the large $N$ asymptotics of the moments and cumulants
computed in Proposition \ref{prop:compute-moments} and Corollary
\ref{corol:compute-cumulants}. We prove the following large $N$
expansion.
\begin{prop}\label{prop:N-exp-cumulants}
  Let $l \geq 1$ and $\bm{n} = (n_{1}, \ldots, n_{l}) \in (\N^{*})^{l}$
  with $n = \sum_{i=1}^{l}n_{i}$,
    \begin{equation*}
    \kappa_{l}(\bm{n}) = \sum_{p+q+u=n/2}\left( \frac{2}{\beta}\right)^{p} (-1)^{q}P_{u}(N)\sum_{\gamma \in \motz_{n, 0}(\theta(\bm{n}))}\sum_{\substack{\sigma \in \Sym^{\gamma}\\ \Orbit( \theta(\bm{n}), \sigma) = 1\\|\sigma|  = p}}e_{q}\Bigl( \gamma(\pi); \pi \in \Cyc(\sigma_{-}) \Bigr)
  \end{equation*}
  where
  $e_{q}(x_{1}, \ldots, x_{m}) = \sum_{1 \leq i_{1} < i_{2} < \cdots < i_{q} \leq m}\prod_{j=1}^{q}x_{i_{j}}$
  is the $q$-th elementary symmetric polynomial and
  $P_{u}(k) = \sum_{i = 1}^{k}i^{u}$ is the $u$-th Newton sum.

  Equivalently, we have the an expansion in powers of $1/N$,
  \begin{equation*}
    \kappa_{l}(\bm{n}) = \sum_{p+q+r+s = n/2}\sum_{\substack{\gamma\in\motz_{n, 0}(\theta(\bm{n}))\\\sigma \in \Sym_{\gamma}, |\sigma| = p\\ \Orbit(\theta(\bm{n}), \sigma ) = 1}} \left( \frac{2}{\beta} \right)^{p}\frac{(-1)^{q}B_{r}}{s+1}\binom{r+s}{r} N^{s+1}e_{q}\Bigl(\gamma(\pi); \pi \in \Cyc(\sigma_{-})\Bigr),
  \end{equation*}
  where $(B_{r})_{r \geq 0} = (1, -1/2, 1/6, \ldots)$ is the sequence
  of Bernoulli numbers, defined inductively by
  \begin{equation}\label{eq:Bernoulli}
    \sum_{k=0}^{n}\binom{n+1}{k}(-1)^{k}B_{k} = \delta_{n, 0} \text{ for all } n \geq 0.
  \end{equation}
\end{prop}
Notice that the permutation $\theta(\bm{n})$ has $l$ cycles and thus if
$|\sigma| = p < l-1$ the group $\left< \theta(\bm{n}), \sigma \right>$ cannot act
transitively on $[n]$, i.e.\
$\Orbit(\theta(\bm{n}), \sigma)$ cannot be $1$. Thus,
the leading order of $\kappa_{l}(\bm{n})$, obtained when $s$ is maximal in
the sum above under the constraint $p \geq l - 1$. The leading order is
given by taking $p = l-1, q = 0, r=0, s = n/2 - l +1$, which gives
\begin{equation*}
  \kappa_{l}(\bm{n}) = \left( \frac{2}{\beta} \right)^{l-1}N^{n/2-l+2}\frac{\# \left\{ (\gamma, \sigma) \colon \begin{aligned}
    &\gamma \in \motz_{n, 0}(\theta(\bm{n})),\sigma \in \Sym_{\gamma},\\ &\Orbit(\theta(\bm{n}), \sigma) = 1\\  &|\sigma| = l-1
  \end{aligned} \right\}}{n/2-l+2} + \order{N^{n/2 - l +1}}.
\end{equation*}
We shall see in Section \ref{sec:combinatorial-cumulants} a combinatorial description of the
terms of the expansion.
\begin{proof}
  We start by noticing that in \eqref{eq:cumulants-combin}, we can
  make the bijective change of variable
  \begin{equation*}
    \begin{cases}
      \motz_{k}^{[N]}(\theta(\bm{k})) &\to \left\{ (h, \gamma') \in \N^{*} \times \motz_{k, 0}(\theta(\bm{k})) \colon h \geq \max \gamma' + 1 \right\}\\
      \gamma &\mapsto (h, \gamma') = (\max \gamma, \max \gamma - \gamma).
    \end{cases}
  \end{equation*}
  We have $\Delta\gamma_{+} = \Delta\gamma'_{-}$, $\Delta\gamma_{-} = \Delta\gamma'_{+}$,
  and $\Delta\gamma_{0} = \Delta\gamma'_{0}$. We get
  \begin{equation*}
    \kappa_{l}(\bm{k}) = (\frac{2}{\beta})^{k/2}\sum_{\gamma \in \motz_{k, 0}(\theta(\bm{k}))}\sum_{h \geq \max \gamma + 1}\sum_{\substack{\sigma \in \Sym^{h-\gamma}\\ \Orbit( \theta(\bm{k}), \sigma) = 1}}\prod_{\pi\in\Cyc(\sigma\vert_{\Delta\gamma_{+}})}\frac{\beta}{2} \left( h - 1 - \gamma(\pi)\right).
  \end{equation*}
  Given $\gamma \in \motz_{k, 0}(\theta(\bm{k}))$, we choose any bijection
  $\tilde{\phi} \colon \Delta\gamma_{+} \to \Delta\gamma_{-}$ satisfying for all
  $i \in \Delta\gamma_{+}$: $i$ and $\gamma(i)$ are part of the same cycle of
  $\theta(\bm{k})$ and
  \begin{equation*}
    \gamma(\tilde{\phi}(i)) = \gamma(i) + 1.
  \end{equation*}
  Such a bijection exists since $\gamma$ is a Motzkin bridge: for any level
  $n$, there are as many up-steps between $n$ and $n+1$ as down-steps
  between $n+1$ and $n$. We extends the definition of $\tilde{\phi}$ to a
  bijection (actually, an involution) $\phi \colon [k] \to [k]$ by
  \begin{equation*}
    \phi(i) = \begin{cases}
      \tilde{\phi}(i) &\text{ if } i \in \Delta\gamma_{+}\\
      \tilde{\phi}^{-1}(i) &\text{ if } i \in \Delta\gamma_{-}\\
      i &\text{ if } i \in \Delta\gamma_{0}.
    \end{cases}
  \end{equation*}
  This bijection allows us to define the change of variable
  \begin{equation*}
    \sigma \in \Sym_{h-\gamma} \mapsto \phi^{-1}\circ\sigma\circ\phi \in \Sym_{\gamma}.
  \end{equation*}
  We thus have
  \begin{equation*}
    \begin{split}
      \sum_{\substack{\sigma \in \Sym_{h-\gamma}\\ \Orbit( \theta(\bm{k}), \sigma ) = 1}}\prod_{\pi\in\Cyc(\sigma\vert_{\Delta\gamma_{+}})}\frac{\beta}{2} \left( h - 1 - \gamma(\pi)\right)
      &= \sum_{\substack{\sigma \in \Sym_{\gamma}\\ \Orbit( \theta(\bm{k}), \sigma ) = 1}}\prod_{\pi\in\Cyc(\phi\circ\sigma\circ\phi\vert_{\Delta\gamma_{+}})}\frac{\beta}{2} \left( h - 1 - \gamma(\pi)\right)\\
      &= \sum_{\substack{\sigma \in \Sym_{\gamma}\\ \Orbit( \theta(\bm{k}), \sigma ) = 1}}\prod_{\pi\in\Cyc(\sigma\vert_{\Delta\gamma_{-}})}\frac{\beta}{2} \left( h - \gamma(\pi)\right).
    \end{split}
  \end{equation*}
  Note that the transitivity condition is not changed because of our
  constraints that $i$ and $\phi(i)$ must be part of the same cycle of
  $\theta(\bm{k})$.

  The cumulant can be rewritten as
  \begin{equation*}
    \kappa_{l}(\bm{n}) = (\frac{2}{\beta})^{n/2}\sum_{\gamma \in \motz_{n, 0}(\theta(\bm{n}))}\sum_{h=\max\gamma+1}^{N}\sum_{\substack{\sigma \in \Sym^{\gamma}\\ \Orbit( \theta(\bm{n}), \sigma ) = 1}}\prod_{\pi\in\Cyc(\sigma_{-})}\frac{\beta}{2} \left( h - \gamma(\pi)\right).
  \end{equation*}
  Notice that when $1 \leq h \leq \max \gamma$, the product is 0 so that
  \begin{equation}\label{eq:substitution-gamma}
    \kappa_{l}(\bm{n}) = (\frac{2}{\beta})^{n/2}\sum_{\gamma \in \motz_{n, 0}(\theta(\bm{n}))}\sum_{h=1}^{N}\sum_{\substack{\sigma \in \Sym^{\gamma}\\ \Orbit( \theta(\bm{n}), \sigma) = 1}}\prod_{\pi\in\Cyc(\sigma_{-})}\frac{\beta}{2} \left( h - \gamma(\pi)\right).
  \end{equation}

  Set $p = |\sigma|$, and notice that $c(\sigma_{-}) = n/2 - p$. We expand the
  product -- reminiscent of a characteristic polynomial -- as
  \begin{equation*}
    \prod_{\pi\in\Cyc(\sigma_{-})}\frac{\beta}{2} \left( h - \gamma(\pi)\right)
    = \left( \frac{\beta}{2} \right)^{n/2-p}\sum_{q+u=n/2 - p}(-1)^{q} e_{q}\left( \gamma(\pi); \pi \in \Cyc(\sigma_{-}) \right)h^{u}.
  \end{equation*}

  Hence, we get
  \begin{equation*}
    \kappa_{l}(\bm{n}) = (\frac{2}{\beta})^{n/2}\sum_{\gamma \in \motz_{n, 0}(\theta(\bm{n}))}\sum_{\substack{\sigma \in \Sym^{\gamma}\\ \Orbit( \theta(\bm{n}), \sigma) = 1}}\left( \frac{\beta}{2} \right)^{n/2-p}\sum_{q+u=n/2 - p}(-1)^{q} e_{q}\left( \gamma(\pi); \pi \in \Cyc(\sigma_{-}) \right)\sum_{h=1}^{N}h^{u}.
  \end{equation*}
  Upon introducing the Newton sum $P_{u}(k) = \sum_{h = 1}^{k}h^{u}$ and rearranging the terms, we get
  \begin{equation}\label{eq:cumulant-Newton}
    \kappa_{l}(\bm{n}) = \sum_{p+q+u=n/2}\left( \frac{2}{\beta}\right)^{p} (-1)^{q}P_{u}(N)\sum_{\gamma \in \motz_{n, 0}(\theta(\bm{n}))}\sum_{\substack{\sigma \in \Sym^{\gamma}\\ \Orbit( \theta(\bm{n}), \sigma) = 1\\|\sigma|  = p}}e_{q}\left( \gamma(\pi); \pi \in \Cyc(\sigma_{-}) \right).
  \end{equation}
  This gives the first claim.

  To get the second one, we use Faulhaber's formula
  \begin{equation}\label{eq:Faulhaber}
    P_{u}(N) = \sum_{h=1}^{N}h^{u} = \sum_{r+s=u}\binom{r+s}{r}\frac{B_{r}}{s+1}N^{s+1}.
  \end{equation}
  Using \eqref{eq:cumulant-Newton} and \eqref{eq:Faulhaber}, we get
  \begin{equation*}
    \kappa_{l}(\bm{n}) = \sum_{p+q+r+s = n/2}\sum_{\substack{\gamma\in\motz_{n, 0}(\theta(\bm{n}))\\\sigma \in \Sym_{\gamma}, |\sigma| = p\\ \Orbit( \theta(\bm{n}), \sigma) = 1}} \left( \frac{2}{\beta} \right)^{p}\frac{(-1)^{q}B_{r}}{s+1}\binom{r+s}{r} N^{s+1}e_{q}\left( \gamma(\pi); \pi \in \Cyc(\sigma_{-}) \right),
  \end{equation*}
  as wanted.
\end{proof}

\section{Maps and labelled hypermaps}
\label{sec:maps-mobiles}

We introduce the notions needed to reinterpret
\eqref{eq:cumulants-combin} in terms of maps. We first recall the
definition of a map, and then discuss the bijection between suitably
labelled maps and labelled hypermaps introduced by Bouttier, Fusy, and
Guitter \cite{bouttier_twopoint_2014}. It is a generalization of
the bijection between pointed planar maps and labelled trees called
mobiles introduced in \cite{bouttier_planar_2004}. In the process, we give a
combinatorial description of these objects in terms of permutations
and Motzkin paths.

\subsection{Maps and permutations}
\label{sec:maps-permutations}

We recall some notions pertaining to maps. For more details, see
\cite{lando_graphs_2004} and \cite{mohar_graphs_2001}.

\begin{definition}\label{def:embedded-graph}
  Let $\Gamma$ be a graph (with possibly multi-edges and loops), seen as a
  1-dimensional cell complex, and $S$ be a connected compact surface
  without boundaries. A \textbf{cellular embedding} of  $\,\Gamma$ into $S$ is an
  embedding $\iota$ of $\Gamma$ into $S$, such that $S \setminus \iota(\Gamma)$ is a disjoint
  union of simply connected open sets of $S$. The corresponding \textbf{embedded graph}
  is the tuple $(\Gamma, S, \iota)$.
\end{definition}

\begin{definition}\label{def:map}
  Two embedded graphs $(\Gamma, S, \iota)$ and $(\Gamma', S', \iota')$ are \textbf{isomorphic} if
  there exists an orientation-preserving homeomorphism $\varphi\colon S \to S'$
  such that $\varphi\circ\iota(\Gamma) = \iota'(\Gamma')$ and $\iota'^{-1}\circ\varphi\circ\iota\vert_{\Gamma}$ is a graph
  isomorphism $\Gamma \to \Gamma'$. A \textbf{map} is a class of connected embedded graphs
  taken up to isomorphism.
\end{definition}
\begin{remark}\label{rem:map-add-structure}
  In the sequel, we will consider maps with additional structure that
  depends on the underlying graph $\Gamma$ or the embedding $\iota$. In
  these cases, the homeomorphism $\varphi$ is taken to furthermore
  preserve this additional structure.
\end{remark}
\begin{definition}\label{def:bipartite-labelled}
  A \textbf{hypermap} is a map whose vertices are colored in white or black,
  and such that every edge connects a white vertex to a black vertex.
  An \textbf{edge-labelled map} is a map whose edges are labelled in a
  bijective way from 1 to $n$, where $n$ is the number of edges in the
  map.
\end{definition}
Until the end of the Section, we work exclusively with maps on
orientable surfaces. We will consider possibly non-orientable maps in
Section \ref{sec:many-betw-suit}.

We orient each edge from its white vertex to its black vertex, and
thus define a left and right side of the edge. Let $e$ and $f$ be
respectively an edge and a face of a hypermap. We say that $e$ is
incident to $f$ or that $f$ is incident to $e$ if $f$ is at the left
of $e$.

\begin{figure}[ht]
  \centering
  \includegraphics[width=0.5\textwidth]{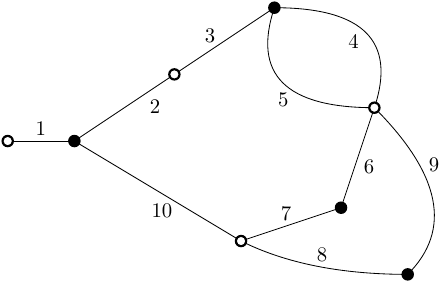}
  \caption{\label{fig:hypermap} A hypermap with labelled edges.}
\end{figure}

Finally, we will use the notion of corners and rooted maps.
\begin{definition}\label{def:corner}
  Consider a map $\map$. A \textbf{corner} of $\map$ is a vertex together with
  an angular sector comprised between two consecutive edges incident
  to the same vertex and the same face. A \textbf{rooted map} is a map with the
  choice of a distinguished oriented corner.
\end{definition}

In the orientable case, a result of Edmonds \cite{edmonds_combinatorial_1960} (see for instance \cite{lando_graphs_2004}
for a modern account) shows that edge-labelled hypermaps with $n$
edges are in bijection with pairs of permutations
$(\theta, \sigma) \in \Sym_{n}^{2}$. We now recall this construction.
\begin{constr}\label{constr:hypermap-to-permutations}
  Consider a hypermap $\hypermap$. We construct a pair of permutations
  $(\theta_{\hypermap}, \sigma_{\hypermap})$. Each black vertex of
  $\hypermap$ corresponds to a cycle of $\theta$ and each white vertex
  to a cycle of $\sigma$. Let $w$ be a white vertex. Assume that when
  going around $w$ in the clockwise direction, we encounter the edges
  labelled $u_{1}, \ldots, u_{k}$. We associate to $w$ the cycle
  $\rho = \cycle{u_{1}, u_{2}, \cdots, u_{k}}$ in $\sigma$. We do this
  for all the white vertices of the map, and proceed similarly for the
  black vertices, which corresponds to cycles of $\theta$.
\end{constr}
This construction defines an injective function from the set of
bipartite labelled maps with $n$ edges to $\Sym_{n}^{2}$. This map can
be shown to be surjective (see
\cite{edmonds_combinatorial_1960}).

It is convenient to define the permutation
$\varphi_{\hypermap} = \theta_{\hypermap}^{-1}\sigma_{\hypermap}^{-1}$. Each
cycle of $\varphi_{\hypermap}$ corresponds to a face of $\hypermap$. Assume
a face $f$ of $\hypermap$ is incident to edges labelled
$u_{1}, \ldots, u_{k}$, and that these labels are encountered in that order
when going around the boundary of the face in clockwise order. Then,
$\cycle{u_{1}, \ldots, u_{k}}$ is a cycle of $\varphi_{\hypermap}$. This result
is proved in \cite[Proposition 1.3.16]{lando_graphs_2004}.

\begin{ex}\label{ex:permutations-hm}
  The hypermap $\hypermap$ depicted in Figure \ref{fig:hypermap} is encoded by the
  permutations
  \begin{equation*}
    \begin{split}
      \theta_{\hypermap} &= \cycle{1, 2, 10}\cycle{3, 4, 5}\cycle{6, 7}\cycle{8, 9}\\
      \sigma_{\hypermap} &= \cycle{1}\cycle{2, 3}\cycle{4, 9, 6, 5}\cycle{7, 8, 10}\\
      \varphi_{\hypermap} &= \cycle{1, 10, 9, 3}\cycle{2, 5, 7}\cycle{4}\cycle{6, 8}.
    \end{split}
  \end{equation*}
\end{ex}

\begin{remark}\label{rem:faces-permutations}
  In particular, the hypermap has $c(\varphi_{\hypermap})$ faces. This
  number of faces is related to the genus $g_{\hypermap}$ of
  $\hypermap$ according to Euler's formula:
  \begin{equation}\label{eq:euler-hypermap}
    \left( c(\theta_{\hypermap}) + c(\sigma_{\hypermap}) \right) - n + c(\varphi_{\hypermap}) = 2 - 2g_{\hypermap}.
  \end{equation}
\end{remark}

\begin{remark}\label{rem:maps-half-edges}
  Maps $\map$ with non-colored vertices can be seen as hypermaps by
  coloring the vertices of $\map$ black and adding a white vertex in
  the middle of each edge. We obtain a hypermap $\hypermap(\map)$ with
  all its white vertices of degree $2$. A hypermap obtained in such a
  way can have its edge labelled and be described by a pair
  $(\theta_{\hypermap(\map)},\sigma_{\hypermap(\map)})$ with $\sigma_{\hypermap}$ a
  matching (recall Definition \ref{def:matching}).

  We call the edges of $\hypermap(\map)$ the \emph{half-edges} of $\map$. To
  each half-edge $h$, we denote the vertex to which it is attached by
  $\vertex(h)$. Furthermore, there is a unique distinct half-edge $h'$
  such that $h$ and $h'$ form an edge. We say that $h'$ is the \textbf{counterpart}
  of $h$. We will say that we label the half-edges of $\map$ to mean
  that we label the edges of $\hypermap(\map)$. We can then set
  \begin{equation*}
    \theta_{\map} \coloneq \theta_{\hypermap(\map)}, \quad \sigma_{\map} \coloneq \sigma_{\hypermap(\map)}, \text{ and } \varphi_{\map} \coloneq \varphi_{\hypermap(\map)}.
  \end{equation*}
\end{remark}

\begin{remark}\label{rem:act-on-H}
  Let $I$ be a finite subset of $\N^{*}$, $\hypermap$ be a hypermap,
  and $E_{\hypermap}$ be the set of edges of the hypermap. When a
  hypermap $\hypermap$ is edge-labelled with labels in $I$, it is
  equipped with a bijection $\lambda\colon E_{\hypermap} \to I$. For any
  permutation $\pi \in \Sym(I)$, we can construct naturally the
  permutation of the edges
  \begin{equation*}
    \pi^{\lambda} = \lambda^{-1} \circ \pi \circ \lambda \in \Sym(E_{\hypermap}).
  \end{equation*}
  In particular, we define naturally the permutations of the edges
  $\theta_{\hypermap}^{\lambda}$, $\sigma_{\hypermap}^{\lambda}$, and
  $\varphi_{\hypermap}^{\lambda}$. We abuse notation in the sequel and omit
  the superscript $\lambda$ when it is not ambiguous.

  This construction also applies to maps: in this case we replace
  $\hypermap$ by a maps $\map$, the set $E_{\hypermap}$ by the set
  $H_{\map}$ of half-edges of $\map$, and
  $\lambda\colon E_{\hypermap} \to I$ by a bijection $H_{\map} \to I$.
\end{remark}

\subsection{Well-labelled hypermaps and suitably labelled maps}
\label{sec:mobil-suit-labell}

In \cite{bouttier_planar_2004}, Bouttier, Di Francesco, and Guitter introduced a celebrated
bijection between maps and a family of trees with labelled vertices
called \emph{mobiles}. In the investigation of the 2-point functions,
Bouttier, Fusy, and Guitter \cite{bouttier_twopoint_2014} introduced a generalization of the
bijection. It allowed them to relate \emph{suitably labelled maps} of any
genus and \emph{well-labelled hypermaps}. We are going to use this bijection
in the sequel.

\begin{definition}\label{def:suitably-labelled}
  A suitably labelled map is a map $\map$ such that each vertex $v$ of
  $\map$ carries a label $l(v) \in \N$ satisfying:
  \begin{itemize}
    \item $\min_{v} l(v) = 0$,
    \item for each edge $e$ between vertices $v$ and $w$ we have
          $|l(v) - l(w)| \leq 1$.
  \end{itemize}

  An edge between two vertices $v$ and $v'$ in a suitably labelled map
  is said to be frustrated if $l(v) = l(v')$. We denote by
  $\Suitably_{n}$ the set of suitably labelled maps with $n$
  half-edges, and $\hat{\Suitably}_{n}$ the set of such suitably
  labelled maps with no frustrated edges.
\end{definition}
Note that the definition given here is slightly different from the one
in \cite{bouttier_twopoint_2014}, where the authors allowed $l(v) \in \Z$ and no constraint on the
minimum of the labels. We similarly give a modified version of a
well-labelled hypermap that mirrors these changes.
\begin{definition}
  A well-labelled hypermap is a hypermap $\hypermap$ such that each
  white vertex $w$ carries a label $l(w) \in \N^{*}$ satisfying:
  \begin{itemize}
    \item $\min_{w}l(w) = 1$,
    \item Let $b$ be a black vertex and $w, w'$ be two white vertices
          adjacent to $b$, with $(b, w)$ and $(b, w')$ consecutive
          edges, in that order in the clockwise orientation. Then,
          $l(w') \geq l(w)-1$.
  \end{itemize}
  Furthermore, if a white vertex $w$ of degree $2$ that is connected to black
  vertices $b_{1}$ and $b_{2}$ is preceded (in the clockwise
  direction) around both of $b_{1}$ and $b_{2}$ by white vertices
  $w_{1}$ and $w_{2}$ of label $l(w_{1}), l(w_{2}) \leq l(w)$, then it
  may be marked. We call these marked vertices \emph{frustrated
    vertices}. We denote by $\welllab_{n}$ the set of well-labelled
  hypermaps with $n$ edges, and $\hat{\welllab}_{n}$ the set of such
  maps with no frustrated vertices.
\end{definition}
Note that the definition of well-labelled hypermap is given in \cite{bouttier_twopoint_2014} in
terms of face-bicolored Eulerian map. The definition we give, in terms
of star-representation of a hypermap, is equivalent.

\paragraph{The cw-type of a face}
We now recall a notation introduced in \cite{bouttier_twopoint_2014} to describe the faces of
suitably labelled maps, and black vertices of hypermaps (corresponding
in to dark faces of hypermaps in the conventions of \cite{bouttier_twopoint_2014}).
\begin{definition}
  The cw-type of a face $f$ of a suitably labelled map is the cyclic
  list of the labels of the vertices adjacent to $f$.

  The cw-type $\tau$ of a black vertex $b$ of a well-labelled hypermap is the
  cyclic list of the the labels of the white vertices adjacent to $b$,
  in clockwise order.

  The lower completion of $\tau$, denoted by $c^{\downarrow}(\tau)$,
  is the cyclic list obtained by inserting $i-1, \ldots, j-1$ between
  two consecutive elements $i \leq j$ of $\tau$.
\end{definition}

\paragraph{From well-labelled hypermap to suitably labelled maps}
Let us recall briefly the construction of \cite{bouttier_twopoint_2014} to go from a
well-labelled hypermap to a suitably labelled map.

\begin{constr}\label{constr:bfg}
  Start from the well-labelled hypermap $(\hat{\hypermap}, l)$. Denote
  by $\min f$ the minimum label of a white vertex incident to a face
  $f$. For each face $f$ we proceed as follows.
  \begin{enumerate}
    \item\label{item:add-vertex-face} Add a new white vertex $w_{f}$
          labelled by $\min f - 1$ in the interior of $f$.
    \item\label{item:insert-he} For each corner $c$ incident to $f$
          and at a white vertex $w$, we attach a half-edge $h_{c}$.
    \item\label{item:edges-face-BDG} For each added half-edge $h_{c}$
          attached to a white vertex $w$ we consider the label $l(w)$.
          If $l(w) = \min f$, we attach this half-edge $h_{c}$ to
          $w_{f}$. Otherwise, we connect $h_{c}$ to the next corner in
          the counterclockwise order which is at a vertex labelled by
          $l(w) - 1$. We call this second corner the \emph{successor}
          of $c$.
  \end{enumerate}
  We do this for all faces of $\hat{\hypermap}$. Finally, we remove
  all the edges and the black vertices of $\hat{\hypermap}$. We obtain
  a map $\hat{\map}$.
\end{constr}

The inverse construction is as follows.
\begin{definition}\label{def:decreasing-he}
  Let $h$ be a half-edge in a suitably labelled map $\hat{\map}$,
  incident to a face $f$. Let $h'$ be the counterpart of $h$. Let $v$ be
  the vertex incident to $h$ and $v'$ be the vertex incident to $h'$.
  We say that $h$ is a \textbf{decreasing half-edge along} $f$ if
  \begin{equation*}
    l(v') = l(v)-1.
  \end{equation*}
  We say $h$ is an \textbf{increasing half-edge along} $f$ if
  \begin{equation*}
    l(v') = l(v)+1.
  \end{equation*}
\end{definition}

\begin{constr}\label{constr:inv-bfg}
  Start from a suitably labelled map $(\hat{\map}, \ell)$.
  \begin{enumerate}
    \item Color all the vertices of $\hat{\map}$ white.
    \item For each face $f$ in $\hat{\map}$, add a new black vertex $b$ in
          its interior. For each decreasing half-edge incident to $f$
          and connected to a white vertex $w$, add an edge between $w$
          and $b$.
    \item Erase all the edges of the original map, and the isolated white vertices.
  \end{enumerate}
\end{constr}

\begin{theorem}{\cite[Theorem
    1]{bouttier_twopoint_2014}}\label{thm:bfg}
  Constructions \ref{constr:bfg} and \ref{constr:inv-bfg} give a bijection between
  $\hat{\Suitably}_{n}$ and $\hat{\welllab}_{n}$. For a well-labelled
  hypermap $\hat{\hypermap}$ corresponding to a suitably labelled map
  $\hat{\map}$,
  \begin{itemize}
    \item each white vertex $w$ of $\hat{\hypermap}$ corresponds to a
          non local minimum vertex $v$ of $\hat{\map}$ of the same
          label;
    \item this vertex $w$ is a local maximum if and only if $v$ is a local
          maximum in $\hat{\map}$;
    \item each face of $\hat{\hypermap}$ corresponds to a local minimum
          vertex of $\hat{\map}$, of label $\min f - 1$;
    \item each black vertex of $\hat{\hypermap}$ of cw-type $\tau$
          corresponds to a face of $\hat{\map}$ of cw-type
          $c^{\downarrow}(\tau)$.
  \end{itemize}
\end{theorem}

Actually, \cite[Theorem 1]{bouttier_twopoint_2014} is stated only as
a bijection between $\hat{\welllab}_{n}$ and $\hat{\Suitably}_{n}$.
Generalizing this to the general case is straightforward using the
duplication of edges trick, used in particular in
\cite{bouttier_planar_2004}. Given a well-labelled hypermap
$(\hypermap, l)$ with frustrated vertices, we produce a suitably
labelled map $(\hat{\map}, \ell)$ with some of its white vertices
marked. Indeed, the non-local minimum vertices of the map $\map$
constructed in Construction \ref{constr:bfg} are the white vertices of
$\hypermap$: the possible marking of white vertices in $\hypermap$
induce a marking of vertices in $\map$. We call those marked vertices
the frustrated vertices of $\hat{\map}$.
\begin{lemma}\label{lem:deg-2-fr-map}
  The frustrated vertices of $\hat{\map}$ are of degree $2$.
\end{lemma}
For each frustrated vertex $v$ in $\hat{\map}$ incident to two edges
$e_{1}$ and $e_{2}$, we remove $v$ from $\hat{\map}$ and glue $e_{1}$
and $e_{2}$ together. We obtain a suitably labelled map $(\map, \ell)$
with one frustrated edge for each removed frustrated vertex.

The construction works in the converse direction: consider a suitably
labelled map $(\map, \ell)$. For each frustrated edge between vertices
$v_{1}$ and $v_{2}$, we add a new vertex $v$ in the middle of it which
we label by $\ell(v) = \ell(v_{1}) + 1 = \ell(v_{2}) + 1$. We obtain
in this way a suitably labelled map without frustrated edge. The
well-labelled hypermap produced by Construction \ref{constr:inv-bfg}
has naturally frustrated vertices. Consider such a vertex $w$,
corresponding to a frustrated vertex $v$ in $\map$. The choice of
labelling for $v$ ensures that $w$ satisfies the condition to be a
frustrated vertex.

In remains to prove Lemma \ref{lem:deg-2-fr-map}.
\begin{proof}[Proof of Lemma \ref{lem:deg-2-fr-map}]
  Let $v$ be a frustrated vertex in $\map$, coming from a frustrated
  vertex in $\hypermap$. The vertex $v$ is at least of degree $2$:
  since $w$ is of degree $2$, it has two corners and thus at step
  \ref{item:insert-he} two half-edges get connected to it. For the
  degree to be at least $3$, there must be another white vertex $w'$
  incident to the same face $f$ as $w$, with $l(w') = l(w)+1$.
  However, when going around $f$ in the counterclockwise orientation,
  the label between two consecutive white vertices decrease at most by
  $1$. this means that $w'$ is the white vertex preceding $w$ when
  going around $f$ in the counterclockwise orientation. By definition
  of a marked vertex, we would have $l(w') \leq l(w)$, a contradiction.
\end{proof}

\subsection{Encoding the labelled hypermaps}
\label{sec:encod-labell-hyperm}

In Section \ref{sec:large-n-expansion}, we expressed the cumulants of
the $\beta$-ensemble in terms of a sum over a Motzkin path
$\gamma \in \motz_{n, 0}(\theta)$ and a permutation.
$\sigma \in \Sym^{\gamma}$. In this Section, we explain how this data
allow us to define a labelled hypermap, and thus a suitably labelled
map by the Bouttier-Fusy-Guitter construction \ref{constr:bfg}. To
introduce the main result of this Section, we define the notion of
restriction of a permutation.
\begin{definition}\label{def:restriction}
  Let $I \subset I'$ two finite sets and $\pi \in \Sym(I')$. We define the
  jump in $\pi$ with respect to $J$ by
  \begin{equation*}
    J_{\theta, I}(j) = \min \left\{ p \in \N^{*} \colon \pi^{p}(j) \in I\right\} \text{ for all } j \in I.
  \end{equation*}
  We define the restriction of $\pi$ to $J$ by
  \begin{equation*}
    \pi\vert_{J \to I}(j) = \pi^{J_{\theta, I}(j)}(j) \text{ for } j \in I.
  \end{equation*}.
\end{definition}
Note that in general $\pi \vert_{J \to J} \in \Sym(J)$ differs from
$\pi\vert_{J} \colon J \to \pi(J)$.

In the following Proposition, we make use of the following abuse of
notation. Let $(\hypermap, l)$ be a well-labelled hypermap, whose
edges are labelled by $I \subset \N^{*}$. For each $i \in I$, there is
a unique white vertex $w_{i}$ incident to the edge labelled $i$. We
write
\begin{equation*}
  l(i) = l(w_{i}).
\end{equation*}
The main result of this Section is the
following.
\begin{prop}\label{prop:gamma-sigma-to-hypermap}
  Let $n \in \N^{*}$, $\theta \in \Sym_{n}$, and $I \subset [n]$. For all $\pi \in \Sym(I)$ we
  define
  \begin{equation*}
    \welllab(\theta) \coloneq \left\{ (\hat{\hypermap}, l) \in \welllab_{n} \colon \exists I \subset [n],\quad \begin{aligned}
      \bullet\, & \hat{\hypermap} \text{ is edge-labelled by }I\\
      \bullet\, & \theta_{\hat{\hypermap}} = \theta_{I \to I}\\
      \bullet\, & l \circ \theta_{\hat{\hypermap}} = l + J_{\theta, I} - 2
    \end{aligned} \right\}.
  \end{equation*}
  and
  \begin{equation*}
    \Comb(\theta) = \left\{ (\gamma, \sigma) \in \motz_{n,0}(\theta)\times \Sym_{n} \colon \quad\begin{aligned}
        \bullet \,& \sigma \in \Sym_{\gamma}\\
        \bullet \,& \Orbit( \theta(\bm{n}), \sigma) = 1\\
      \end{aligned} \right\}.
  \end{equation*}

  Construction \ref{constr:gamma-sigma-to-hypermap} below gives a
  bijection
  \begin{equation*}
    \Comb(\theta) \to \welllab(\theta).
  \end{equation*}
\end{prop}
The set $\Comb(\theta, I)$ appear naturally in the expression of the
cumulants. The last condition is a technical assumption needed for a
proper labelling of the edges of the corresponding hypermap.

\begin{constr}\label{constr:gamma-sigma-to-hypermap}
  The inverse of Construction \ref{constr:hypermap-to-permutations}
  defines a edge-labelled hypermap $\hypermap$. Notice that each cycle
  of $\sigma_{+}$ -- corresponding to an element of $\Delta\gamma_{+}$
  -- corresponds to a white vertex of degree 1 in $\hypermap$. We
  remove these vertices to obtain a new hypermap $\hat{\hypermap}$.
  The frustrated vertices of $\hat{\hypermap}$ are the vertices
  corresponding to the cycles of $\sigma_{0}$. We have
  \begin{equation*}
    \theta_{\hat{\hypermap}} = \theta\vert_{\Delta\gamma_{-}\sqcup\Delta\gamma_{0} \to \Delta\gamma_{-}\sqcup\Delta\gamma_{0}} \text{ and } \sigma_{\hat{\hypermap}} = \sigma\vert_{\Delta\gamma_{-}\sqcup\Delta\gamma_{0}}.
  \end{equation*}

  We now explain how
  $\gamma$ induces a labelling of the white vertices of
  $\hat{\hypermap}$. Let $w$ be a white vertex of $\hat{\hypermap}$
  that corresponds to
  $\pi \in \Cyc(\sigma_{-}) \cup \Cyc(\sigma_{0})$. We set
  \begin{equation*}
    l(w) =
    \begin{cases}
      \gamma(\pi) &\text{ if } \pi \in \Cyc(\sigma_{-})\\
      \gamma(\pi) + 1 &\text{ if } \pi \in \Cyc(\sigma_{0}).
    \end{cases}
  \end{equation*}
  We write in the sequel
  $(\hat{\hypermap}, l) = \Psi(\gamma, \theta, \sigma)$.
\end{constr}

\begin{figure}[ht]
  \centering
  \includegraphics[width=0.8\textwidth]{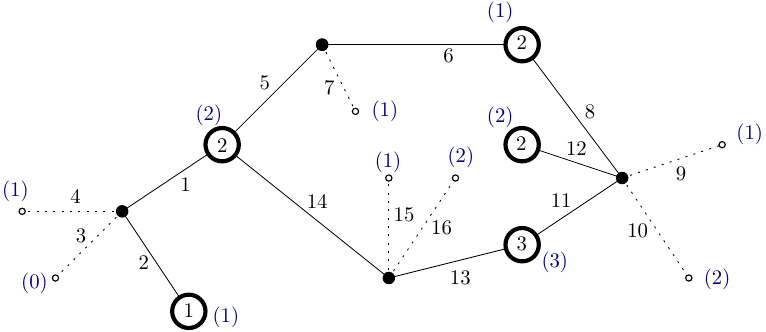}
  \caption{\label{fig:hypermap-from-path} A labelled hypermap
    $(\hat{\hypermap}, l) = \Psi(\gamma, \theta, \sigma)$. We wrote in parenthesis the
    value of the path $\gamma$ at each vertex. We displayed in dotted edges
    the edges to white vertices corresponding to elements of $\Delta\gamma_{+}$.
    They belong to $\hypermap$ but not to $\hat{\hypermap}$.}
\end{figure}

\begin{ex}\label{ex:labelled-hypermap}
  The labelled hypermap displayed in Figure
  \ref{fig:hypermap-from-path} is obtained from
  \begin{equation*}
    \begin{split}
      (\gamma(i))_{i \in [16]} &= (2, 1, 0, 1, 2, 1, 1, 1, 1, 2, 3, 2, 3, 2, 1, 2)\\
      \theta &= \cycle{1, 2, 3, 4}\cycle{5, 6, 7}\cycle{8, 9, 10, 11, 12}\cycle{13, 14, 15, 16}\\
      \sigma &= \cycle{1, 5, 14}\cycle{6, 8}\cycle{11, 13}.
    \end{split}
  \end{equation*}
  Note that
  \begin{equation*}
    \begin{split}
      \Delta\gamma_{+} &= \left\{ 3, 4, 7, 9, 10, 15, 16 \right\}\\
      \Delta\gamma_{0} &= \left\{ 6, 8 \right\}\\
      \Delta\gamma_{-} &= \left\{ 1, 2, 5, 11, 12, 13, 14 \right\}.
    \end{split}
  \end{equation*}
\end{ex}

\begin{proof}[Proof of Proposition \ref{prop:gamma-sigma-to-hypermap}]
  We start by showing that
  $(\hat{\hypermap}, l) = \Psi(\gamma, \theta, \sigma)$ is a
  well-labelled hypermap. Recall that $\min \gamma = 0$. Thus,
  $\min_{w} l(w) \geq 1$, and for $i$ such that $\gamma(i) = 0$ we
  have $\gamma(\theta(i)) \in \{0, 1\}$ so that $i$ is the label of an
  edge connected to a white vertex $w$. The label of $w$ is
  $l(w) = 1$. We have shown that $\min_{w} l(w) = 1$.

  Consider a black vertex $b$, and two white vertices $w$ and $w'$
  such that $(b, w)$ and $(b, w')$ are consecutive edges around $b$,
  in that order in the clockwise direction. Let $i, j$ be the labels
  of these edges. We have
  $i, j \in \Delta\gamma_{0}\cup \Delta\gamma_{-}$. Thus, as $\gamma$
  is a Motzkin path, we have that $\gamma(i) \leq \gamma(j) + 1$. It
  implies that $l(w)-1 \leq l(w')$. The hypermap is thus well-labelled.

  We remark that by construction its edges are labelled by elements of
  $I = \Delta\gamma_{-}\sqcup\Delta\gamma_{0}$, that
  $\theta_{\hat{\hypermap}} = \theta\vert_{I \to I}$, and that since
  $\gamma$ is a Motzkin path, the difference of label between two
  consecutive white vertices incident to a black vertex is given by
  the number of removed univalent white vertices minus one. It gives
  \begin{equation*}
    \gamma_{\hat{\hypermap}} \circ \theta_{\hat{\hypermap}} = \gamma_{\hat{\hypermap}} + J_{\theta, I} - 2.
  \end{equation*}

  We now show that the mapping is a bijection by constructing the
  inverse map. Let $(\hat{\hypermap}, l) \in \welllab(\theta)$, and
  let $I$ be a edge-labelling set. We will define a function $\tilde{l}$ on
  $[n]$. On $I$, it is defined by $\tilde{l}\vert_{I} = l$.
  We construct a new hypermap $\hypermap$ from $\hat{\hypermap}$. We
  add degree one white vertices, whose incident edges are labelled by
  elements of $[n] \setminus I$. Fix a black vertex $b$ and let
  $\pi_{b} = \cycle{u_{1}, \ldots, u_{k}}$ be the corresponding cycle in
  $\theta_{\hat{\hypermap}} = \theta_{I \to I}$. For each $j \in [k]$,
  let $p_{j}$ be the jump as in Definition \ref{def:restriction}:
  \begin{equation*}
    p_{j} = J_{\theta, I}(u_{j}) = \min \left\{ p \in \N^{*} \colon \theta^{p}(u_{j}) \in I \right\}.
  \end{equation*}
  We add after the edge labelled $u_{j}$ in the clockwise orientation
  $p_{j} - 1 = J_{\theta, I}(u_{j})$ edges connected to univalent white vertices,
  labelled by
  \begin{equation*}
    \theta^{1}(u_{j}), \theta^{2}(u_{j}), \ldots, \theta^{p_{j} - 1}(u_{j}).
  \end{equation*}
  We set for all $1 \leq i \leq p_{j}-1$,
  \begin{equation*}
    \tilde{l}(\theta^{i}(u_{j})) = \tilde{l}(u_{j}) - 2 + i.
  \end{equation*}
  Since we have (with the convention $u_{k+1} = u_{1}$)
  \begin{equation*}
    l \circ \tilde{\theta}^{p_{j}}(u_{j}) = l(u_{j+1}) = l \circ \theta\vert_{I \to I}(u_{j}) = \tilde{l}(u_{j}) + p_{j} - 2,
  \end{equation*}
  we see that $\tilde{l}$ is a Motzkin path without flat steps. Let
  $F\subset I$ be the set of labels of edges connected to frustrated
  vertices. We then set for all $i \in [n]$
  \begin{equation*}
    \gamma(i) = \tilde{l} - \ind_{F},
  \end{equation*}
  where $\ind_{F}$ denote the indicator function of $F$. The function
  \(\gamma\) is a Motzkin path with $\Delta\gamma_{0} = F$. We extend $\sigma_{\hypermap}$
  by the identity to a permutation of $[n]$. By construction, we then
  have $\sigma \in \Sym_{\gamma}$, and the connectedness of $\hypermap$ ensures
  that $\left< \theta, \sigma \right>$ acts transitively on $[n]$, i.e.\ that
  $\Orbit(\theta, \sigma) = 1$. We thus have that
  \begin{equation*}
    (\gamma, \sigma) \in \Comb(\theta).
  \end{equation*}
  The mapping just constructed is the required inverse: it is clear
  that the mapping between $(\gamma, \sigma)$ and $(\hypermap, \ell)$ is 1-to-1, and the
  inverse just define allow to reconstruct the data erased when going
  from $\hypermap$ to $\hat{\hypermap}$.
\end{proof}

\begin{remark}\label{rem:comb-hypermap-cw}
  In the case of the hypermap
  $(\hat{\hypermap}, l) = \Phi(\gamma, \theta, \sigma)$ with
  $\Delta\gamma_{0} = \emptyset$, the clockwise cyclic type of a black
  vertex $b$ corresponding to a cycle $\pi \in \Cyc(\theta)$ is the
  cyclic list with entries
  \begin{equation}\label{eq:cwtype-b}
    \tau = \left(\gamma\pi^{p}(j)\right)_{p \in I_{\pi}}\,, \text{ with } I_{\pi} = \{1 \leq p \leq \#\Supp \pi\colon \pi^{p}(j) \in  \Delta\gamma_{-}\},
  \end{equation}
  for $j \in \Supp \pi$. Otherwise said, it is the sublist of
  $(\gamma\pi^{p}(j))_{1 \leq p \leq \#\Supp \pi}$ obtained by keeping
  only the down steps in the original Motzkin bridge. The lower
  completion of $\tau$ is the cyclic list with entries
  \begin{equation*}
    c^{\downarrow}(\tau) = (\gamma\pi^{p}(j))_{1 \leq p \leq \#\Supp \pi}\,, \text{ for some }j \in \Supp \pi\,.
  \end{equation*}
\end{remark}

\subsection{Labelling the half-edges in the Bouttier-Fusy-Guitter bijection}
\label{sec:labell-he-BFG}

The bijection of Theorem \ref{thm:bfg} is between sets of hypermap and
maps that are not half-edge or edge-labelled. We now explain how the
edge labels of a well-labelled hypermap $(\hypermap, l)$ get
transported to a half-edge-labelling of a suitably labelled map
$(\map, \ell)$.

Fix a well-labelled hypermap $(\hypermap, l)$. In Construction
\ref{constr:bfg}, we add one edge for each corner of
$\hat{\hypermap}$. We now explain how to label the two half-edges
making up each of those edges, see Figure \ref{fig:label-mobile}.

In step \ref{item:edges-face-BDG} of Construction \ref{constr:bfg} applied to $\hat{\hypermap}$, we added
a half-edge $h_{c}$ for each corner $c$ in $\hat{\hypermap}$. The
corner $c$ is based at a white vertex $w$, and is delimited by two
edges: $e_{1}$ to $e_{2}$ in the clockwise direction around $w$. Let
$i$ be the label of $e_{2}$. In $\hat{\map}$, let $h_{c}'$ be the next half-edge around
$w$ after $h_{c}$ in the clockwise orientation . We label $h_{c}'$ by $i$

With the previous procedure, we labelled only half of the total number
of half-edges in $\hat{\map}$: exactly the half-edges following the
decreasing half-edges around white vertices.

\begin{figure}[ht]
  \centering
  \includegraphics[width=0.3\textwidth]{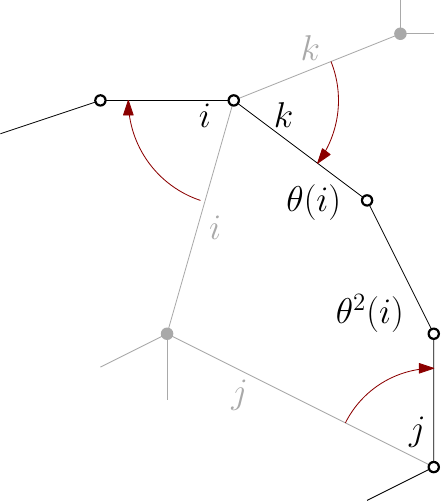}
  \caption{\label{fig:label-mobile} The labelling procedure. The vertices and edges of
    $\hat{\map}$ are in black and the ones of $\hat{\hypermap}$ are in
    grey.}
\end{figure}

To label the other half-edges we proceed as follows. Let $h$ be an
increasing half-edge along a face $f$ in $\hat{\map}$. Assume that it
is not the counterpart of an half-edge incident to a frustrated
vertex. We explore the face $f$ in the clockwise direction starting
from $h$, and stop once we encounter a decreasing half-edge $h'$. Let
$h_{1} = h, h_{2}, \ldots, h_{k}$ be the increasing half-edge we
encounter during this exploration. Let $i$ be the label of $h'$. We
label $h$ by $\theta^{k}(i)$.

Note that we do not label the counterparts of half-edges incident to
frustrated vertices in the resulting suitable map with frustrated
vertices. To finish the procedure, we remove the frustrated half-edges
and construct a map $\map$ from $\hat{\map}$.
\begin{constr}\label{constr:remove-frustrated}
  Consider a frustrated vertex $v$ in $\hat{\map}$. Let $h_{1}, h_{2}$
  be the two half-edges attached to $v$, and
  $\tilde{h}_{1}, \tilde{h}_{2}$ be two half-edges such that $h_{i}$
  and $\tilde{h}_{i}$ are counterparts of one another, for $i = 1, 2$.

  Assume that $i_{1}$ and $i_{2}$ are respectively the labels of
  $h_{1}$ and $h_{2}$. We remove $v, h_{1},$ and $h_{2}$. We connect
  $\tilde{h}_{1}$ and $\tilde{h}_{2}$ together. Finally, label
  $\tilde{h}_{1}$ by $i_{2}$ and $\tilde{h}_{2}$ by $i_{1}$. We denote
  the map we obtain after treating all frustrated vertices in this way
  by $\map$.
\end{constr}

\begin{figure}[ht]
  \centering
  \includegraphics[width=0.6\textwidth]{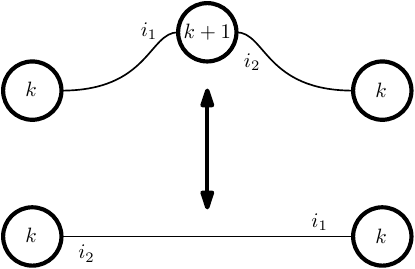}
  \caption{\label{fig:cut-frustrated} Constructions \ref{constr:remove-frustrated} and \ref{constr:duplicate-edge}.}
\end{figure}

The inverse construction is then as follows.
\begin{constr}\label{constr:duplicate-edge}
  Consider a suitably labelled map $(\map, \ell)$. For each frustrated
  edge between vertices $v_{1}$ and $v_{2}$, made of half-edges
  $\tilde{h}_{1}$ and $\tilde{h}_{2}$ (with
  $\tilde{h}_{1}, \tilde{h}_{2}$ respectively attached to
  $v_{1}, v_{2}$), we proceed as follows. We add a vertex $v$ with 2
  half-edges $h_{1}, h_{2}$ attached to it. We label $v$ by
  $\ell(v_{1}) + 1$. Assume that $i_{1}, i_{2}$ are the labels of
  $\tilde{h}_{1}, \tilde{h}_{2}$. We erase the labels of
  $\tilde{h}_{1}$ and $\tilde{h}_{2}$, and we label $h_{1}$ by $i_{2}$
  and $h_{2}$ by $i_{1}$. The resulting labelled map,
  $(\hat{\map}, \ell)$, is a suitably labelled map with no frustrated
  edge.
\end{constr}

\begin{lemma}\label{lem:hatmap-labelling}
  This labelling is well-defined and no two half-edges receive the
  same label. Furthermore, $\varphi_{\map} = \theta$.
\end{lemma}
\begin{proof}
  Each half-edge in $\hat{\map}$ is either decreasing or increasing.
  If it is decreasing, it corresponds to a unique corner of a white
  vertex (and hence a unique edge) in $\hat{\hypermap}$. We see this
  in Construction \ref{constr:inv-bfg}. If it is increasing, it is
  incident to a unique face in $\hat{\map}$ and the well-defined
  labelling of decreasing half-edges determines the labelling of the
  increasing half-edges.

  Let us show that $\varphi_{\hat{\map}} = \theta$. This will imply in
  particular that no two edges receive the same label. Let $f$ be a
  face of $\hat{\map}$ and $h, h'$ be two half-edges incident to $f$,
  consecutive when going around the face in the clockwise direction.
  Let $i$ and $j$ be their labels: $j = \varphi(i)$. If $h'$ is increasing,
  we immediately see that we have $j = \theta(i)$. If $h'$ is decreasing,
  let $h''$ be the first decreasing half-edge along $f$ encountered
  when exploring $f$ in the counterclockwise direction starting from
  $h$. Let $j'$ be its label. Considering Construction \ref{constr:inv-bfg}, we see
  that $\theta^{p}(j') = j$ for some $p \geq 1$ (for $i \in [p-1]$,
  $\theta^{i}(j')$ is the label of an increasing half-edge). We thus have
  $\theta^{p-1}(j') = i$ and $j = \theta(i)$.
\end{proof}

Using this labelling, we can describe the vertices of $\map$ in terms
of permutations. Let $v$ be a white vertex in $\map$. Let
$u_{1}, \ldots, u_{d}$ be the labels of the half-edges at $v$ which are
part of an edge connecting $v$ to a vertex of strictly smaller label,
encountered in this order when going in the clockwise direction around
$v$. Define
\begin{equation}\label{eq:def-pi-v}
  \pi_{v} = \cycle{u_{1}, \ldots, u_{d}}.
\end{equation}
Let $e$ be a frustrated edge in $\map$ made of half-edges labelled by
$i$ and $j$. Define
\begin{equation*}
  \pi_{e} = \cycle{i, j}.
\end{equation*}
\begin{lemma}\label{lem:degree-hyper-map}
  Let $e$ be a frustrated edge in $\map$, $w$ be
  the corresponding frustrated white vertex in $\hat{\hypermap}$, and
  $\pi$ be the cycle corresponding to $w$. Then, $\pi_{e} = \pi$.

  Let $v$ be a vertex in $\map$, $w$ be the corresponding white vertex
  in $\hat{\hypermap}$, and $\pi$ be the cycle corresponding to $w$.
  Then, $\pi_{v} = \pi$.

  In particular, if $w$ is of degree $d$, then $v$ has exactly $d$
  neighbors of label $l(v) - 1$.
\end{lemma}
\begin{proof}
  A frustrated edge $e$ in $\map$ is constructed from a frustrated
  vertex $v$ in $\hat{\map}$ of degree $2$. This latter vertex is
  constructed from a frustrated vertex $w$ in $\hat{\hypermap}$. If
  $w$ is represented by $\cycle{i, j}$, the two labels of the incident
  half-edges get transported by the labelling procedure: the two
  half-edges connected to $v$ are decreasing half-edges. This shows
  the first claim.

  Let $w$ be a white vertex in
  $\hat{\hypermap}$, $\pi= \cycle{u_{1}, \ldots, u_{k}}$ be the corresponding
  cycle in $\sigma_{-}$, and $v$ be the corresponding vertex in
  $\hat{\map}$. Consider the half-edges added at step \ref{item:insert-he} of
  Construction \ref{constr:bfg}. If $w$ is of degree $d$, $d$ such half-edges are
  attached to the corners of $w$. These half-edges are labelled by
  $u_{1}, \ldots, u_{k}$ in that order when going around the vertex in the
  clockwise direction. At step \ref{item:edges-face-BDG}, these half-edge are connected to
  vertices with label $l(v) - 1$. Now, consider a half-edge $h'$
  attached at step \ref{item:insert-he} to another white vertex $w'$. If at step \ref{item:edges-face-BDG},
  $h'$ gets connected to $w$, then $l(w') = l(w) + 1$. It may be that
  $w'$ is a frustrated vertex, and gets removed in Construction \ref{constr:remove-frustrated}:
  then through the half-edge $h'$, $w$ is connected to a vertex of
  label $l(w)$. Hence, through this procedure we created exactly $d$
  edges connecting $w$ to a vertex of degree $l(w) - 1 = l(v) - 1$,
  and the half-edges connected to $v$ that are part of these $d$ edges
  are labelled by $u_{1}, \ldots, u_{k}$ in that order.
\end{proof}

\begin{theorem}\label{thm:bijection-labelled-mobiles}
  Fix $n \geq 1$ and $\theta \in \Sym_{n}$. Define
  \begin{equation*}
    \Suitably(\theta) = \left\{ (\map, \ell) \in \Suitably_{n} \colon \map \text{ is half-edge-labelled by } [n],  \varphi_{\map} = \theta \right\}.
  \end{equation*}

  The previous construction
  gives a bijection
  \begin{equation*}
    \Psi\colon \Comb(\theta) \to \Suitably(\theta).
  \end{equation*}

  Furthermore, if $(\map, \ell) = \Psi(\gamma, \sigma)$,
  \begin{enumerate}
    \item\label{item:vertex-cycle} each vertex $v$ of of $\map$ that is not a local minimum
          corresponds to a cycle $\pi_{v}$ as defined by
          \eqref{eq:def-pi-v}, with $\pi_{v} \in \Cyc(\sigma_{-})$, and has
          label $\ell(v) = \gamma(\pi_{v})$;
    \item\label{item:frustrated-cycle} each frustrated edge $e$ of $\map$ corresponds to a cycle
          $\pi_{e} \in \Cyc{\sigma_{0}}$ of length 2;
    \item\label{item:product-cycle} $\sigma = \prod_{v}\pi_{v} \prod_{e}\pi_{e}$ where the products are on the
          vertices that are not local minima and on the frustrated
          edges.
  \end{enumerate}
\end{theorem}
\begin{proof}
  If we forget about the labelling of the half-edges, the map $\Psi$ is
  obtained by composing the bijection of Proposition
  \ref{prop:gamma-sigma-to-hypermap} and the one of Theorem
  \ref{thm:bfg}.

  Fix $(\gamma, \sigma) \in \Comb(\theta)$, $(\hat{\hypermap}, l) = \Phi(\gamma, \sigma)$,
  and $(\map, \ell) = \Psi(\gamma, \sigma)$. The labelling of the edges of
  $\hat{\hypermap}$ determines the labelling of the decreasing
  half-edges of $\map$. The unique determination of the other labels
  follows from the constraint that $\varphi_{\map} = \theta$. Lemma
  \ref{lem:hatmap-labelling} shows that with out choice of labelling
  we do have $\varphi_{\map} = \theta$. Conversely, from a suitable
  map $(\map, \ell)$ with $\varphi_{\map} = \theta$, we can recover
  the labels of the corresponding hypermap by erasing the labels of
  the increasing half-edges.

  The second part of the Theorem is a consequence of Lemma \ref{lem:degree-hyper-map}, and in
  the case of point \ref{item:frustrated-cycle} of Constructions \ref{constr:remove-frustrated} and \ref{constr:duplicate-edge} as well.
\end{proof}

\section{Combinatorial description of the cumulants}
\label{sec:combinatorial-cumulants}

We now re-express the cumulants \eqref{eq:cumulants-combin} in terms of
suitably labelled maps.

\subsection{Expression in terms of the distances and proof of Theorem \ref{thm:exp-cumulants}}
\label{sec:expr-terms-dist}

The cumulants can be rewritten in terms of sums over suitably labelled
maps. Indeed, Theorem \ref{thm:bijection-labelled-mobiles} allows us to replace the sum on Motzkin
paths and permutation in Proposition \ref{prop:N-exp-cumulants} with a sum on suitably
labelled maps. This will be the key element of the proof of Theorem
\ref{thm:exp-cumulants}.

Before rewriting the expansion of cumulants in terms of suitably
labelled maps, we reinterpret the terms
$e_{q} \left( \gamma(\pi); \pi \in \Cyc(\sigma_{-}) \right)$. We now show these terms
correspond to product of distances in a map. Consider a suitably
labelled map $\map$, and denote by $V_{\map}^{\min}$ the set of local
minima of $\map$ and $V_{\map}^{\star} = V_{\map} \setminus V_{\map}^{\min}$. For
a vertex $v \in V_{\map}$, we set
\begin{equation*}
  d_{v} = \min_{v^{*} \in V^{\min}_{\map}} \left( d(v^{*}, v) + \ell(v^{*}) \right),
\end{equation*}
where $d(u, u')$ is the graph distance between two vertices $u$ and
$u'$. By the second part of Theorem \ref{thm:bijection-labelled-mobiles}, if
$v \in V_{\map \setminus V_{\map}^{\min}}$ corresponds to a cycle $\pi$, then the
label of $v$ corresponds to $\gamma(\pi)$. On the other hand, the label of
$v$ is $d_{v}$, as explained in \cite[Remark 1]{bouttier_twopoint_2014}. The argument goes
as follow: consider any geodesic from $v$ to some
$v^{*} \in V^{\min}_{\map}$, the labels along the geodesic are
necessarily weakly decreasing (by steps of $0$ or $1$). There exists a
choice of $v^{*}$ and of a geodesic with strictly decreasing labels to
$v^{*}$. In that case the length of the geodesic is the distance
between $v$ and $v^{*}$ but also the difference of the labels of $v$
and $v^{*}$.

Hence, using Theorem \ref{thm:bijection-labelled-mobiles}, we can rewrite the sum in Proposition
\ref{prop:N-exp-cumulants} as
\begin{equation*}
  \sum_{\substack{\gamma\in\motz_{n, 0}(\theta(\bm{n}))\\\sigma \in \Sym_{\gamma}, |\sigma| = p\\ \Orbit( \theta(\bm{n}), \sigma ) = 1}} e_{q}\left(\gamma(x); c \in \Cyc(\sigma_{-})\right) = \sum_{\substack{\map \in \Suitably(\theta)\\\# V_{\map}^{\star} = n/2 - p}} e_{q} \left( d_{v}; v \in V_{\map}^{\star}\right).
\end{equation*}
We introduce the notation of average over sum of maps of a symmetric polynomial $f$ to be
\begin{equation*}
  \left\langle f \right\rangle_{\theta, p} = \sum_{\substack{\map \in \Suitably(\theta)\\\# V_{\map}^{\star} = n/2 - p}} f \left( d_{v}; v \in V_{\map}^{\star}\right).
\end{equation*}
This allows us to rewrite the expression for the cumulants in compact
form:
\begin{equation}\label{eq:cumulant-distance}
      \kappa_{l}(\bm{n}) = \sum_{p+q+r=n/2}\left( \frac{2}{\beta}\right)^{p} (-1)^{q}P_{r}(N)\left\langle e_{q}\right\rangle_{\theta, p}.
\end{equation}
This proves Theorem \ref{thm:exp-cumulants}. Note that the index $p$
gives non-zero contribution if and only if $p \geq l - 1$. Indeed,
otherwise the maps counted in $\langle e_{q} \rangle_{\theta, p}$ are
necessarily disconnected. This can be seen using Euler's formula (see
Remark \ref{rem:prop-suitably}).

Note that we also have by expanding $P_{r}$:
\begin{equation}\label{eq:cumulant-distance-Bernoulli}
  \kappa_{l}(\bm{n}) = \sum_{p+q+r+s = n/2}\left( \frac{2}{\beta} \right)^{p}\frac{(-1)^{q}B_{r}}{s+1}\binom{r+s}{r} N^{s+1} \left< e_{q} \right>_{\theta, p}.
\end{equation}

\begin{remark}\label{rem:prop-suitably}
  Notice that by Euler's formula:
  \begin{equation*}
    2 - 2g_{\map} = \left( \frac{n}{2} - p + \# V_{\map}^{\min} \right) - \frac{n}{2} + l = l+\# V_{\map}^{\min} - p.
  \end{equation*}
  Hence, when $p$ increases, either the number of minima increase, or
  the genus increases.
\end{remark}

\subsection{Analysis of the first two orders}
\label{sec:analysis-first-two}

We now turn to the first two orders of the cumulants computed in
Proposition \ref{prop:N-exp-cumulants}. This Section will prove part of Corollary
\ref{corol:leading-orders}.

The leading order is obtained by considering the term
$s = n/2 - l +1, p = l-1, q = r = 0$ in \eqref{eq:cumulant-distance-Bernoulli}. It gives
\begin{equation*}
  \kappa_{l}(\bm{n}) = \left( \frac{2}{\beta} \right)^{l-1}N^{n/2-l+2}\frac{\# \left\{ \map \in \Suitably(\theta(\bm{n})) \colon \# V_{\map}^{\star} = n/2 - l + 1 \right\}}{n/2-l+2} (1 + \order{1/N}).
\end{equation*}
Notice that by Remark \ref{rem:prop-suitably}, we are considering maps with
\begin{equation*}
  2 - 2g_{\map} = 1 + \# V_{\map}^{\min}.
\end{equation*}
Since $\# V_{\map}^{\min} \geq 1$, this equation is only satisfied when
$g_{\map} = 0$ and $\# V_{\map}^{\min} = 1$. In this case, suitably
labelled maps correspond exactly to pointed planar maps with face
profile $\theta(\bm{n})$. As $n/2-l+2$ is the total number of vertices in
the map, we get that
\begin{equation*}
  \kappa_{l}(\bm{n}) = \left( \frac{2}{\beta} \right)^{l-1}N^{n/2-l+2} \#\{\text{edge-labelled planar maps with face profile } \theta(\bm{n})\}(1 + \order{1/N}),
\end{equation*}
which is the first order of Corollary \ref{corol:leading-orders}. We have recovered for all
$\beta > 0$ the result of Abdesselam, Anderson, and Miller
\cite{abdesselam_tridiagonalized_2014}.

To treat the sub-leading order, we prove the following Proposition.
\begin{prop}\label{prop:sub-S2}
  Let $n \in \N^{*}$ and $\theta \in \Sym_{n}$. We define the set of
  suitably labelled map with two local minima:
  \begin{equation*}
    \Suitably_{2}(\theta) = \left\{ (\map, \ell) \in \Suitably(\theta)\colon \# V_{\map}^{\min} = 2 \right\}.
  \end{equation*}

  We have
  \begin{equation*}
    N^{l-2-n/2}\kappa_{l}(\bm{n}) = \left( \frac{2}{\beta} \right)^{l-1}\#\Maps_{0}(\theta(\bm{n})) + \left( \frac{2}{\beta} \right)^{l-1}\frac{1}{N}\left(\frac{2}{\beta} - 1\right)\frac{\#\Suitably_{2}(\theta(\bm{n}))}{n/2-l+1} + \order{\frac{1}{N^{2}}},
  \end{equation*}
  with $\Maps_{0}(\theta(\bm{n}))$ the set of half-edge-labelled
  planar maps with face profile $\theta(\bm{n})$.
\end{prop}
The sub-leading order of the cumulant is thus described by the suitably
labelled maps on the sphere with exactly two local minima. In Section
\ref{sec:many-betw-suit}, we give another description of this object
in terms of non-orientable maps on $\RP^{2}$. The full proof of
Corollary \ref{corol:leading-orders} will follow from Proposition
\ref{prop:sub-S2} and Theorem \ref{thm:bijection-1/2} proved in
Section \ref{sec:many-betw-suit}.

The proof is based on two mappings that we now introduce. We define
the set of suitably labelled maps with two local minima and $m+1$
global minima
\begin{equation*}
  \Suitably_{2}(\theta, m) = \left\{ (\map, \ell) \in \Suitably(\theta)\colon \# V_{\map}^{\min} = 2, \# \ell^{-1}({0}) = m+1 \right\}.
\end{equation*}
The integer $m$ is in $\{0, 1\}$. It is $0$ if we consider maps with
exactly one global minimum (vertex with label $0$), and $1$. We define
the sets of suitably labelled maps with one local minimum and a choice
of vertex with an additional label, which is either strictly positive,
or zero:
\begin{equation*}
  \begin{split}
    \Suitably_{1, +}(\theta) &= \left\{ (\map, \ell,  v, k)\colon (\map, \ell) \in \Suitably(\theta), v \in V_{\map}^{\star}, \# V_{\map}^{\min} = 1,  1 \leq k < d(v, V_{\map}^{\min}) \right\}\\
    \Suitably_{1, 0}(\theta) &= \left\{ (\map, \ell,  v)\colon (\map, \ell) \in \Suitably(\theta), v \in V_{\map}^{\star}, \# V_{\map}^{\min} = 1 \right\}.
  \end{split}
\end{equation*}

We construct a bijection
$\phi_{1} \colon \Suitably_{1, +}(\theta) \to \Suitably_{2}(\theta, 0)$ and a two-to-one
mapping $\phi_{2} \colon \Suitably_{1, 0}(\theta) \to \Suitably_{2}(\theta, 1)$. The
mappings are constructed by changing the label of the vertex with
additional label to make it a local minimum. The bijection is as
follows. Let $(\map, \ell, v, k) \in \Suitably_{1, +}(\theta)$. Denote the
unique local minimum of $(\map, \ell)$ by $v^{*}$. We define the
labelling function by $\ell'(v) = k$, $\ell'(v^{*}) = 0$, and by
\begin{equation*}
  \ell'(u) = \min_{u^{*} \in \{v^{*}, v\}} \left( \ell'(u^{*}) + d(u^{*}, u) \right)
\end{equation*}
for any other vertex $u$. The second mapping $\phi_{2}$ is constructed
similarly, by replacing $k$ with $0$.
\begin{lemma}\label{lem:bij-S2-S11}
  The labelled map $(\map, \ell')\coloneq \phi_{1}(\map, \ell, v, k)$ is in
  $\Suitably_{2}(\theta, 0)$ and $\phi_{1}$ is a bijection.

  The labelled map $(\map, \ell')\coloneq \phi_{2}(\map, \ell, v)$ is in
  $\Suitably_{2}(\theta, 1)$ and $\phi_{2}$ is two-to-one.
\end{lemma}
\begin{proof}
  We first show that the map $(\map, \ell')$ is suitably labelled. Let
  $u, u'$ be two adjacent vertices in $\map$. Let $u_{*}$ (resp.
  $u_{*}'$) be the vertex in $\{v^{*}, v\}$ closest to $u$ (resp.
  $u'$). Assume that $\ell'(u) \geq 2 + \ell'(u')$. As $u$ and $u'$ are
  adjacent, this means that $u_{*} \neq u'_{*}$. If $u_{*} = v$, then
  we have
  \begin{equation*}
    d(u, v^{*}) \geq d(u, v) + k \geq 2 + d(u', v^{*}),
  \end{equation*}
  which contradicts the triangular inequality as $d(u, u') = 1$. If
  $u'_{*} = v$, we similarly have
  \begin{equation*}
    d(u, v) + k \geq d(u, v^{*}) \geq 2 + d(u', v) + k,
  \end{equation*}
  a contradiction.

  Two minima of $(\map, \ell')$ are then $v^{*}$ and $v$ as any other
  vertex $u$ adjacent to $v$ has label
  \begin{equation*}
    \ell'(u) = \min(d(v^{*}, u), k + d(v, u)) \geq \min(d(v^{*}, v) - 1, k) \geq k.
  \end{equation*}
  They are the only minima. Indeed, for all vertex
  $u \notin \{v, v^{*}\}$, let $u^{*}$ be the vertex in $\{v, v^{*}\}$
  closest to $u$. Let
  $u_{0} = u, u_{1}, \ldots, u_{k-1}, u_{k} = u^{*}$ be the vertices
  on a geodesic from $u$ to $u^{*}$. We then have
  \begin{equation*}
    \ell'(u_{1}) \leq d(u^{*}, u_{1}) = d(u^{*}, u) - 1.
  \end{equation*}
  This proves the first claim for $\phi_{1}$ and $\phi_{2}$.

  Finally, we can invert $\phi_{1}$ as follows: given
  $(\map, \ell) \in \Suitably_{2}(\theta, 0)$ and $v^{*}$ the unique
  vertex with $\ell(v^{*}) = 0$ and $v'$ the other minimum, we can
  construct a new labelling function
  \begin{equation*}
    \ell'(v) = d(v^{*}, v).
  \end{equation*}
  This gives an element
  $(\map, \ell', v', \ell(v')) \in \Suitably_{1, +}(\theta)$.

  For $\phi_{2}$, given an element
  $(\map, \ell) \in \Suitably_{2}(\theta, 1)$, there are two vertices
  with label $0$. We can thus construct two distinct preimages.
\end{proof}

\begin{proof}[Proof of Proposition \ref{prop:sub-S2}]

Lemma \ref{lem:bij-S2-S11} implies
\begin{equation*}
  \# S_{2}(\theta, 0) = \# S_{1, +}(\theta) \text{ and } \# S_{2}(\theta, 1) = \frac{1}{2}\# S_{1, 0}(\theta).
\end{equation*}
The sub-leading order of $\kappa_{l}(\bm{n})$ is then
\begin{equation*}
  \begin{split}
    \sum_{u+q+r=1}&\left( \frac{2}{\beta} \right)^{l-1+u}\frac{(-1)^{q}B_{r}}{n/2-l+1}\binom{r+n/2-l}{r} N^{n/2-l-1} \left< e_{q} \right>_{\theta, l-1+u}\\
    &= \left( \frac{2}{\beta} \right)^{l-1}\frac{N^{n/2-l-1} }{n/2-l+1} \left(\frac{2}{\beta}\left< 1 \right>_{\theta, l} -  \left< e_{1} \right>_{\theta, l-1} + \frac{n/2-l+1}{2}\left< 1 \right>_{\theta, l-1}  \right)\\
    &= \left( \frac{2}{\beta} \right)^{l-1}\frac{N^{n/2-l-1} }{n/2-l+1} \left(\frac{2}{\beta} - 1\right)\left( \#\Suitably_{2}(\theta, 0) + \#\Suitably_{2}(\theta, 1) \right),
  \end{split}
\end{equation*}
as
\begin{equation*}
  \begin{split}
    \left< 1 \right>_{\theta, l} &= \#\Suitably_{2}(\theta, 0) + \#\Suitably_{2}(\theta, 1)\\
    \left< e_{1} \right>_{\theta, l-1} &= \#\Suitably_{1, 0}(\theta) + \#\Suitably_{1, +}(\theta)\\
    (n/2-l-1)\left< 1 \right>_{\theta, l-1} &= \#\Suitably_{1, 0}(\theta).
  \end{split}
\end{equation*}

\end{proof}

We may wonder if a similar proof holds beyond the first sub-leading
order. In theses cases, the mappings $\phi_{1}$ and $\phi_{2}$ must be
defined differently.

\section{A many-to-one map between suitably labelled maps and non-orientable maps on $\RP^{2}$}
\label{sec:many-betw-suit}

We now propose a way to interpret the suitably labelled maps appearing
in the sub-leading order of the expansion of the cumulants of the
$\beta$-ensemble as non-orientable maps on $\RP^{2}$. To do so, we will
interpret them as determining the lift of a map on a non-orientable
surface on its orientable double-covering. Note that we produce a
many-to-one mapping and not a bijection as we consider labelled
non-orientable maps on $\RP^{2}$. Fixing a face profile using a
permutation determines an orientation of the faces in the
non-orientable map, an information that is redundant in an orientable
map.

\begin{figure}[ht]
  \begin{subfigure}[b]{0.33\textwidth}
    \centering
    \includegraphics[width=.5\linewidth]{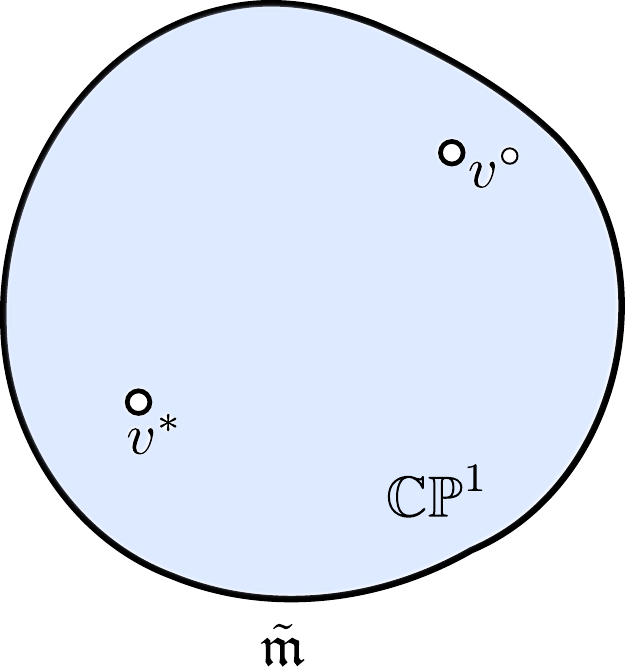}
    \caption{Starting map with two minima}
  \end{subfigure}
  \begin{subfigure}[b]{0.33\linewidth}
    \centering
    \includegraphics[width=.5\linewidth]{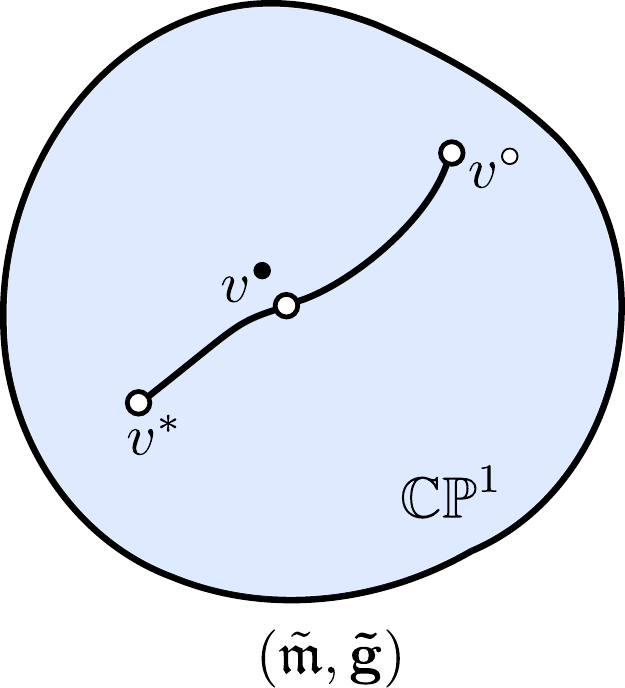}
    \caption{Choice of a path}
  \end{subfigure}
  \begin{subfigure}[b]{0.33\linewidth}
    \centering
    \includegraphics[width=.5\linewidth]{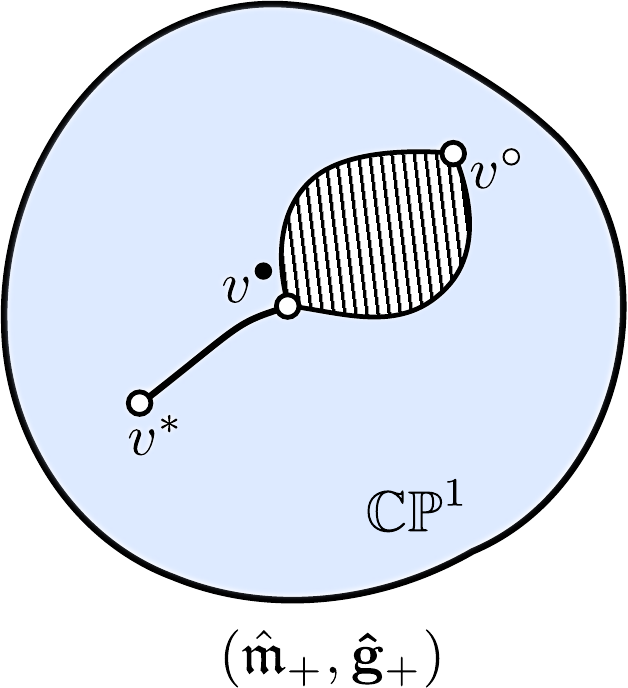}
    \caption{Opening of a face}
  \end{subfigure}
  \begin{subfigure}[b]{0.33\linewidth}
    \centering
    \includegraphics[width=.8\linewidth]{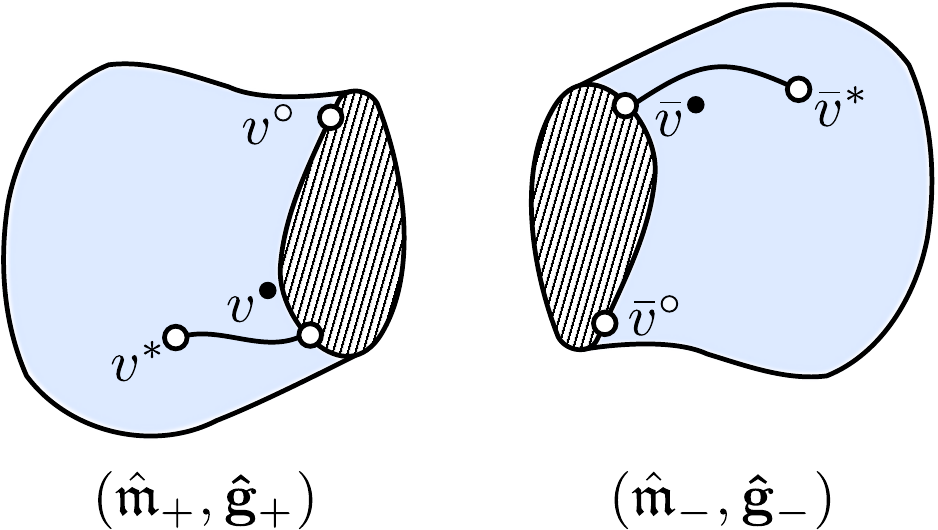}
    \caption{Adding the mirror map}
  \end{subfigure}
  \begin{subfigure}[b]{0.33\linewidth}
    \centering
    \includegraphics[width=.6\linewidth]{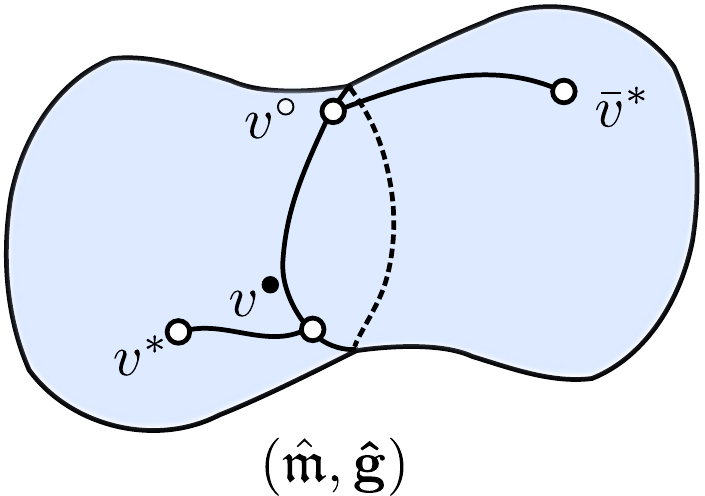}
    \caption{Gluing the two maps}
  \end{subfigure}
  \begin{subfigure}[b]{0.33\linewidth}
    \centering
    \includegraphics[width=.5\linewidth]{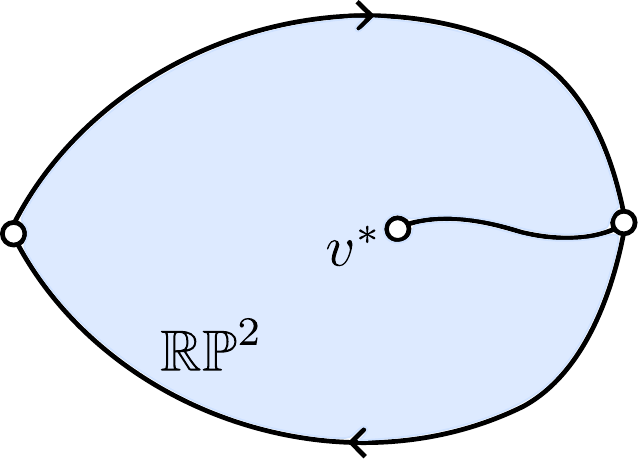}
    \caption{Projection to get a map on $\RP^{2}$}
  \end{subfigure}
  \caption{\label{fig:sequence-explain}The different steps of the construction of a map on $\RP^{2}$.}
\end{figure}

We now give an informal explanation of the construction, relying on
Figure \ref{fig:sequence-explain}. Rather than explaining the many-to-one mapping, giving a
suitably labelled map with two local minima from a map on $\RP^{2}$,
we give a right inverse to this mapping: from a suitably labelled map
with two local minima it gives a map on $\RP^{2}$. At this point,
several notions may not be clear, and will be explained in this
section. The construction is as follows. We start with a suitably
labelled map on the sphere with two minima $v^{*}$ and $v^{\circ}$ (step
(a)). We construct a path by choosing the leftmost geodesic from
$v^{*}$ to $v^{\circ}$ (step (b)). This choice of path determines a third
vertex $v^{\bullet}$, and a way to open a new ``boundary face'' in the map
(step (c)). We then take the mirror image of this map with boundary
(step (d)), and glue the two mirror images together along their
boundary face (step (e)). The resulting map can be seen naturally as a
map on the orientation covering of $\RP^{2}$, and can be projected to
give a map on $\RP^{2}$ (step (f)).

We shall recall first some facts concerning the orientation covering
of a surface in Section \ref{sec:orient-double-cover}. We then explain how we may encode maps on
possibly non-orientation coverings in Section \ref{sec:combin-non-orientable-maps}. We then define maps
on the orientation covering of a surface, maps equipped with an
involution, in Section \ref{sec:maps-orient-cover}. We then explain in details steps (c), (d),
and (e) in Section \ref{sec:cutt-gluing-suit}. Finally, we explain how paths are chosen in
Section \ref{sec:choosing-path-map} and give the full many-to-one mapping in Section \ref{sec:mapping-maps-rp2}.

\subsection{The orientation double covering}
\label{sec:orient-double-cover}

We recall a few facts on the orientable double covering of a
non-orientable surface. See for instance the book of Lee \cite{lee_introduction_2012} for more.

Consider a connected manifold $M$. We can construct an orientable
manifold $\hat{M}$ and a continuous surjective map $\pi\colon \hat{M} \to M$
such that $(\hat{M}, \pi)$ is a double covering of $M$. Informally, the
construction is as follows: there are two choices of orientation
locally around each point $P$ of $M$. These two choices determine two
sheets of the covering above a neighborhood of $P$. A surface is
orientable if and only if we can make a consistent global choice of
orientation. In that case, there are exactly two choices of global
orientation, and $\hat{M}$ is the union of two disconnected copies of
$M$: each copy corresponds to a choice of orientation. The manifold
$\hat{M}$ is equipped with an involution without fixed point, which
inverses the two sheets above $P$, or equivalently change the
orientation around $P$.

This double covering is called the \textbf{orientation covering} of $M$. The
connectedness property alluded to above is summarized in the following
Theorem.
\begin{theorem}[{\cite[Theorem 15.41]{lee_introduction_2012}}]
  Let $\pi \colon \hat{M}' \to M$ be the orientation covering of $M$.
  If $M$ is orientable, then $\hat{M}$ has two connected components
  and the restriction of $\pi$ to any of these component is a
  homeomorphism. If $M$ is not orientable, then $\hat{M}$ is connected.
\end{theorem}

The orientation covering is unique is the sense of the following
Theorem.
\begin{theorem}[See for instance {\cite[Theorem 15.42]{lee_introduction_2012}}]\label{thm:unique-orientable}
  Let $\pi' \colon \hat{M}' \to M$ be an orientable double covering of a
  non-orientable manifold $M$. Then, this covering is isomorphic to
  the orientation covering.
\end{theorem}
In the sequel, we consider the orientation covering of $\RP^{2}$. It
is topologically a sphere.

An important part of the mapping described in this section is that a
non-orientable map canonically defines a map on its orientation
covering. Let us now detail why it is so. Note that starting from now,
and until the end of the section, the notation $\hat{\cdot}$ (as in
$\hat{S}, \hat{\map}, \ldots$) denote objects related to some
orientation covering.
\begin{constr}\label{constr:lift-map}
  Let $\map$ be a non-orientable map. Consider a graph embedding
  $(\Gamma, S, \iota)$ in the class $\map$. The surface $S$ has a
  connected orientation covering $p\colon \hat{S} \to S$. We lift to
  $\hat{S}$ the image of $\Gamma$, $\iota(\Gamma)$. We obtain a graph
  embedding $(\hat{\Gamma}, \hat{S}, \hat{\iota})$ in $\hat{S}$ with
  twice the number of vertices, edges, and faces of $\map$. We thus
  define a map $\hat{\map}$ on the orientation covering of $S$. This
  map is well defined and does not depend on the particular choice of
  graph embedding $(\Gamma, S, \iota)$, since the orientation covering
  is unique up to isomorphism by Theorem \ref{thm:unique-orientable}.
\end{constr}

The orientation covering $\hat{S}$ is equipped with an
orientation-reversing involution without fixed point,
$\inv_{\hat{S}}$. This involution descends to an involution on the set
of vertices, edges and faces of $(\Gamma, S, \iota)$.

\subsection{Combinatorial description of non-orientable maps}
\label{sec:combin-non-orientable-maps}

We now describe a way to encode non-orientable maps as triples of
matchings (recall Definition \ref{def:matching}). The construction we now describe is
due to Tutte \cite{tutte_graph_1984} (see also \cite{godsil_algebraic_2001}). Starting from now, and until
the end of this Section, we abuse notation and define permutational
models that acts either on sets of labels or set of half-edges or
flags, as explained in Remark \ref{rem:act-on-H}.

\begin{definition}\label{def:flags-sides}
  Let $\map$ be a map, orientable or non-orientable. Let $h$ be an
  edge in $\map$. We can distinguish between two sides of $h$. We call
  a side of a half-edge a \textbf{flag}. We denote by $\Flag_{\map}$ the set of
  flags of $\map$. Let $\flag$ be a flag on a half-edge $h$. There is
  a unique face $f$ on this side $\flag$ of $h$. We say that $f$ is
  incident to $\flag.$
\end{definition}

Let $\map$ be a map with $n$ half-edges. To define a flag-labelling
function, we consider an extended set of labels $I$ of size $2n$.
A flag-labelling function is then a bijection
$\lambda\colon \Flag_{\map} \to I$. Let $\lambda$ be such a function.

We then define three matchings $\tau_{\map}, \rho_{\map}, \mu_{\map}$ as
follows. We define their action on the set of flags, but by Remark
\ref{rem:act-on-H}, using the labelling $\lambda$, they equivalently act on the index set
$I$. The cycles of $\tau_{\map}$ are $\cycle{\flag, \flag'}$ where
$\flag$ and $\flag'$ are the two flags associated to a same half-edge.
Consider an edge $e$. Let $(\flag_{1}, \flag_{1}')$ and
$(\flag_{2}, \flag_{2}')$ be two pairs of flags with
$\flag_{i}, \flag_{i}'$ associated to the same half-edge of $e$, and
$\flag_{1}$ and $\flag_{2}$ (resp. $\flag_{1}'$ and $\flag_{2}'$) on
the same side of $e$. Then
$\cycle{\flag_{1}, \flag_{2}}\cycle{\flag_{1}', \flag_{2}'}$ are two
cycles of $\rho_{\map}$. Finally, consider a corner of $\map$. This
corner is made of two flags, $\flag_{1}$ and $\flag_{2}$. Then,
$\cycle{\flag_{1}, \flag_{2}}$ is a cycle of $\mu_{\map}$.

\begin{figure}[ht]
  \centering
  \includegraphics[width=0.5\textwidth]{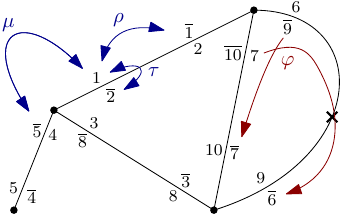}
  \caption{\label{fig:no-map} A non-orientable map with labelled flags. The edge with flags labelled $6, \bar{6}, 9, \bar{9}$ must be twisted to allow for an embedding in the plane. The blue and red arrows show the action of $\mu, \rho, \tau$ and $\varphi$. }
\end{figure}

In the sequel, we will often use the following an extended set of
labels. Let $n \in \N^{*}$. We define the set of ``barred'' integers
$\{\bar{1}, \bar{2}, \ldots, \bar{n}\} = [\bar{n} ]$. The extended set
of label is denoted by
$[n, \bar{n}] = \{1, \bar{1}, \ldots, n, \bar{n}\}$. For a subset
$ I \subset [n, \bar{n}]$, we define
\begin{equation*}
  \bar{I} = \left\{ \bar{i} \colon i \in I \cap [n] \right\} \cup \left\{ i \colon \bar{i} \in I \cap [\bar{n}] \right\}.
\end{equation*}

\begin{ex}\label{ex:no-map}
  The matchings describing the map displayed in Figure
  \ref{fig:no-map} are:
  \begin{equation*}
    \begin{split}
      \tau_{\map} &= \cycle{1, \bar{2}}\cycle{2, \bar{1}}\cycle{3, \bar{8}}\cycle{8, \bar{3}}\cycle{4, \bar{5}}\cycle{5, \bar{4}}\cycle{6, \bar{9}}\cycle{9, \bar{6}}\cycle{7, \overline{10}}\cycle{10, \bar{7}}\\
      \rho_{\map} &= \cycle{1, \bar{1}}\cycle{2, \bar{2}}\cycle{3, \bar{3}}\cycle{4, \bar{4}}\cycle{5, \bar{5}}\cycle{6, 9}\cycle{7, \bar{7}}\cycle{8, \bar{8}}\cycle{\bar{6}, \bar{9}}\cycle{10, \overline{10}}\\
      \mu_{\map} &= \cycle{1, \bar{5}}\cycle{2, \overline{10}}\cycle{3, \bar{2}}\cycle{4, \bar{8}}\cycle{5, \bar{4}}\cycle{6, \bar{1}}\cycle{7, \bar{9}}\cycle{8, \bar{6}}\cycle{9, \bar{7}}\cycle{10, \bar{3}}.
    \end{split}
  \end{equation*}
\end{ex}

As for the orientable maps, we can introduce the permutation
\begin{equation*}
  \varphi_{\map} = \rho_{\map}\mu_{\map},
\end{equation*}
which describes the faces of $\map$. Indeed, each face corresponds to
two cycles: each cycle correspond to an exploration of the face in a
different direction. For instance, in Figure \ref{fig:no-map}, the
same face is described by $\cycle{\bar{2}, \bar{3}, \overline{10}}$
and $\cycle{3, 2, 10}$.

\subsection{Maps on the orientation covering}
\label{sec:maps-orient-cover}

We now explain how an orientable half-edge labelled map that is
equipped with a involution without fixed point that reverses
orientation (in a sense to be defined) can be seen as being a map on
the orientation covering of some non-orientable surface. This will
give the inverse of Construction \ref{constr:lift-map}.
\begin{definition}\label{def:orientation-reversing}
  Let $I \subset \N^{*}$ be finite and $\hat{\map}$ be an orientable
  half-edge labelled map with labels in $I \sqcup \bar{I}$. Let $\inv$
  be a matching of $I \sqcup \bar{I}$. We say that $\inv$ is
  orientation-reversing if
  \begin{enumerate}
    \item\label{item:reversing}\begin{equation*} \varphi_{\map} \circ \inv = \inv \circ \varphi_{\map}^{-1} \text{
                                       and
                                       } \alpha_{\map} \circ \inv = \inv \circ \alpha_{\map}^{-1},
                                     \end{equation*}
                                     \item\label{item:no-fixed-cycle} For each triple of cycles $(\pi, \pi', \pi'') \in \Cyc(\varphi_{\hat{\map}}) \times \Cyc(\alpha_{\hat{\map}}) \times \Cyc(\sigma_{\hat{\map}})$, we have
                                     \begin{equation*}
                                       \left( \inv \circ \pi \circ \inv \right)^{-1} \neq \pi, \quad (\inv\circ \pi' \circ\inv)^{-1} \neq \pi', \, \text{ and } \quad (\inv\circ\pi''\circ\inv)^{-1} \neq \alpha_{\map}\circ\pi''\circ\alpha_{\map}.
                                     \end{equation*}
  \end{enumerate}
\end{definition}
\begin{remark}\label{rem:involution-no-fixed}
  Note that this definition implies that
  \begin{equation*}
    \begin{split}
      \left( \inv\circ\varphi_{\map}\circ\inv \right)^{-1} &= \varphi_{\map}\\
      \left( \inv\circ\alpha_{\map}\circ\inv \right)^{-1} &= \alpha_{\map}\\
      \left( \inv\circ\sigma_{\map}\circ\inv \right)^{-1} &= \alpha_{\map}\circ\sigma_{\map}\circ\alpha_{\map}.
    \end{split}
  \end{equation*}
  In particular, condition \ref{item:no-fixed-cycle} ensures that
  $\inv$ descends to an involution without fixed point on the sets of
  cycles of the three permutations $\sigma_{\map}, \alpha_{\map}$, and
  $\varphi_{\map}$.
\end{remark}

Using this notion, we can construct a bijection between flag-labelled
non-orientable maps and maps on their orientation covering. It allows
us to study non-orientable maps in terms of orientable maps.
\begin{prop}\label{prop:covering-no}
  Let $S$ be a compact non-orientable surface, and $\hat{S}$ be its
  orientation covering. Fix a permutation $\varphi$. The construction
  below gives a bijection between the half-edge labelled maps
  $\hat{\map}$ on $\hat{S}$ with face profile given by $\varphi$ and
  equipped with an orientation-reversing matching $\inv$, and the set
  of flag-labelled non-orientable maps $\map$ on $S$ with face profile
  $\varphi$ and $\rho_{\map} = \inv$.

  Furthermore, $\map$ is described by the matchings defined below in
  \eqref{eq:def-matchings}, and any labelling on the vertices of $\hat{\map}$ that is
  invariant by $\inv$ descends to a labelling of $\map$.
\end{prop}

For convenience, we assume that the set of labels is $[n, \bar{n}]$
for some $n \in \N^{*}$ and that $\varphi \in \Sym([n, \bar{n}])$. Fix
$\hat{\map}$ with half-edge labelling function $\hat{\lambda}$. Let $\inv$
be an orientation reversing matching. We first explain how the
involution $\inv$ induces an involution of the underlying surface.
Consider any embedded graph $(\hat{\Gamma}, \hat{S}, \hat{\iota})$ representing
$\hat{\map}$. We define $\inv_{\hat{S}}\colon \hat{S} \to \hat{S}$ a
continuous involution of the surface $\hat{S}$. We first define it on
the vertices, then on the edges, and finally on the faces of the
embedded graph $(\hat{\Gamma}, \hat{S}, \hat{\iota})$.

\begin{itemize}
  \item Let $u$ be a vertex of $\hat{\map}$. It corresponds to a point
        $P \in \hat{S}$ and to a cycle $\pi$ of $\sigma_{\hat{\map}}$.
        By Definition \ref{def:orientation-reversing}, the permutation
        $(\alpha_{\hat{\map}}\inv \pi \inv\alpha_{\hat{\map}})^{-1}$
        is a cycle of $\sigma_{\hat{\map}}$. Hence, it corresponds to
        a vertex $\bar{u} \neq u$ of $\hat{\map}$ and a point
        $P' \in \hat{S}$. We set $\inv_{\hat{S}}(P) = P'$. Similarly,
        $\inv_{\hat{S}}(P') = P$.
  \item Consider an edge of $\hat{\map}$, made of two half-edges
        $h_{1}$ and $h_{2}$. It corresponds to a path $e$ on
        $\hat{S}$. There is a unique edge made of the two half-edges
        $\bar{h}_{1}$ and $\bar{h}_{2}$ of labels
        $\inv(\hat{\lambda}(h_{1}))$ and $\inv(\hat{\lambda}(h_{2}))$
        respectively. This second edge corresponds to a path $e'$ on
        $\hat{S}$. We define $\inv_{\hat{S}}$ on $e$ to be any
        homeomorphism from $e$ to $e'$ that makes $\inv_{\hat{S}}$ a
        continuous involution on $\hat{\iota}(\hat{\Gamma})$.
  \item Finally, consider a face $f$ corresponding to a cycle $\pi$ of
        $\varphi_{\hat{\map}}$. There is a distinct face $\bar{f}$
        corresponding to the cycle $(\inv\pi\inv)^{-1}$ in
        $\varphi_{\hat{\map}}$. We define $\inv_{\hat{S}}$ on $f$ to
        be any homeomorphism from $f$ to $\bar{f}$ that makes
        $\inv_{\hat{S}}$ a continuous involution. Such an extension
        exists: by the Jordan-Schönflies theorem (see for instance
        \cite[Section 2.2]{mohar_graphs_2001}) we may construct a
        bijection extension of $\inv_{\hat{S}}$ on half of the
        faces, and define $\inv_{\hat{S}}$ on the other half of the
        faces so that it is an involution.
\end{itemize}

\begin{lemma}\label{lem:inv-reverse-orientation}
  The map $\inv_{\hat{S}}$ is a continuous involution without fixed points
  that reverses the orientation.
\end{lemma}
\begin{proof}
  We constructed $\inv_{\hat{S}}$ to be an involution. Remark
  \ref{rem:involution-no-fixed} implies that $\inv_{\hat{S}}$ has no fixed
  points.

  To see that $\inv_{\hat{S}}$ is orientation-reversing, we consider any
  face $f$ corresponding to a region $D$ (homeomorphic to a disc) in
  $\hat{S}$ and a cycle $\pi$ of $\varphi_{\map}$. The half-edges around $f$
  are $h_{1}, \ldots, h_{d}$ in the clockwise orientation. The disk
  $\inv_{\hat{S}}(D) \subset \hat{S}$ corresponds to a face $f'$ in $\hat{\map}$
  and to the cycle
  $\pi' = \left( \inv \pi \inv \right)^{-1}\inv \neq \pi$. Let $h_{i}'$ be
  the unique half-edge with label $\inv(\hat{\lambda}(h_{i}))$. We have
  \begin{equation*}
    \begin{split}
      \pi(\hat{\lambda}(h_{i})) &= \hat{\lambda}(h_{i+1})\\
      \pi'(\hat{\lambda}(h_{i}')) &= \inv \pi^{-1} \inv(\hat{\lambda}(h_{i}')) = \inv \pi^{-1}(\hat{\lambda}(h_{i})) = \inv(\hat{\lambda}(h_{i-1})) = \hat{\lambda}(h_{i-1}').
    \end{split}
  \end{equation*}
  Hence, $\inv_{\hat{S}}$ is orientation reversing in $D$. It is then
  orientation reversing globally.
\end{proof}

The quotient space $S = \hat{S}/\inv_{\hat{S}}$ is a surface, and the projection
$p_{\hat{S}}\colon \hat{S} \to S$ is a double covering of $S$. The associated deck
transformation is $\inv_{\hat{S}}$. By Lemma
\ref{lem:inv-reverse-orientation}, $\inv_{\hat{S}}$ is an
orientation-reversing involution. By Theorem \ref{thm:unique-orientable}, $\hat{S}$ is isomorphic to
the orientation covering of $S$.

The projection $p_{S}$ allows us to define a graph embedding in $S$.
Denote by $V_{\hat{\Gamma}}, H_{\hat{\Gamma}}$, and $E_{\hat{\Gamma}}$ the sets of vertices,
half-edges, and edges of $\hat{\Gamma}$. The new graph is $\Gamma$ with sets of
vertices, half-edges, and edges given by
\begin{equation*}
  \begin{split}
    V_{\Gamma} &= \left\{ \{u, \bar{u}\}\colon u \in V_{\hat{\Gamma}} \right\}\,,\\
    H_{\Gamma} &= \left\{ \{h, \bar{h}\} \colon h \in H_{\hat{\Gamma}} \right\}\,,\\
    E_{\Gamma} &= \left\{ \{h_{1}, h_{2}\}\colon h_{1} = \{g_{1}, \bar{g}_{1}\}, h_{2} = \{g_{2}, \bar{g}_{2}\}, g_{1}, g_{2} \in H_{\hat{\Gamma}} \right\}\,.
  \end{split}
\end{equation*}
The graph embedding $\iota\colon \Gamma \to S$ is obtained by taking the image by
$p_{\hat{S}}$ of $\hat{\iota}(\hat{\Gamma})$. Note that two vertices
$\hat{\iota}(u), \hat{\iota}(\bar{u})$ in $\hat{S}$ have the same image by
$p_{\hat{S}}$, and correspond to a unique vertex $\{u, \bar{u}\}$ in
$\Gamma$. Similarly, each half-edge, edge, or face of the graph embedded in
$S$ have two preimages in $\hat{S}$.

Now, by definition, any two choices of $(\hat{\Gamma}, \hat{S}, \hat{\iota})$
and $(\tilde{\Gamma}, \tilde{S}, \tilde{\iota})$ in the class of $\hat{\map}$
are isomorphic. If we denote by $\psi\colon \hat{S} \to \tilde{S}$ the
orientation preserving homeomorphism between the two surfaces we have
that $\inv_{\hat{S}}$ and $\psi^{-1}\inv_{\tilde{S}}\psi$ are homotopic (due
to the possibly different choices to map corresponding edges and faces
together). It implies that $\hat{S}/\inv_{\hat{S}}$ and
$\tilde{S}/\inv_{\tilde{S}}$ are homeomorphic, and that $p_{\hat{S}}$
and $p_{\tilde{S}}$ are isomorphic coverings. This shows that we can
define the map $\map$ to be the isomorphism class of $(\Gamma, S, \iota)$. This
construction is well-defined and does no depend on the choice of
$(\hat{\Gamma}, \hat{S}, \hat{\iota})$. This concludes the construction: we
have constructed from $\hat{\map}$ a new map $\map$, which is
non-orientable (as its orientation covering is connected). We shall
abuse notation and refer to $\hat{\map}$ as the orientation covering
of $\map$.

Let us now describe the permutational model associated to $\map$. We
start by defining a flag-labelling function $\lambda$. Consider a flag
$\flag$ in $\map$, i.e.\ a side of a half-edge $h$ (see Section \ref{sec:combin-non-orientable-maps}).
This flag has two preimages in the orientation covering $\hat{\map}$,
$\hat{\flag}_{1}$ and $\hat{\flag}_{2}$. In $\hat{\map}$, only one of
these two flags is on the left side of a half-edge $h$. This follows
from the fact that $\inv$ reverses the orientation. We label the flag
$\flag$ by $\hat{\lambda}(h)$, i.e.\ we set $\lambda(\flag) = \hat{\lambda}(h)$. The $2n$
flags are thus labelled by the elements of $[n, \bar{n}]$. We now
define three matchings $\tau, \mu$, and $\rho$. We set
\begin{equation}\label{eq:def-matchings}
  \tau_{\map} = \alpha_{\hat{\map}}\inv, \quad \rho_{\map} = \inv, \text{ and } \mu_{\map}= \tau\sigma_{\hat{\map}}^{-1} = \inv\varphi_{\hat{\map}}.
\end{equation}
Definition \ref{def:orientation-reversing} implies that these three permutations are indeed
matchings: the fact that they are involutions follows from the first
part of the definition, the fact that they do not have fixed point
follow from the second part.
\begin{lemma}\label{lem:matchings-describe-1/2}
  The triple of matchings $(\tau, \rho, \mu)$ describes $\map$.
\end{lemma}
\begin{proof}
  Let $\flag$ and $\flag'$ be two flags part of the same side of a
  same edge of $\map$. These two flags are incident to a face $f$.
  Denote by $h_{1}$ (respectively $h'_{1}$) the half-edge whose left
  side is a preimage of $\flag$ (resp. $\flag'$). Let $\flag_{1}$ and
  $\flag_{1}'$ the left sides of $h_{1}$ and $h'_{1}$, and $\flag_{2}$
  and $\flag_{2}'$ the right sides of the counterpart $h_{2}$ and
  $h_{2}'$ of $h_{1}$ and $h'_{1}$. The continuous involution sends
  the pair of flags $\flag_{1}, \flag_{2}$ to
  $\flag_{1}', \flag_{2}'$. Hence, it sends the label of $h$ to the
  label of $h'$, i.e.\ $\inv(\lambda(\flag)) = \lambda(\flag')$. On the other hand,
  we have by definition $\rho_{\map}(\lambda(\flag)) = \lambda(\flag')$. Thus,
  $\inv = \rho_{\map}$.

  Now, notice that $\hat{\lambda}(h_{2}') = \hat{\lambda}(\alpha_{\hat{\map}}(h_{2}))$
  is the label of the flag on the other side of $\flag$, i.e.\
  $\tau_{\map}(\lambda(\flag)) = \alpha_{\hat{\map}}\rho_{\map}(\lambda(\flag))$. This means that
  \begin{equation*}
    \tau_{\map} = \alpha_{\hat{\map}}\inv = \inv\alpha_{\hat{\map}}.
  \end{equation*}

  Finally, let $\flag_{3}$ the other flag in the corner $\flag$ is
  part of. Its preimage at the left of a half-edge is in the same corner as $h_{2}'$. Hence we have
  \begin{equation*}
    \lambda(\flag_{3}) = \mu(\lambda(\flag)) = \sigma_{\hat{\map}}\alpha_{\hat{\map}}\inv = \sigma_{\hat{\map}}\inv\alpha_{\hat{\map}} = \inv \alpha_{\map} \sigma_{\hat{\map}}^{-1} = \inv \varphi_{\hat{\map}}.
  \end{equation*}
\end{proof}

The faces of $\map$ are then described by the permutation
$\varphi_{\map} = \rho\mu$. We have
\begin{equation*}
 \varphi_{\map} = \inv \inv \varphi_{\hat{\map}} = \varphi_{\hat{\map}}.
\end{equation*}

\begin{proof}[Proof of Proposition \ref{prop:covering-no}]
  We explained how to construct $\map$ from $\hat{\map}$, let us now
  give the inverse construction.

  Let $\map$ be a non-orientable map on $S$ which is flag-labelled by
  $\lambda\colon\Flag_{\map} \to [n]$. We explained in Construction \ref{constr:lift-map} how to
  construct a map $\hat{\map}$ on the orientation covering $\hat{S}$
  of $S$. The orientation covering $\hat{S}$ is endowed with an
  orientation-reversing involution $\inv_{\hat{S}}$. The half-edges of
  $\hat{\map}$ are naturally labelled: the flag $\hat{\flag}$ at the
  left side of a half-edge $h$ has one image by $p$, $\flag$. We set
  $\hat{\lambda}(h) = \lambda(\flag)$. The continuous involution $\inv_{\hat{S}}$
  induces an involution $\inv$ on the labels of the half-edges of
  $\hat{\map}$ as follows. Let $h$ be a half-edge in $\hat{\map}$ and
  $h'$ its counterpart. The half-edge $h$ is incident to a face $f$
  and the image by $\inv_{\hat{S}}$ of $h'$, $h''$ is incident to
  $\inv_{\hat{S}}(f)$. We set $\inv(\hat{\lambda}(h)) = \hat{\lambda}(h'')$. This
  coincides with $\rho_{\map}$. We now explain why this involution must
  be an orientation-reversing matching. Indeed, if condition \ref{item:reversing} of
  Definition \ref{def:orientation-reversing} were not satisfied, $\inv_{\hat{S}}$ would not be
  orientation-reversing. If condition \ref{item:no-fixed-cycle} were not satisfied, there
  would be a face, edge, or vertex whose image by $\inv_{\hat{S}}$
  would be itself. By Brouwer fixed point Theorem, $\inv_{\hat{S}}$
  would have a fixed point: it contradicts the fact that
  $\inv_{\hat{S}}$ is a continuous involution without fixed point.

  If a labelling $\ell$ of the vertices of $\hat{\map}$ is invariant
  by $\inv$, then for any vertex $v$ in $\map$, its two preimages in
  $\hat{\map}$ have the same label and we can label $v$ in a
  well-defined way.
\end{proof}

\subsection{Cutting and gluing suitably labelled maps}
\label{sec:cutt-gluing-suit}

The goal of this Section is to define a mapping from the set of
suitably labelled maps with two local minima to the set of maps on
$\RP^{2}$. The procedure starts by choosing a path in a suitably
labelled map. We start by describing the procedure with a quite
general choice of path. We obtain an injective mapping between sets of
suitably labelled maps equipped with a curve.

We give some definition regarding what we mean by a path in a map.
\begin{definition}\label{def:left-right}
  A \textbf{path} of length $l \geq 1$ is a sequence of half-edges
  $\bm{g} = (g_{1}, \ldots, g_{2l})$, with $g_{2i-1}, g_{2i}$ the two
  half-edges of a same edge for all $i = 1, \ldots, l$, and
  $g_{2i}, g_{2i+1}$ incident to the same vertex for $i = 1, \ldots, l-1$.
  The length of the path is $\# \bm{g} = l$. The inverse of $\bm{g}$
  is the path $\bm{g^{-1}}$:
  \begin{equation*}
    \bm{g^{-1}} = (g_{2l}, \ldots, g_{1}).
  \end{equation*}

  We say a path is a \textbf{loop} if $\vertex(g_{1}) = \vertex(g_{2l})$. A
  path is \textbf{simple} if $\vertex(g_{2i}) \neq \vertex(g_{2j})$ and
  $\vertex(g_{2i-1}) \neq \vertex(g_{2j-1})$ for all $i \neq j$.

  The concatenation of two paths $\bm{g}$ and $\bm{h}$ such that
  $\vertex(g_{2\#\bm{g}}) = \vertex(h_{1})$ is
  $\bm{g} \sqcup \bm{h} = (g_{1}, \ldots, g_{2\# \bm{g}}, h_{1}, \ldots, h_{2\# \bm{h}})$.
  Finally, if $\bm{g}$ is simple, and if $u$ and $v$ are two vertices
  such that $u = \vertex(g_{2p+1})$ and $v= \vertex(g_{2q})$, we
  denote the subpath of $\bm{g}$ from $u$ to $v$ by
  \begin{equation*}
    \bm{g}\vert_{u \to v} = (g_{2p+1}, \ldots, g_{2q}).
  \end{equation*}
\end{definition}

An important assumption on some of the paths we consider is that they
are \emph{good} paths.
\begin{definition}\label{def:good-path}
  Let $\bm{g} = (g_{i})_{1 \leq i \leq 2l}$ be a simple path of length $l$
  in a suitably labelled map $(\map, \ell)$. Set
  $v_{0} = \vert{g_{1}}$ and $v_{i} = \vertex(g_{2i})$ for $i \in [l]$. We say
  $\bm{g}$ is a \textbf{good path} if $v_{0} \neq v_{l}$ and the function
  \begin{equation*}
    \ell_{\bm{g}} \colon \begin{cases}
      \{0, 1, \ldots, l\} &\to \N\\
      i &\mapsto \ell(v_{i})
    \end{cases}
  \end{equation*}
  has exactly two local minima, achieved at $i=0$ and $i=l$, with
  $\ell_{\bm{g}}(0) = \ell_{\bm{g}}(l)$, and either one local maximum, or
  two local maxima attained at consecutive values.

  We say a simple loop $\bm{g}$ is a \textbf{good loop} if it can be written as
  the concatenation $\bm{g} = \bm{g}_{1} \sqcup \bm{g}_{2}$ of two
  good paths $\bm{g}_{1}$ and $\bm{g_{2}}$ with
  $\ell_{\bm{g_{1}}} = \ell_{\bm{g_{2}}}$.
\end{definition}
The reason why we define good paths is that they have nice symmetry
properties that will be useful when gluing maps together along good
loops in the sequel.

\begin{ex}\label{ex:good-path}
  A simple path $\bm{g}$ with $\ell_{\bm{g}}$ given by
  \begin{equation*}
    \left( \ell_{\bm{g}}(i) \right)_{i=0, 1, \ldots, l} = (0, 1, 2, 2, 1, 0)
  \end{equation*}
  is good. However, if
  \begin{equation*}
    \left( \ell_{\bm{g}}(i) \right)_{i=0, 1, \ldots, l} = (0, 1, 2, 3, 2, 3, 2, 1, 0),
  \end{equation*}
  it is not good.
\end{ex}

We now describe several transformations that can be applied to a
suitably labelled map $(\map, \ell)$ with a distinguished good path
$\bm{g}$.

\subsubsection{Opening a slit}
\label{sec:opening-slit}

The first transformation corresponds to adding a new face to $\map$.
This new face will be seen as a boundary. In the process, we will add
new faces, with new labels that will be barred integer, for
convenience. The new half-edges we add will also be denoted with a
bar. We assume that if a half-edge $h$ is labelled by $i$, then
$\bar{h}$ is labelled by $\bar{i}$.

\begin{definition}\label{def:map-with-boundary}
  A face is simple if when going around it, each edge is encountered
  exactly once. A \textbf{map with boundary} $(\map, f)$ is a map $\map$ with
  a distinguished simple face $f$. A boundary of $\map$ is any choice
  of simple loop $\bm{g} = (g_{1}, \ldots, g_{2l})$ such that $f$ is
  incident to $g_{2j-1}$ for $j \in [l]$.

  If $\map$ is half-edge labelled with labels in a set $I$, we denote
  by $\varphi_{\map}(f)$ the cycle representing $f$ and we set
  \begin{equation*}
    \varphi_{\map \setminus f} = \varphi_{\map} \left( \varphi_{\map}(f) \right)^{-1}\vert_{I \setminus \Supp\varphi_{\map}(f)}.
  \end{equation*}
\end{definition}

Maps with boundaries can be seen as being embedded in other maps.
\begin{definition}\label{def:embedding-map}
  Let $(\map, f)$ be a half-edge labelled map with boundary and
  $\hat{\map}$ be a half-edge labelled map. We say that $\map$ is embedded in
  $\hat{\map}$ if
  \begin{equation*}
    \Cyc(\varphi_{\map \setminus f}) \subset \Cyc(\varphi_{\hat{\map}}) \text{ and }\Cyc(\alpha_{\map}) \subset \Cyc(\alpha_{\hat{\map}}).
  \end{equation*}
\end{definition}

We now describe the construction. An example is depicted in Figures
\ref{fig:good-path} and \ref{fig:open-slit}.
\begin{constr}\label{constr:open-slit}
  Consider a suitably labelled map $(\map, \bm{g})$ and $\bm{g}$, a good path of length $l \geq 1$.
  Let $v_{0} = \vertex(g_{1})$ and $v_{i} = \vertex(g_{2i})$ for
  $i \in [l]$. Each $\pi \in \Cyc(\sigma_{\map})$ corresponds to a vertex
  of $\map$. For each such cycle, we proceed as follows.
  \begin{itemize}
    \item If $\pi$ corresponds to none of the
          $v_{i}, i = 0, 1, \ldots, l$, we leave it unchanged.
    \item If $\pi$ corresponds to a vertex $v_{i}$ for $i \in [l-1]$, it can
          be written
          $\cycle{g_{2i}, u_{1}, \ldots, u_{d}, g_{2i+1}, v_{1}, \ldots, v_{d'}}$.
          We replace $\pi$ by
          \begin{equation*}
            \pi' = \cycle{u_{1}, \ldots, u_{d}, g_{2i+1}, \overline{g_{2l-2i+2}}}\cycle{v_{1}, \ldots, v_{d}, g_{2i}, \overline{g_{2l-2i-1}}}.
          \end{equation*}
    \item If $\pi$ corresponds to $v_{0}$, it can be written
          $\cycle{g_{1}, u_{1}, \ldots, u_{d}}$. We replace $\pi$ by
          \begin{equation*}
            \pi' = \cycle{g_{1}, u_{1}, \ldots, u_{d}, \overline{g_{2l-1}}}.
          \end{equation*}
    \item If $\pi$ corresponds to $v_{l}$, it can be written
          $\cycle{g_{2l}, v_{1}, \ldots, v_{d}}$. We replace $\pi$ by
          \begin{equation*}
            \pi' = \cycle{\overline{g_{2}}, v_{1}, \ldots, v_{d}, g_{2l}}.
          \end{equation*}
  \end{itemize}
  We obtain a new permutation $\sigma'$. We set
  \begin{equation*}
    \varphi' = \varphi_{\map} \tilde{\varphi}, \text{ with } \tilde{\varphi} = \cycle{\overline{g_{2l-1}}, \overline{g_{2l-3}}, \ldots, \overline{g_{1}}, \overline{g_{2}}, \overline{g_{4}}, \ldots, \overline{g_{2l}}},
  \end{equation*}
  and $\alpha' = (\varphi')^{-1}(\sigma')^{-1}$. The permutations $(\sigma', \alpha')$
  determine a half-edge labelled map $\map'$ with a marked face --
  represented by the cycle $\tilde{\varphi}$. The vertex-labelling $\ell$ of
  $\map$ induces a vertex-labelling $\ell'$ of $\map'$: each vertex $v'$
  of $\map'$ is constructed from a vertex $v$ of $\map$, we set
  $\ell'(v') \coloneq \ell(v)$.
\end{constr}

\begin{figure}
    \centering
    \begin{minipage}{0.48\textwidth}
        \centering
        \includegraphics[width=0.95\textwidth]{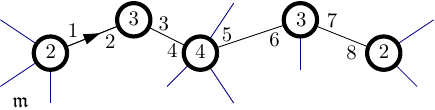} 
        \caption{\label{fig:good-path} Example of a good path.}
    \end{minipage}\hfill
    \begin{minipage}{0.48\textwidth}
        \centering
        \includegraphics[width=0.95\textwidth]{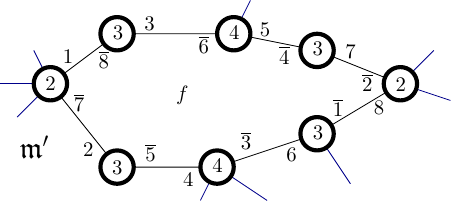} 
        \caption{\label{fig:open-slit} Opening of a new face along the good path.}
    \end{minipage}
\end{figure}

Note that $\alpha'$ is a matching by construction: the edges in
$\bm{g}$, which are represented by a cycle
$\cycle{g_{2i-1}, g_{2i}}$ of $\alpha_{\map}$, become the pair of
edges represented by
$\cycle{g_{2i-1}, \overline{g_{2l-2i+2}}}\cycle{g_{2i}, \overline{g_{2l-2i+1}}}$
for $i \in [l-1]$.

\begin{lemma}\label{lem:open-slit-suitable}
  The map $(\map', \ell')$ constructed in Construction
  \ref{constr:open-slit} is a suitably labelled map with boundary $f$.
  The boundary of $f$ is a good loop.
\end{lemma}
\begin{proof}
  We start by showing that the labelling is suitable. Let $v$ and $v'$
  two vertices of $\map'$ which are connected by an edge. These two
  vertices are constructed from vertices $\hat{v}$ and $\hat{v}'$ in
  $\map$. By construction, if $v$ and $v'$ are connected by an edge,
  so are $\hat{v}$ and $\hat{v}'$. The fact that $(\map, \ell)$ is a
  suitably labelled map thus implies that $(\map', \ell')$ is a
  suitably labelled map.

  In Construction \ref{constr:open-slit}, each edge of $\bm{g}$ gets
  duplicated. We can choose a boundary of $f$ to be a loop
  $\bm{g'} = (g_{i}')_{i \in [4l]}$ with $g'_{i} = g_{i}$ and
  $g'_{2l+i} = g_{2l-i}$ for $i \in [2l]$. As $\bm{g}$ is a good loop, we
  have $\ell_{\bm{g}}(i) = \ell_{\bm{g}}(l-i)$ for
  $i = 0, 1, \ldots, l$. It follows that $\bm{g'}$ is a good loop.
\end{proof}

\subsubsection{The mirror map}
\label{sec:mirror-map}

Given a half-edge labelled, suitably labelled map $(\map, \ell)$, we may
construct the ``mirror map'', obtained after changing the orientation of
all the vertices in $\map$.

This reversing of the orientation is encoded at the level of the
permutation by the following transformation.
\begin{definition}\label{def:reverse-permutation}
  Let $n \in \N^{*}$, $I \subset [n, \bar{n}]$, and
  $\sigma \in \Sym(I)$. Each cycle $\pi \in \Cyc(\sigma)$ can be
  written
  \begin{equation*}
    \pi = \cycle{u_{1}, \ldots, u_{d}}.
  \end{equation*}
  We set
  \begin{equation*}
    \bar{\pi} = \cycle{\overline{u_{d}}, \ldots, \overline{u_{1}}},
  \end{equation*}
  and
  \begin{equation*}
    \bar{\sigma} = \prod_{\pi \in \Cyc(\sigma)}\bar{\pi} \in \Sym(\bar{I}).
  \end{equation*}
\end{definition}
We have in particular that for two permutations $\sigma_{1}$ and $\sigma_{2}$,
\begin{equation*}
  \bar{\sigma}_{1}\bar{\sigma}_{2} = \overline{\sigma_{2}\sigma_{1}} \text{ and } \bar{\sigma}_{1}^{-1} = \overline{\sigma_{1}^{-1}}.
\end{equation*}

\begin{figure}[ht]
    \centering
    \begin{minipage}{0.48\textwidth}
        \centering
        \includegraphics[width=0.95\textwidth]{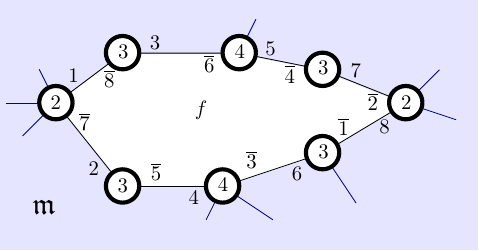}
        \caption*{(a)}
    \end{minipage}\hfill
    \begin{minipage}{0.48\textwidth}
        \centering
        \includegraphics[width=0.95\textwidth]{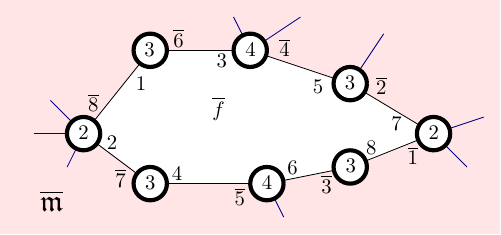}
        \caption*{(b)}
      \end{minipage}
      \caption{\label{fig:mirror-map}(a) The map of Figure \ref{fig:open-slit}, (b) its mirror map.}
\end{figure}

The map $\bar{\map}$ determined by
$(\bar{\alpha}_{\map}^{-1}\bar{\varphi}_{\map}^{-1}, \bar{\alpha}_{\map})$ is called
the mirror of $\map$. We denote by $\inv_{\map}$ the mapping from the
set of half-edges of $\map$ and the set of half-edges of $\bar{\map}$
sending a half-edge to its mirror image. In particular, if a half-edge
$h$ is labelled by $u$, the half-edge $\inv_{\map}(h)$ is labelled by
$\bar{u}$. We abuse notation and denote by $\inv_{\map}(v)$ and
$\inv_{\map}(f)$ the mirror image in $\map'$ of a vertex $v$ of a face
$f$ in $\map$. Finally, we define $\bar{\ell}$ by
\begin{equation*}
  \bar{\ell}(\inv_{\map}(v)) = \ell(v)\text{ for any vertex $v$ of } \map.
\end{equation*}
\begin{lemma}\label{lem:suitably-mirror}
  The constructed map $(\bar{\map}, \bar{\ell})$ is a suitably labelled map.
\end{lemma}
\begin{proof}
  The mirror construction descends to a bijection between the
  underlying graphs of $\map$ and $\bar{\map}$ that preserves the
  labelling of the vertices. The fact that $(\map, \ell)$ is a suitably
  labelled map implies the result.
\end{proof}

Note that the vertex permutation of $\map'$ is
\begin{equation}\label{eq:vertex-mirror}
  \sigma_{\bar{\map}} = \alpha_{\bar{\map}}^{-1}\varphi_{\bar{\map}}^{-1} = \bar{\alpha}_{\map}^{-1}\bar{\varphi}_{\map}^{-1} = \left( \bar{\varphi}_{\map}\bar{\alpha}_{\map} \right)^{-1} = \left( \overline{\alpha_{\map}\varphi_{\map}} \right)^{-1} = \bar{\alpha}_{\map}\left( \overline{\varphi_{\map}\alpha_{\map}} \right)^{-1}\bar{\alpha}_{\map}^{-1} = \bar{\alpha}_{\map}\bar{\sigma}_{\map}\bar{\alpha}_{\map}^{-1}.
\end{equation}

\subsubsection{Gluing along a face}
\label{sec:gluing-along-face}

The last construction we define is how to glue a map with boundary to
its mirror map, along their distinguished faces. We use the two
previous constructions. Fix a suitably labelled map
$(\map_{0}, \ell_{0})$ and a good path $\bm{g_{0}}$ in $\map_{0}$. We use
Construction \ref{constr:open-slit} to obtain a suitably labelled map $(\map, \ell)$ with
boundary face $f$. Let $(\bar{\map}, \bar{\ell})$ be the mirror map of
$(\map, \ell)$. It is a map with boundary face $\inv_{\map}(f)$. Let
$\bm{g}$ be a choice of boundary of $f$ in $\map$ such that
$\bm{g} = (g_{i})_{i\in[4l]}$ is a good loop. There is a canonical way to
glue $\map$ and $\bar{\map}$ along $f$. The labels of the half-edges
of the boundary of $\map$ and $\bar{\map}$ were constructed to be the
same. There is a natural way to identify an edge on the boundary of
$\map$ to an edge on the boundary of $\bar{\map}$.

We define the boundary permutation
\begin{equation*}
  \alpha_{\partial f} = \prod_{i=1}^{l}\cycle{g_{2i-1}, \overline{g_{2l-2i+2}}}\cycle{g_{2i}, \overline{g_{2l-2i+1}}}.
\end{equation*}

We then notice that the labels of the half-edges that are not on the
boundary are integers in $\map$ and barred integers in $\bar{\map}$.
This means that $\varphi_{\map}$ and $\varphi_{\bar{\map}}$ have
disjoint support, and that two cycles in
$\Cyc(\alpha_{\map}) \setminus \Cyc(\partial_{\map})$ and
$\Cyc(\alpha_{\bar{\map}}) \setminus \Cyc(\partial_{\map})$ have
disjoint support as well. Finally, we have
\begin{equation*}
  \Cyc(\alpha_{\map}) \cap \Cyc(\alpha_{\bar{\map}}) = \Cyc(\alpha_{\partial f}).
\end{equation*}

\begin{figure}[ht]
  \centering
  \includegraphics[width=0.6\textwidth]{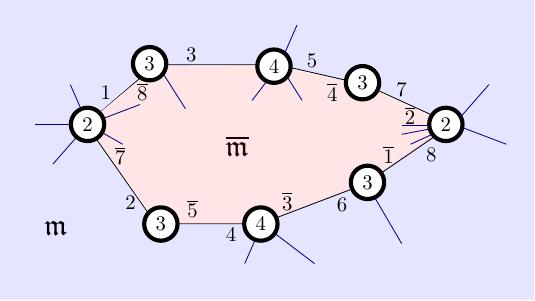}
  \caption{\label{fig:glue-mirror} Map obtained by gluing a map to its mirror map.}
\end{figure}

We thus define the permutations
\begin{equation*}
  \begin{split}
    \varphi_{\map\sqcup\bar{\map}} &= \varphi_{\map}\varphi_{\bar{\map}}\\
    \alpha_{\map\sqcup\bar{\map}} &= \alpha_{\map}\alpha_{\bar{\map}}\alpha_{\partial f}\\
    \sigma_{\map\sqcup\bar{\map}} &= \alpha_{\map\sqcup\bar{\map}}^{-1}\varphi_{\map\sqcup\bar{\map}}^{-1},
  \end{split}
\end{equation*}
The pair of permutations
$(\sigma_{\map\sqcup\bar{\map}}, \alpha_{\map\sqcup\bar{\map}})$ defines a map.
By construction, the maps with boundary $(\map, f)$ and
$(\bar{\map}, \inv_{\map}(f))$ are embedded in $\map\sqcup\bar{\map}$.
Furthermore, every face (and hence vertex) of $\map\sqcup\bar{\map}$
can be seen as being part of either $\map$ or $\bar{\map}$ (or both,
in the case of vertices). Thus, for every vertex $v$ of
$\map\sqcup \bar{\map}$ we set
\begin{equation*}
  \ell_{\map\sqcup\bar{\map}}(v) = \begin{cases}
    \ell(v) &\text{ if $v$ is part of } \map\\
    \bar{\ell}(v) &\text{ if $v$ is part of } \bar{\map}.
  \end{cases}
\end{equation*}

Any boundaries of $\map$ and $\bar{\map}$ are made of the same
half-edges (possibly in a different cyclic order). They form a loop
$\bm{g}_{\map\sqcup\bar{\map}}$. It can be chosen to be a good loop
since the boundary of $\map$ is a good loop by Lemma
\ref{lem:open-slit-suitable}.

\begin{lemma}\label{lem:glue-planar}
  The map $(\map\sqcup\bar{\map}, \ell_{\map\sqcup\bar{\map}})$ is a
  suitably labelled map. It is orientable and its genus is twice the
  genus of $\map$.
\end{lemma}
\begin{proof}
  Each edge of $\map\sqcup\bar{\map}$ can be seen as being part of
  either $\map$ or $\bar{\map}$ (or both). Hence, if an edge $e$
  between vertices $v$ and $v'$ can be seen as being part of, say,
  $\map$, we have
  \begin{equation*}
    \left| \ell_{\map\sqcup\bar{\map}}(v) - \ell_{\map\sqcup\bar{\map}}(v') \right| = \left| \ell(v) - \ell(v') \right| \leq 1.
  \end{equation*}
  Furthermore, the minimum label of a vertex in $\map\sqcup\bar{\map}$ is $0$.

  The vertices and edges that are not
  part of this loop are part of exactly one of the embedded maps
  $\map$ and $\bar{\map}$. The vertices and edges part of this loop
  are part of both $\map$ and $\bar{\map}$. As it is a simple loop,
  the number of vertices and edges that are part of
  $\bm{g}_{\map\sqcup\bar{\map}}$. Considering that $\map$ and
  $\bar{\map}$ have the same genus, Euler formula implies that the
  genus of $\map\sqcup\bar{\map}$ is twice the genus of $\map$.
\end{proof}

The mirror map $\inv_{\map}$ allows us to define an involution $\inv$
on the set of labels of the half-edges of $\map\sqcup\bar{\map}$. It
is defined as follows. Let $h$ be a half-edge in
$\map\sqcup\bar{\map}$ with label $i = \lambda_{\map\sqcup\bar{\map}}(h)$. We set
\begin{equation*}
  \inv(i) = \begin{cases}
    \lambda_{\map\sqcup\bar{\map}}(\inv_{\map}(h)) &\text{ if $h$ is in }\map\\
    \lambda_{\map\sqcup\bar{\map}}(\inv_{\map}^{-1}(h)) &\text{ if $h$ is in }\bar{\map}\\
  \end{cases} = \bar{i}.
\end{equation*}

\begin{lemma}\label{lem:prop-inv}
  The mapping $\inv$ is an orientation-reversing matching (in the
  sense of Definition \ref{def:orientation-reversing}).
\end{lemma}
\begin{proof}
  By definition of $\bar{\alpha}_{\map}$ (recall Definition
  \ref{def:reverse-permutation}), we have
  \begin{equation*}
    \inv\varphi_{\map}\inv = \varphi_{\bar{\map}}^{-1} \quad \text{ and } \quad \inv \alpha_{\map} \inv = \bar{\alpha}_{\map}^{-1} = \alpha_{\bar{\map}}.
  \end{equation*}
  We also have $\inv \alpha_{\partial f} \inv = \alpha_{\partial f}$. This gives
  \begin{equation*}
    \inv\varphi_{\map\sqcup\bar{\map}} = \varphi_{\map\sqcup\bar{\map}}^{-1}\inv \quad \text{ and } \quad \inv \alpha_{\map\sqcup\bar{\map}}^{-1} = \alpha_{\map\sqcup\bar{\map}} \inv.
  \end{equation*}
  Using this, we have
  \begin{equation*}
    \inv \sigma_{\map\sqcup\bar{\map}}^{-1}\inv = \inv \varphi_{\map\sqcup\bar{\map}}\alpha_{\map\sqcup\bar{\map}} \inv = \varphi_{\map\sqcup\bar{\map}}^{-1}\alpha_{\map\sqcup\bar{\map}} = \alpha_{\map\sqcup\bar{\map}} \sigma_{\map\sqcup\bar{\map}}\alpha_{\map\sqcup\bar{\map}}.
  \end{equation*}

  By construction, the involution sends the label of a half-edge $h$
  to the label of another half-edge $h'$ which is neither incident to
  the same vertex, nor incident to the same face, nor part of the same
  edge. It is easy to see for the faces and the edges: the conjugation
  by $\inv$ replace all the elements of a cycle by their barred
  versions. The cycle of the faces have support in either the integers
  or the barred integers, it is also the case for the edges that are
  not on $\bm{g}_{\map\sqcup\bar{\map}}$. We can then check that no
  cycle $\pi \in \alpha_{\partial f}$ satisfies $\inv \pi \inv = \pi$. For the
  vertices, assume that there exists a cycle
  $\pi \in \Cyc(\sigma_{\map\sqcup\bar{\map}})$ such that
  $(\alpha_{\map\sqcup\bar{\map}}\inv\sigma_{\map\sqcup\bar{\map}}\alpha_{\map\sqcup\bar{\map}})^{-1} = \pi$.
  Necessarily, the vertex corresponding to this cycle is on
  $\bm{g}_{\map\sqcup\bar{\map}}$. If it is not the case the support
  of $\pi$ is in the integers or the barred integers, without the
  elements of the support of $\alpha_{\partial f}$. We can then proceed as
  for the faces. Let $i$ be the only element in
  $\Supp \pi \cap \Supp \alpha_{\partial f} \cap \N^{*}$. We have that
  $\pi^{-1}(i) \neq i$ (\(\pi\) is incident to at least two edges) and
  thus
  $\inv \alpha_{\map\sqcup\bar{\map}}(i) = \inv \alpha_{\partial f}(i) \in \Supp \pi\cap \Supp \alpha_{\partial f} \cap \N^{*}$.
  However, this set is the singleton $\left\{ i \right\}$ and
  \begin{equation*}
    \inv \alpha_{\map\sqcup\bar{\map}}(i) = i.
  \end{equation*}
  The permutation $\inv\alpha_{\map\sqcup\bar{\map}}$ has no fixed point, we have reached a contradiction.
\end{proof}

Let us sum up what has been achieved so far. Starting from a suitably
labelled map $(\tilde{\map}, \ell)$ an a good path $\bm{\tilde{g}}$,
we produced a suitably labelled map with boundary
$(\hat{\map}_{+}, \ell_{+})$. We glue this map to its mirror image
$\hat{\map}_{-}$ to obtain a new suitably labelled map
$(\hat{\map}, \hat{\ell})$ in which both $\hat{\map}_{+}$ and
$\hat{\map}_{-}$ are embedded. The map $(\hat{\map}, \hat{\ell})$ is equipped with a
good loop $\bm{\hat{g}}$. Finally, thanks to Lemma
\ref{lem:prop-inv}, we can use the construction of Section
\ref{sec:maps-orient-cover} to see $\hat{\map}$ as being a map on the
orientation covering of a non-orientable map $\map$.

\begin{prop}\label{prop:bij-first-constr}
  The mapping just described, which associates, to a suitably labelled
  map $(\tilde{\map}, \tilde{\ell})$ equipped with a good path
  $\bm{\tilde{g}}$, the suitably labelled map $(\hat{\map}, \hat{\ell})$,
  is injective.
\end{prop}
\begin{proof}
  This follows from the fact that there is a left inverse to this
  mapping which we now describe. From $(\hat{\map}, \hat{\ell})$ we can
  recover the embedded map $(\hat{\map}_{+}, \hat{\ell}_{+})$: it is the
  unique embedded map whose non-boundary faces are exactly the ones
  with labels in the integers. Then, we glue the edges along the
  boundary face together. We glue together the two edges incident to
  each vertex of minimal label on the boundary. We do this repeatedly
  until there is no boundary face. The labels erased during this
  procedure are barred integers that where added in the first step
  (when opening the slit)
\end{proof}

The construction above depended on a choice of good path
$\bm{\tilde{g}}$ in $\tilde{\map}$. We now explain how to choose it in
a canonical way.

\subsection{Choosing a path in a map}
\label{sec:choosing-path-map}

We now explain how to choose a path $\bm{\tilde{g}}$ in a suitably
labelled map $(\tilde{\map}, \tilde{\ell})$, and characterize the image
of this path in the glued map $(\hat{\map}, \hat{\ell})$. We will show
that when considering a map on an orientation covering, there is a
canonical choice of loop, which we call equilibrium loop. Starting
from this Section, we assume that $\tilde{\map}$ is planar (and thus,
so is $\hat{\map}$).

\subsubsection{Local roots and leftmost paths}
\label{sec:local-roots-leftmost}

\begin{definition}\label{def:root}
  Let $\map$ be a suitably labelled map with labelled half-edges. The
  half-edge $h^{*}$ with minimal label among those attached to a
  vertex of label $0$ is said to be the \textbf{root}. The vertex $v^{*}$ it is
  attached to is the \textbf{root vertex}.

  A \textbf{local root} at a vertex $v$ is the choice of a half-edge incident
  to $v$.
\end{definition}

The notion of a local root allows us to define an ordering of the
half-edges at a vertex.
\begin{definition}\label{def:order-he}
  Let $v$ be a vertex with a local root $h$. Let
  $h = h_{1}, h_{2}, \ldots, h_{d}$ be the half-edges around $v$ in
  the clockwise order. We say that $h_{i}$ is to the left of
  $h_{j}$ if $i < j$.
\end{definition}

This ordering of the half-edges defines an ordering of the paths
starting at a vertex equipped with a local root.
\begin{definition}\label{def:left-right}
  Consider two paths $\bm{g} = (g_{i})_{1 \leq i \leq 2l}$ and
  $\bm{g'} = (g'_{i})_{1 \leq i \leq 2l'}$ with same starting vertex
  $v = \vertex(g_{1}) = \vertex(g'_{1})$. Assume that $v$ is equipped
  with a local root $h$. By convention, set $g_{0} = g'_{0} = h$. If
  there exists $i$, the first index such that $g_{2i+1} \neq g'_{2i+1}$
  then taking $g_{2i} = g_{2i}'$ to be the local root at
  $v' = \vertex(g_{2i}) = \vertex(g'_{2i})$, we say that $\bm{g}$ is
  at the left of $\bm{g'}$ if $g_{2i+1}$ is to the left of
  $g'_{2i+1}$. If there are no such $i$, then the shortest of the two
  paths is said to be to the left of the other.
\end{definition}
Note that this ordering of the paths defines a total order of the
paths started at a locally rooted map.

\begin{definition}\label{def:geodesic}
  A geodesic between two vertices $v$ and $v'$ is a path of shortest
  length (for the graph distance) between $v$ and $v'$.
\end{definition}

\begin{figure}[ht]
  \centering
  \includegraphics[width=0.25\textwidth]{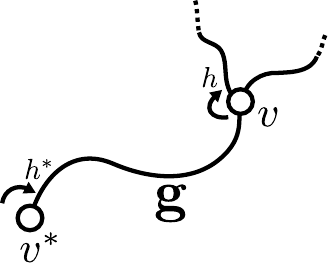}
  \caption{\label{fig:choose-local-root} Choice of a local root as in Construction \ref{constr:local-rooting}.}
\end{figure}

\begin{constr}\label{constr:local-rooting}
  Consider a map $\map$ with labelled half-edges. Let us explain how
  the root $h^{*}$ at the root vertex $v^{*}$ induces a choice of
  local root for each vertex $v$ in $\map$. Among the geodesics from
  $v^{*}$ to $v$, there is a leftmost geodesic
  $\bm{g} = (g_{i})_{1 \leq i \leq 2d}$. We choose the local root at $v$ to
  be the half-edge following $g_{2d}$ in the clockwise orientation
  around $v$. See Figure \ref{fig:choose-local-root}
\end{constr}

Furthermore, the (global) root allows to order the vertices.
\begin{definition}\label{def:left-vertex}
  Let $v^{*}$ be the root vertex, and $v, v'$ two distinct vertices.
  Let $\bm{g}$ and $\bm{g'}$ be respectively the leftmost geodesic
  from $v^{*}$ to $v$ and from $v^{*}$ to $v'$. We say that $v$ is to
  the left of $v'$ if $\bm{g}$ is to the left of $\bm{g'}$.
\end{definition}

\subsubsection{Leftmost good geodesics}
\label{sec:leftm-good-geod}

Consider a planar map $(\tilde{\map}, \tilde{\ell}) \in \Suitably_{n}$ with
two local minima, $v^{*}$ and $v^{\circ}$. Assume that $v^{*}$ is
the root vertex, and thus $0 = \tilde{\ell}(v^{*}) \leq \tilde{\ell}(v^{\circ})$.
We now distinguish a path $\bm{\tilde{g}}$ in $\tilde{\map}$. To do so
we introduce another notion.
\begin{definition}\label{def:good-path}
  A \textbf{good geodesic} starting from a vertex $v$ is a path from $v$ to a
  distinct vertex $v'$ with $\tilde{\ell}(v) = \tilde{\ell}(v')$, of minimal
  length among the paths from $v$ to a distinct vertex with label
  $\tilde{\ell}(v)$.
\end{definition}
%
We define $\bm{\tilde{h}}$ to be the leftmost geodesic from $v^{*}$ to
$v^{\circ}$, with respect to the ordering given by the global root
at $v^{*}$. As $v^{\circ}$ is a local minimum, there is a unique
vertex $v^{\bullet} \neq v^{\circ}$ encountered by $\bm{\tilde{h}}$ such that
$\tilde{\ell}(v^{\circ}) = \tilde{\ell}(v^{\bullet})$. We may have
$v^{\bullet} = v^{*}$ if $\tilde{\ell}(v^{\circ}) = 0$. We denote by
$\bm{\tilde{g}}$ the subpath of $\bm{\tilde{h}}$ from $v^{\bullet}$ to
$v^{\circ}$.

\begin{lemma}\label{lem:chosen-good-g}
  The path $\bm{\tilde{g}}$ -- up to reorienting it from $v^{\circ}$ to
  $v^{\bullet}$ -- is a good path and a good geodesic from $v^{\circ}$.
\end{lemma}
\begin{proof}
  The $\tilde{\ell}(\tilde{g}_{2i})$ of $\tilde{g}_{2i}$ is
  $\min(\tilde{\ell}(v^{\circ}) + i, \tilde{\ell}(v^{\circ}) + \# \bm{\tilde{g}} - i)$
  as $\bm{\tilde{g}}$ is a geodesic. This expression immediately shows
  that $\bm{\tilde{g}}$ is a good path.

  Assume that there is a vertex $u'$ with label $\ell(u') = \ell(v^{\circ})$ and a path $\bm{h'}$ between
  $v^{\circ}$ and $u'$ such that the length of $\bm{h'}$ is strictly
  smaller than $\bm{\tilde{g}}$. As $u'$ is either $v^{*}$ or not a
  local minimum, there is a path of length
  $\tilde{\ell}(u') = \tilde{\ell}(v^{\circ})$ between $v^{*}$ and $u'$. We
  can thus construct a path strictly shorter than $\bm{\tilde{h}}$
  between $v^{*}$ and $v^{\circ}$. This contradicts the fact that
  $\bm{\tilde{h}}$ is a geodesic.
\end{proof}

\subsubsection{Equilibrium loops}
\label{sec:equilibrium-loops}

Fix a suitably labelled map $(\hat{\map}, \hat{\ell})$, equipped with an
orientation-reversing matching $\inv$ that preserves the labels, i.e.\
$\hat{\ell}\circ \inv = \hat{\ell}$. In this section, we explain how to
choose in a unique way a good loop in $\hat{\map}$. The construction
is depicted on Figure \ref{fig:constr-eq-loop}. We denote by $v^{*}$
the root of $\hat{\map}$ and by $\bar{v}^{*} = \inv(v^{*})$ its image
by the involution.

\begin{constr}\label{constr:equilibrium-loop}
  Let $\bm{\hat{h}}$ be the leftmost geodesic from $v^{*}$ to
  $\bar{v}^{*}$, and $\inv(\bm{\hat{h}})$ be the path obtained from
  $\bm{\hat{h}}$ by applying $\inv$ to each of its half-edges. Let $k$
  be the number of vertices in $\bm{\hat{h}}$ that are also in
  $\inv(\bm{\hat{h}})$. Since $\inv(v^{*}) = \bar{v}^{*}$, we have
  $k \geq 2$. Let $u_{1}= v^{*}, u_{2}, \ldots, u_{k} = \bar{v}^{*}$ be these
  vertices, in the order they are encountered by $\bm{\hat{h}}$. Note
  that for all $i \in [k]$, $\inv(u_{i})$ is also both in $\bm{\hat{h}}$ and
  $\inv(\bm{\hat{h}})$, and $\inv$ does not have fixed point on the
  set of vertices. Hence, $k$ is even.

  The loop we construct is
  \begin{equation*}
    \bm{\hat{g}_{\eq}} = \bm{\hat{h}}\vert_{u_{k/2} \to u_{k/2+1}} \sqcup \inv ( \bm{\hat{h}})\vert_{u_{k/2+1} \to u_{k/2}}.
  \end{equation*}
\end{constr}
\begin{figure}[ht]
  \centering
  \includegraphics[width=0.6\textwidth]{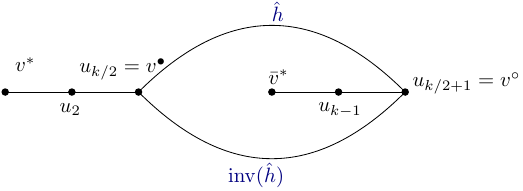}
  \caption{\label{fig:constr-eq-loop} Construction of the equilibrium loop.}
\end{figure}

\begin{definition}[Equilibrium loop]\label{def:equilibrium-loop}
  The loop $\bm{\hat{g}_{\eq}}$ of Construction \ref{constr:equilibrium-loop} is called the
  equilibrium loop.
\end{definition}

Let us give some properties of the loop $\bm{\hat{g}_{eq}}$.
\begin{definition}[Invariant loop]
  Let $\bm{g}$ be a simple loop of even length. We say it is invariant if for all
  $i \in [\# \bm{g}]$,
  \begin{equation*}
    \inv_{\hat{\map}} \left( \vertex(g_{2i-1}) \right) = \vertex(g_{\# \bm{g}+2i -1}).
  \end{equation*}
\end{definition}

\begin{lemma}\label{lem:h-inv-good-loop}
  The loop equilibrium loop $\bm{\hat{g}_{eq}}$ is a simple,
  invariant, good loop.
\end{lemma}
\begin{proof}
  The fact that $\bm{\hat{g}_{eq}}$ is simple follows by construction:
  both $\bm{\hat{h}}$ and $\inv(\bm{\hat{h}})$ are simple paths as
  they are geodesics. Furthermore, we chose the paths so that the only
  vertices that are both in
  $\bm{\hat{h}}\vert_{u_{k/2} \to u_{k/2+1}}$ and
  $\inv(\bm{\hat{h}})\vert_{u_{k/2+1} \to u_{k/2}}$ are
  their endpoints.

  The fact that $\bm{\hat{g}_{\eq}}$ is invariant also follows by
  construction: $\inv( \bm{\hat{h}})\vert_{u_{k/2+1} \to u_{k/2}}$ is
  obtained from $\bm{\hat{h}}\vert_{u_{k/2} \to u_{k/2+1}}$ by
  applying $\inv$ to each of its half-edges.

  Finally, the fact that $\bm{\hat{g}_{\eq}}$ is good follows from the fact that
  \begin{equation*}
    \hat{\ell}(u_{k/2}) = \hat{\ell} \circ \inv (u_{k/2}) = \hat{\ell}(u_{k/2 + 1}),
  \end{equation*}
  and from the fact that $\bm{\hat{h}}$ and
  $\inv( \bm{\hat{h}})$ are geodesics: the maximum of
  their labels may be attained only once, or twice at consecutive
  vertices.
\end{proof}
By the Jordan curve theorem, $\bm{\hat{g}_{\eq}}$ separates
$\hat{\map}$ into two embedded maps, whose boundary is
$\bm{\hat{g}_{\eq}}$. We denote by $\hat{\map}_{+}$ the embedded map
containing the face incident to the global root, and by
$\hat{\map}_{-}$ the other embedded map.
\begin{lemma}\label{lem:balanced}
  Let $v^{\bullet} = u_{k/2}$ and $v^{\circ} = u_{k/2 + 1}$ in
  Construction \ref{constr:equilibrium-loop}. Let
  $v_{1}, \ldots, v_{d}$ the neighbors of $v^{\circ}$ in
  $\hat{\map}_{+}$. We have for all $i \in [d]$ that
  \begin{equation*}
    \hat{\ell}(v_{i}) \geq \hat{\ell}(v^{\circ}).
  \end{equation*}
\end{lemma}
\begin{proof}
  Assume that there exists $v'$ in $\hat{\map}_{+}$, a neighbor of
  $v^{\circ}$, with $\hat{\ell}(v') = \hat{\ell}(v^{\circ}) - 1$. The only
  minima of vertex labels in $\hat{\map}$ are attained at $v^{*}$ and
  $\bar{v}^{*}$. Hence, there is a geodesic with strictly decreasing
  vertex labels, of length $\hat{\ell}(v')$, from $v'$ to one of $v^{*}$
  or $\bar{v}^{*}$. Since the labels are strictly decreasing, the
  geodesic may not cress the curve $\bm{\hat{g}_{\eq}}$, so the
  geodesic goes to $v^{*}$. Hence, we can construct a path of length
  $\ell(v^{\circ})$ to $v^{*}$. This contradicts the fact that
  $\bm{\hat{h}}$ is a geodesic.
\end{proof}

Let us now take $(\hat{\map}, \hat{\ell})$ to be the glued map,
constructed from $(\tilde{\map}, \tilde{\ell})$ and the leftmost good
geodesic $\bm{\tilde{g}}$ as defined in Section \ref{sec:leftm-good-geod}. The path
$\bm{\tilde{g}}$ corresponds to a good loop $\bm{\hat{g}}$ in
$\hat{\map}$. We may identify edges of $\tilde{\map}$ that are not on
$\bm{\tilde{g}}$, to edges in the embedded map $\hat{\map}_{+}$. In
particular,
\begin{equation*}
  \bm{\tilde{h}}\vert_{v^{*} \to v^{\bullet}}
\end{equation*}
can be seen as a path in $\hat{\map}_{+}$. The path $\bm{\hat{g}}$ can
be written $\bm{\hat{g}} = \bm{\hat{g}_{1}}\sqcup \bm{\hat{g}_{2}}$,
with $\bm{\hat{g}_{1}}$ starting at $v^{\bullet}$. We can thus see $\bm{\tilde{h}}$ as being embedded in $\hat{\map}_{+}$:
\begin{equation*}
  \bm{\tilde{h}} = \bm{\tilde{h}}\vert_{v^{*} \to v^{\bullet}} \sqcup \bm{\hat{g}_{1}} \quad \text{ in } \hat{\map}_{+}.
\end{equation*}

\begin{prop}\label{prop:g-tile-eq-loop}
  The loop $\bm{\hat{g}}$ is the equilibrium loop $\bm{\hat{g}_{\eq}}$
  of $\hat{\map}$.
\end{prop}
The proof relies on the following lemma concerning $\bm{\tilde{h}}$.
\begin{lemma}\label{lem:h-tilde-leftmost-geod}
  The path $\bm{\tilde{h}}$ is the leftmost geodesic in $\hat{\map}$
  between $v^{*}$ and $v^{\circ}$.
\end{lemma}
For the proofs of both Lemma \ref{lem:h-tilde-leftmost-geod} and Proposition \ref{prop:g-tile-eq-loop}, we use the following paths. Given a vertex $u$ in $\bm{\hat{g}}$, we set
\begin{equation*}
  \bm{\hat{g}_{u-}} = \begin{cases}
    \bm{\hat{g}_{1}}\vert_{v^{\bullet} \to u} &\text{ if $u$ is in } \bm{\hat{g}_{1}}\\
    \bm{\hat{g}_{2}^{-1}}\vert_{v^{\bullet} \to u} &\text{ if $u$ is in } \bm{\hat{g}_{2}},\\
  \end{cases}
  \quad \text{ and } \quad \bm{\hat{g}_{u+}} = \begin{cases}
    \bm{\hat{g}_{1}}\vert_{u \to v^{\circ}} &\text{ if $u$ is in } \bm{\hat{g}_{1}}\\
    \bm{\hat{g}_{2}^{-1}}\vert_{u \to v^{\circ}} &\text{ if $u$ is in } \bm{\hat{g}_{2}}.\\
  \end{cases}
\end{equation*}
Note that
$\# \bm{\hat{g}_{u'-}} + \# \bm{\hat{g}_{u'-}} = \# \bm{\hat{g}_{1}} = \# \bm{\hat{g}_{2}}$.

\begin{proof}[Proof of Lemma \ref{lem:h-tilde-leftmost-geod}]
  Notice first that $\bm{\tilde{h}}$ is the leftmost geodesic in
  $\hat{\map}_{+}$ between $v^{*}$ and $v^{\circ}$. Let $\bm{\tilde{h}'}$
  be the leftmost good geodesic in $\hat{\map}$ between $v^{*}$ and
  $v^{\circ}$. We assume that $\bm{\tilde{h}} \neq \bm{\tilde{h}'}$. We show
  that in that case, we can construct a path in $\hat{\map}_{+}$ that
  is either shorter or to the left of $\bm{\tilde{h}}$. This will
  contradict the fact that $\bm{\tilde{h}}$ is the leftmost good
  geodesic.

  Let $u$ be the last vertex
  such that
  \begin{equation*}
    \bm{\tilde{h}}\vert_{v^{*} \to u} = \bm{\tilde{h}'}\vert_{v^{*} \to u},
  \end{equation*}
  and $u'$ be the first vertex strictly after $u$ along
  $\bm{\tilde{h}'}$ that is in $\bm{\hat{g}}$. Since we assumed that
  $\bm{\tilde{h}} \neq \bm{\tilde{h}'}$, we have that $u \neq v^{\circ}$
  and $u'$ is well-defined.

  We consider the path $\bm{\tilde{h}'}\vert_{v^{*}\to u'}$. Since $\bm{\tilde{h}'}$ is a geodesic,
  \begin{equation}\label{eq:ineq-pf-lem-geod}
    \# \bm{\tilde{h}'}\vert_{v^{*}\to u'} \leq \# \left( \bm{\tilde{h}}\vert_{v^{*}\to v^{\bullet}} \sqcup \bm{\hat{g}_{u'-}} \right).
  \end{equation}
  There are two situations, depending on whether
  $\bm{\tilde{h}'}\vert_{u \to u'}$ is contained in $\hat{\map}_{+}$
  or $\hat{\map}_{-}$. Assume we are in the former case, as depicted in Figure \ref{fig:geodesic-plus}.
  \begin{figure}[htbp]
    \centering
    \includegraphics[width=0.7\textwidth]{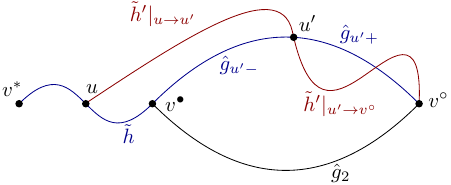}
    \caption{\label{fig:geodesic-plus} The case where $\bm{\tilde{h}'}\vert_{u \to u'}$ is in $\hat{\map}_{+}$.}
  \end{figure}

  If the inequality \eqref{eq:ineq-pf-lem-geod} is
  strict, then
  \begin{equation*}
    \bm{\tilde{h}'}\vert_{v^{*}\to u'}\sqcup \bm{\hat{g}_{u'+}}
  \end{equation*}
  is strictly shorter than $\bm{\tilde{h}}$, this contradict the fact
  that $\bm{\tilde{h}}$ is a geodesic in $\hat{\map}_{+}$. Thus, we assume there is
  equality in \eqref{eq:ineq-pf-lem-geod}. If
  $\bm{\tilde{h}'}\vert_{v^{*}\to u'}$ is at the left of
  $\bm{\tilde{h}}$, then
  \begin{equation*}
    \bm{\tilde{h}'}\vert_{v^{*} \to u'} \sqcup \bm{\hat{g}_{u'+}}
  \end{equation*}
  is a geodesic form $v^{*}$ to $v^{\circ}$, contained in
  $\hat{\map}_{+}$, at the left of $\bm{\tilde{h}}$. This is a
  contradiction. If $\bm{\tilde{h}'}\vert_{v^{*}\to u'}$ is at the
  right of $\bm{\tilde{h}}$, then $\bm{\tilde{h}'}$ is at the right of $\bm{\tilde{h}}$ and necessarily $\# \bm{\tilde{h}'}\vert_{u' \to v^{\circ}} < \# \bm{\hat{g}_{u'+}}$. We may thus define
  \begin{equation*}
    \bm{\tilde{h}^{(1)}} = \bm{\tilde{h}}\vert_{v^{*} \to v^{\bullet}} \sqcup \bm{\hat{g}_{u'-}} \sqcup \bm{\tilde{h}'}\vert_{u' \to v^{\circ}}.
  \end{equation*}
  It is a geodesic from $v^{*}$ to $v^{\circ}$, at the left of
  $\bm{\tilde{h}}'$. This contradicts the fact that $\bm{\tilde{h}'}$
  is the leftmost geodesic from $v^{*}$ to $v^{\circ}$. Thus,
  $\bm{\tilde{h}'}\vert_{u \to u'}$ cannot be in $\hat{\map}_{+}$.

  We now assume that $\bm{\tilde{h}'}\vert_{u \to u'}$ is in
  $\hat{\map}_{-}$. It implies in particular that $u$ is in
  $\bm{\hat{g}_{1}}$. This situation is depicted in Figures \ref{fig:geodesic-minus-less} and \ref{fig:geodesic-minus-eq}.

  \begin{figure}[htbp]
    \centering
    \begin{minipage}{0.48\textwidth}
        \centering
        \includegraphics[width=0.95\textwidth]{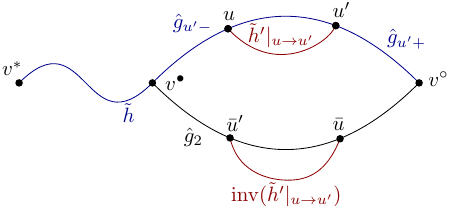} 
        \caption{\label{fig:geodesic-minus-less} The case where $\bm{\tilde{h}'}\vert_{u \to u'}$ is in $\hat{\map}_{-}$, and $u'$ is in $\bm{\hat{g}}_{1}$.}
    \end{minipage}\hfill
    \begin{minipage}{0.48\textwidth}
        \centering
        \includegraphics[width=0.95\textwidth]{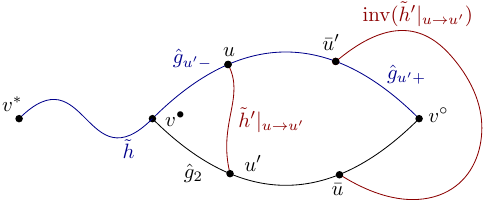} 
        \caption{\label{fig:geodesic-minus-eq} The case where
          $\bm{\tilde{h}'}\vert_{u \to u'}$ is in $\hat{\map}_{-}$,
          and $u'$ is in $\bm{\hat{g}}_{2}$.}
    \end{minipage}
\end{figure}

  The inequality \eqref{eq:ineq-pf-lem-geod} holds also in this case. If \eqref{eq:ineq-pf-lem-geod} is strict, then
  \begin{equation*}
    \bm{\tilde{h}^{(2)}} = \bm{\tilde{h}}\vert_{v^{*} \to v^{\bullet}} \sqcup \inv \left( \bm{\tilde{h}'}\vert_{v^{\bullet} \to u'}\sqcup \bm{\hat{g}_{u'+}} \right)
  \end{equation*}
  is a path from $v^{*}$ to $v^{\circ}$, contained in
  $\hat{\map}_{+}$, and strictly shorter than $\bm{\tilde{h}}$. This
  contradicts the fact that $\bm{\tilde{h}}$ is a geodesic in
  $\hat{\map}_{+}$. Hence, \eqref{eq:ineq-pf-lem-geod} is an equality. In that case, if $u'$
  is in $\bm{\hat{g}_{1}}$, we see that
  \begin{equation*}
    \# \bm{\tilde{h}}\vert_{v^{*} \to u'} = \# \bm{\tilde{h}'}\vert_{v^{*} \to u'},
  \end{equation*}
  and that $\bm{\tilde{h}}$ is at the left of $\bm{\tilde{h}'}$ at
  $u$. We can thus construct a path of the same length as
  $\bm{\tilde{h}'}$ that is to the left of $\bm{\tilde{h}'}$. This is
  a contradiction. If $u'$ is in $\bm{\hat{g}_{2}}$, then
  $\bm{\tilde{h}^{(2)}}$ is at the left of $\bm{\tilde{h}}$, is of the
  same length, and is contained in $\hat{\map}_{+}$. This contradicts
  the fact that $\bm{\tilde{h}}$ is the leftmost geodesic from $v^{*}$
  to $v^{\circ}$.

  We thus necessarily have that $\bm{\tilde{h}} = \bm{\tilde{h}'}$.
\end{proof}
\begin{proof}[Proof of Proposition \ref{prop:g-tile-eq-loop}]
  Lemma \ref{lem:h-tilde-leftmost-geod} implies that $\bm{\tilde{h}}$
  is the leftmost geodesic from the root vertex $v^{*}$ to $v^{\circ}$.

  Consider now the leftmost geodesic $\bm{\hat{h}}$ from $v^{*}$ to
  $\bar{v}^{*} = \inv(v^{*})$. Let $u$ be the last vertex such that
  \begin{equation*}
    \bm{\tilde{h}}\vert_{v^{*} \to u} = \bm{\hat{h}}\vert_{v^{*} \to u}.
  \end{equation*}
  Assume that $u \neq v^{\circ}$ and let $u' \neq u$ be the last vertex along
  $\bm{\hat{h}}$ that is both in $\bm{\hat{g}}$ and
  $\bm{\hat{h}}$.

  \begin{figure}[ht] \centering
    \includegraphics[width=0.6\textwidth]{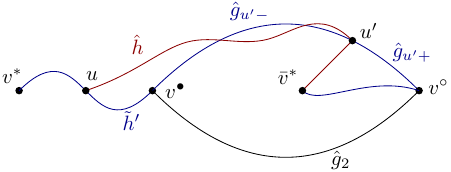}
    \caption{\label{fig:proof-g-geq} The different paths involved in the proof of Proposition \ref{prop:g-tile-eq-loop}.}
  \end{figure}

  We define the path
  \begin{equation*}
    \bm{\tilde{h}'} = \bm{\tilde{h}} \sqcup \inv(\bm{\tilde{h}}\vert_{v^{*} \to v^{\bullet}}).
  \end{equation*}
  This is a path from $v^{*}$ to $\bar{v}^{*}$.

  We claim that $\bm{\tilde{h}'}$ is a geodesic from $v^{*}$ to
  $\bar{v}^{*}$. Assuming the claim for now, we observe that
  $\bm{\hat{h}}$ is at the left of $\bm{\tilde{h}'}$. Thus, the path
  \begin{equation*}
    \bm{\hat{h}'} = \bm{\hat{h}}\vert_{v^{*} \to u'} \sqcup \bm{\hat{g}_{u'+}}
  \end{equation*}
  is at the left of $\bm{\tilde{h}'}$ as well. However,
  \begin{equation*}
    \# \bm{\hat{h}'} = \# \bm{\hat{h}}\vert_{v^{*} \to u'} + \# \bm{\hat{g}_{u'+}} \leq \# \left( \bm{\tilde{h}}\vert_{v^{*} \to v^{\bullet}} \sqcup \bm{\hat{g}_{u'-}}\right) + \# \bm{\hat{g}_{u'+}} = \# \bm{\tilde{h}}.
  \end{equation*}
  Hence, $\bm{\hat{h}'}$ is a geodesic from $v^{*}$ to $v^{\circ}$ on
  the left of $\bm{\tilde{h}}$. As $\bm{\tilde{h}}$ is the leftmost
  geodesic from $v^{*}$ to $v^{\circ}$, we have
  $\bm{\tilde{h}} = \bm{\hat{h}'}$. This implies that
  $\bm{\tilde{h}}\vert_{v^{*} \to u'} = \bm{\hat{h}}\vert_{v^{*} \to u'}$.
  This can only be satisfied if $u = u' = v^{\circ}$, i.e.\ if
  $\bm{\tilde{h}} = \bm{\hat{h}}\vert_{v^{*} \to v^{\circ}}$. This
  immediately implies that $v^{\bullet}$ is part of $\bm{\hat{h}}$ and
  \begin{equation*}
    \bm{\tilde{h}}\vert_{v^{\bullet} \to v^{\circ}} = \bm{\hat{h}}\vert_{v^{\bullet} \to v^{\circ}}.
  \end{equation*}
  Hence, we deduce that $\bm{\hat{g}} = \bm{\hat{g}_{\eq}}$.

  We now prove the claim. We proceed in two parts: we first prove that
  \begin{equation}\label{eq:proof-geq-1}
    \# \bm{\hat{h}}\vert_{v^{*} \to u'} = \# \left( \bm{\tilde{h}}\vert_{v^{*} \to v^{\bullet}} \sqcup \bm{\hat{g}_{u'-}}\right),
  \end{equation}
  and then that
  \begin{equation}\label{eq:proof-geq-2}
    \# \bm{\hat{h}}\vert_{u' \to \bar{v}^{*}} = \# \left( \bm{\hat{g}_{u'+}}\sqcup \bm{\tilde{h}'}\vert_{v^{\circ}\to \bar{v}^{*}}\right).
  \end{equation}
  This will imply the claim:
  \begin{equation*}
    \# \bm{\hat{h}}
    = \# \bm{\hat{h}}\vert_{v^{*} \to u'} + \# \bm{\hat{h}}\vert_{u' \to \bar{v}^{*}}
    = \# \bm{\tilde{h}}\vert_{v^{*} \to v^{\bullet}} + \underbrace{\# \bm{\hat{g}_{u'-}} + \# \bm{\hat{g}_{u'+}}}_{= \# \bm{\hat{g}_{1}}} + \#\bm{\tilde{h}'}\vert_{v^{\circ}\to \bar{v}^{*}}
    = \# \bm{\tilde{h}'}.
  \end{equation*}

  To prove \eqref{eq:proof-geq-1}, we notice that since $\bm{\hat{h}}$
  is a geodesic,
  \begin{equation*}
    \# \bm{\hat{h}}\vert_{v^{*} \to u'} \leq \# \left( \bm{\tilde{h}}\vert_{v^{*} \to v^{\bullet}} \sqcup \bm{\hat{g}_{u'-}}\right).
  \end{equation*}
  However, if the inequality were strict, we could construct
  \begin{equation*}
    \bm{\hat{h}^{(1)}} = \bm{\hat{h}}\vert_{v^{*} \to u'} \sqcup \bm{\hat{g}_{u'+}},
  \end{equation*}
  a path from $\bm{v^{*}}$ to $v^{\circ}$ that is strictly shorter than
  $\bm{\tilde{h}}$. This would contradict the fact that
  $\bm{\tilde{h}}$ is a geodesic.

  To prove \eqref{eq:proof-geq-2}, we proceed similarly: $\bm{\hat{h}}$
  being a geodesic implies
  \begin{equation*}
    \# \bm{\hat{h}}\vert_{u' \to \bar{v}^{*}} \leq \# \left( \bm{\hat{g}_{u'+}} \sqcup \bm{\hat{h}}\vert_{v^{\bullet}\to\bar{v}^{*}}\right).
  \end{equation*}
  If the inequality were strict, the path
  \begin{equation*}
    \bm{\hat{h}^{(2)}} = \inv \left( \bm{\hat{g}_{u'-}} \sqcup \bm{\hat{h}}\vert_{u'\to\bar{v}^{*}} \right)
  \end{equation*}
  would be a path from $v^{*}$ to $v^{\circ}$ that is strictly shorter than $\bm{\tilde{h}}$.
  This conclude the proof of the claim.
\end{proof}

\subsection{The mapping for maps on $\RP^2$}
\label{sec:mapping-maps-rp2}

We can now give the full construction.
\begin{constr}\label{constr:suitably-to-no}
  Consider a suitably labelled map $(\tilde{\map}, \tilde{\ell})$ with
  two local minima. Choose a path $\bm{\tilde{g}}$ as in Section
  \ref{sec:leftm-good-geod}. Construct the glued map
  $(\hat{\map}, \hat{\ell})$ as in Section \ref{sec:cutt-gluing-suit}.
  This map comes equipped with an orientation-reversing matching
  $\inv$. Using the construction of Section
  \ref{sec:maps-orient-cover}, we see $\hat{\map}$ as a map on the
  orientation covering of a map $\map$ on the projective plane
  $\RP^{2}$. This map is naturally pointed: the pointed vertex is the
  common image of the two vertices labelled by $0$.
\end{constr}

We now give the inverse construction. The idea is as follows: we can
choose canonically a loop -- the equilibrium loop -- in the orientable
double covering of $\map$. We can then contract the loop to obtain a
suitably labelled map with two local minima.

\begin{constr}[Lifting the map]\label{constr:lift-map-eq}
  Let $\map$ be a non-orientable map on the projective plane $\RP^{2}$
  which is flag-labelled by $\lambda\colon\Flag_{\map} \to [n]$ and
  with vertex profile $\theta\bar{\theta}$. We assume that $\map$ is
  pointed, i.e.\ that there is a distinguished vertex $v$ in $\map$.
  We label the vertices of $\map$ by their geodesic distance to $v$,
  giving a suitable labelling $\ell$ of $\map$. Using the bijection of
  Proposition \ref{prop:covering-no}, we construct a half-edge
  labelled map $\hat{\map}$, equipped with an orientation-reversing
  matching $\inv = \rho_{\map}$. This is the orientation covering map
  of $\map$. The labelling of the vertices of $\map$ induces a
  labelling of the vertices of $\hat{\map}$: for every vertex
  $\hat{v}$ of $\hat{\map}$, there is a unique image vertex $v$ (by
  the projection from the orientation covering) in $\map$. We set
  \begin{equation*}
    \hat{\ell}(\hat{v}) = \ell(v).
  \end{equation*}
  The map $(\hat{\map}, \hat{\ell})$ is a suitably labelled map: the
  minimum of the labels is 0, and the difference between the labels of
  two vertices connected by an edge is at most one since every edge in
  $\hat{\map}$ is in the preimage of an edge in $\map$. In
  $\hat{\map}$, we choose $\bm{\hat{g}_{\eq}}$ to be the unique
  equilibrium loop in $\hat{\map}$ (associated with $\inv$).
\end{constr}

Before constructing the map $\tilde{\map}$, we exchange some of the
labels in $\hat{\map}$.
\begin{constr}[Flipping the labels]\label{constr:label-flipping}
  By the Jordan curve theorem, the equilibrium loop
  $\bm{\hat{g}_{\eq}}$ separates $\hat{\map}$ into two embedded maps:
  $\hat{\map}_{+}$ containing the root face, and $\hat{\map}_{-}$ the
  other one. Each face in $\hat{\map}$ corresponds to a cycle of $\theta$
  or $\bar{\theta}$. All the labels of a face are either in $[n]$ in the
  first case or in $[\bar{n}]$ in the second case. For each face $f$
  in $\hat{\map}_{+}$, if the labels of the half-edges incident to $f$
  (i.e.\ such that $f$ is at their left) are in $[\bar{n}]$, exchange
  their label with the one obtained by applying $\inv$. Similarly,
  change the labels of the half-edges in a face of $\hat{\map}_{-}$
  with the one obtained by applying $\inv$ if they are in $[n]$. This
  mapping that flips the labels is $2^{c(\theta) - 1}$-to-$1$, as each pair
  of faces in $\hat{\map}$, except the one of the root face, may be
  flipped.
\end{constr}

Once the labels are exchanged we can glue $\bm{\hat{g}_{\eq}}$ to
itself to obtain a suitably labelled map with two local minima.
\begin{constr}[Closing the slit]\label{constr:close-slit}
  Consider now the map $\hat{\map}_{+}$ with boundary $\bm{\hat{g}}$,
  embedded in $\hat{\map}$. We glue the boundary to itself to remove
  the boundary face. The good loop $\bm{\hat{g}}$ can be written as
  the concatenation of two good paths
  \begin{equation*}
    \bm{\hat{g}} = \bm{\hat{g}_{1}} \sqcup \bm{\hat{g}_{2}}.
  \end{equation*}
  Let $v^{\bullet}$ and $v^{\circ}$ be the first and last vertex of
  $\bm{\hat{g}_{1}}$, they are the two minima of these good paths. We
  glue the two paths together, identifying the vertices as follows:
  for each $i = 1, \ldots, \# \bm{\hat{g}_{1}}-1$ there are exactly two vertices at
  distance $i$ to $v^{\bullet}$, we identify them. Denote by
  $\tilde{\map}$ the resulting map.
\end{constr}
Note that thanks to Construction \ref{constr:label-flipping} the face permutations of
$\tilde{\map}$ is $\theta$.

\begin{lemma}\label{lem:minima}
  The resulting map $\tilde{\map}$ has exactly two local minima.
\end{lemma}
\begin{proof}
  In $\hat{\map}^{+}$ there may be local minima only at the root, or
  on the boundary of $\hat{\map}^{+}$. The two possible vertices where
  there might be minima are $v^{\bullet}$ and $v^{\circ}$. However, by
  construction, one of them, say $v^{\bullet}$, is connected to the
  root and may not be a local minimum. Lemma \ref{lem:balanced},
  however, implies that $v^{\circ}$ is a local minimum in
  $\tilde{\map}$, as the vertices that may be of lower label that it
  get removed in Construction \ref{constr:close-slit}.
\end{proof}

We can now state the main theorem of this section.
\begin{theorem}\label{thm:bijection-1/2}
  The previous construction made of the tree steps given in
  Constructions \ref{constr:lift-map-eq}, \ref{constr:label-flipping},
  and \ref{constr:close-slit}, gives a $2^{\#\theta - 1}$-to-$1$
  mapping between the set of pointed labelled maps on the projective
  plane with face profile given by $\theta\bar{\theta}$, and the set of suitably
  labelled maps with two local minima and face profile $\theta$.
\end{theorem}
\begin{proof}
  Denote by $\Phi$ this mapping. Lemma \ref{lem:minima} implies that this construction
  gives a well-defined map to the set of suitably labelled map with
  two local minima. The fact that the face profile remains $\theta$ is a
  consequence of the construction. In particular, the faces are not
  modified during the cutting and gluing. The mapping $\Phi$ can be seen
  as the composition of a mapping $\Phi_{1}$, $2^{\#\theta - 1}$-to-$1$, that
  flips the labels, see Construction \ref{constr:label-flipping}, and a bijection $\Phi_{2}$ that
  consists in cutting the orientable double cover and gluing the sides
  appropriately. The fact that $\Phi_{2}$ is a bijection follows from
  Propositions \ref{prop:covering-no}, Proposition \ref{prop:bij-first-constr}, and the fact that given an
  orientable double cover, there is a unique equilibrium loop by
  Construction \ref{constr:equilibrium-loop}, along which we cut. This is the same curve we
  construct in the reverse construction by Proposition \ref{prop:g-tile-eq-loop}.
\end{proof}
Theorem \ref{thm:bijection-1/2} allows us to conclude the proof of
Corollary \ref{corol:leading-orders}.
\begin{proof}[Proof of Corollary \ref{corol:leading-orders}]
  We proved in Proposition \ref{prop:sub-S2} that:
  \begin{equation*}
    N^{l-2-n/2}\kappa_{l}(\bm{n}) = \left( \frac{2}{\beta} \right)^{l-1}\#\Maps_{0}(\theta(\bm{n})) + \left( \frac{2}{\beta} \right)^{l-1}\frac{1}{N}\left(\frac{2}{\beta} - 1\right)\frac{\#\Suitably_{2}(\theta(\bm{n}))}{n/2-l+1} + \order{\frac{1}{N^{2}}}.
  \end{equation*}
  Denote by $\Maps_{1/2}(\theta(\bm{n}))$ the set of edge-labelled
  maps on $\RP^{2}$ with face profile $\theta(\bm{n})$. Theorem
  \ref{thm:bijection-1/2} implies:
  \begin{equation*}
    \left( 1 + n/2 - l \right)\# \Maps_{1/2}(\theta(\bm{n})) = 2^{l-1}\#\Suitably_{2}(\theta(\bm{n})),
  \end{equation*}
  as $\left( 1 + n/2 - l \right)$ is the number of choice of a marked vertex in a map on $\RP^{2}$ with $l$ faces and $n/2$ edges. Hence,
  \begin{equation*}
    N^{l-2-n/2}\kappa_{l}(\bm{n}) = \left( \frac{2}{\beta} \right)^{l-1}\#\Maps_{0}(\theta(\bm{n})) + \left( \frac{2}{\beta} \right)^{l-1}\frac{1}{2^{l-1}N}\left(\frac{2}{\beta} - 1\right)\# \Maps_{1}(\theta(\bm{n})) + \order{\frac{1}{N^{2}}},
  \end{equation*}
  which is the wanted result.
\end{proof}

\appendix
\section{The limit $\beta \to \infty$ and the roots of Hermite polynomials}
\label{sec:limit-beta-inf}

Let us take the limit $\beta \to \infty$ in \eqref{eq:cumulant-distance-Bernoulli}. We get for all $n \geq 2$ even,
\begin{equation}\label{eq:cumulant-beta-inf}
  \kappa_{1}(n) = \sum_{q+r+s = n/2-l+1}\frac{(-1)^{q}B_{r}}{s+1}\binom{r+s}{r} N^{s+1} \left< e_{q} \right>_{\theta, l-1}.
\end{equation}
The terms of the expansion are linear combinations of expectations of
product of distances in planar maps with one face of degree $n$: with
Remark \ref{rem:prop-suitably} in mind, we see that since $\# V_{\map}^{\min} \geq 1$, we
have $l + \# V_{\map}^{\min} - (l-1) \geq 2$. Necessarily,
$\# V_{\map}^{\min} = 1$ and the genus must be zero.

If we take $l = 1$, the case of trees, we can push the computation
further. We have that
\begin{equation}\label{eq:tri-Hermite}
  T_{\infty}^{N} \coloneq \lim_{\beta \to \infty}T_{\beta}^{N} = \begin{pmatrix}
    0 & \sqrt{N - 1} & 0 & 0 & 0 & \dots \\
    \sqrt{N - 1} & 0 & \sqrt{N - 2} & 0 & 0 & \dots  \\
    0 & \sqrt{N - 2} & 0 & \ddots & \ddots & \dots  \\
    \vdots & \ddots & \ddots & \ddots & \sqrt{2} & 0 \\
    0 & \cdots & 0 &  \sqrt{2} & 0 & \sqrt{1}\\
    0 & \cdots & 0 & 0 & \sqrt{1} & 0\\
  \end{pmatrix}.
\end{equation}
The eigenvalues of $T_{\infty}^{N}$ are the roots of the Hermite
polynomials $\He_{N}$ defined by
\begin{equation}\label{eq:def-hermite}
  \He_{N}(x) \coloneq (-1)^{N}\ee^{x^{2}/2}\left( \frac{\dd}{\dd x} \right)^{N} \ee^{-x^{2}/2} \text{ for } N \geq 1.
\end{equation}
In particular, if we denote by
$h_{1,N} \leq h_{2,N} \leq \cdots \leq h_{N, N}$ the roots of
$\He_{N}$. We have that
\begin{equation}\label{eq:hermite-cumulant}
  p_{n}(h_{1, N}, \ldots, h_{N, N}) = \kappa_{1}(n),
\end{equation}
for all $N \geq 1$. This is easily seen from the fact that the
characteristic polynomial of $T_{\infty}^{N}$ satisfies
\begin{equation*}
  \begin{cases}
    &\det(z - T_{\infty}^{1}) = z\\
    &\det(z - T_{\infty}^{N + 1}) = z\det(z - T_{\infty}^{N}) - N\det(z - T_{\infty}^{N-1}),
  \end{cases}
\end{equation*}
the same induction equations as for $\left( \He_{N} \right)_{N \geq 1}$
(see for instance \cite[Section 3.2.2]{anderson_introduction_2010} for many properties of the
Hermite polynomials).

We give one application of \eqref{eq:hermite-cumulant}. The two leading orders of
$p_{n}(h_{1, N}, \ldots, h_{N, N})$ are given by
\cite{kornyik_wigner_2016}:
\begin{equation*}
  N^{-n/2 -1}p_{n}(h_{1, N}, \ldots, h_{N, N}) = \Cat_{n/2} - \frac{1}{N} \left( 2^{n-1} - \binom{n-1}{n/2} \right) + \order{\frac{1}{N^{2}}}.
\end{equation*}
We recover that the number of planar trees with $n/2$ edges is the
Catalan number $\Cat_{n/2}$, and obtain that
\begin{equation*}
  \#\Suitably_{2}(\cycle{1, 2, \ldots, n}) = \frac{n}{2}\left( 2^{n-1} - \binom{n-1}{n/2} \right),
\end{equation*}
or equivalently
\begin{equation*}
  \sum_{q+r = 1}\frac{(-1)^{q}B_{r}}{n/2}\binom{r+n/2-1}{r}\left< e_{q} \right>_{\theta, 0} = \frac{1}{2} \left< 1 \right>_{\theta, 0} - \frac{2}{n} \left< d \right>_{\theta, 0} =  - \left( 2^{n-1} - \binom{n-1}{n/2} \right),
\end{equation*}
that is:
\begin{equation*}
  \frac{\left< d \right>_{\theta, 0}}{\left< 1 \right>_{\theta, 0}} = \frac{2^{n-2}n}{\Cat_{n/2}(n/2 + 1)} - \frac{n^{2}}{8(n/2 + 1)} \sim \frac{1}{2}\sqrt{\frac{\pi}{8}} n^{3/2} \text{ as } n \to \infty.
\end{equation*}
The above quantity is the average distance between two (distinguished)
uniform vertices in a tree. This is related to the expectation of the
area under a Brownian excursion
$\E \mathcal{B}_{\text{ex}} = \sqrt{\pi/8}$, see \cite{janson_brownian_2007}.

\printbibliography

\end{document}